\documentclass{scrartcl}
\usepackage[utf8]{inputenc}
\usepackage[dvipsnames]{xcolor}
\usepackage[colorlinks=true, linkcolor=, citecolor=blue]{hyperref}
\usepackage{
  amsmath
  ,amsfonts
  ,amssymb
  ,amsthm
  ,mathtools                    %
  ,physics                      %
  ,microtype                    %
  ,geometry                     %
  ,enumitem                     %
  ,dsfont                       %
  ,wrapfig
  ,subcaption
}

\usepackage[many]{tcolorbox}
\usepackage[perpage]{footmisc}
\usepackage[toc,page]{appendix}
\usepackage[]{caption}
\usepackage[style=numeric-comp, backend=biber, hyperref, sortcites, sorting=none]{biblatex}
\AtEveryBibitem{%
  \clearfield{url}%
  \clearfield{primaryclass}%
  \clearfield{eprintclass}%
  \clearfield{doi}%
  \clearfield{issn}%
  \clearfield{isbn}%
  \clearfield{pages}%
}

\newlength\tindent
\renewcommand{\indent}{\hspace*{\tindent}}

\numberwithin{equation}{section}

\newtheoremstyle{mystyle}%
{}{}%
{}{}%
{\bfseries}{:}{0.5em}%
{%
{\thmname{#1}\,{\thmnumber{#2}}}\thmnote{ \normalfont(#3)}}%

\theoremstyle{mystyle}
\newtheorem{theorem}{Theorem}[section]
\newtheorem{definition}[theorem]{Definition}
\newtheorem{remark}[theorem]{Remark}
\newtheorem{lemma}[theorem]{Lemma}
\newtheorem{corollary}[theorem]{Corollary}

\newtheorem{proposition}[theorem]{Proposition}

\theoremstyle{definition}

\newcommand{\addsidelinetoenv}[1]{
    \tcolorboxenvironment{#1}{
    boxrule=0pt,
    boxsep=0pt,
    blanker,
    borderline west={2pt}{0pt}{Black!30!white},
    before skip=10pt,
    after skip=10pt,
    left=8pt,
    right=0em,
    breakable,
  }
}
\makeatletter
\@for\name:={theorem,definition,lemma,corollary,proposition,example}\do{%
  \expandafter\addsidelinetoenv\expandafter{\name}
}

\renewenvironment{proof}[1][\proofname]{%
   \par\pushQED{\qed}\normalfont%
   \topsep6\p@\@plus6\p@\relax
   \trivlist\item[\hskip\labelsep{\color{gray}\itshape #1\@addpunct{.}}]%
   \ignorespaces
}{%
   \popQED\endtrivlist\@endpefalse\smallskip
}
\makeatother

\let\temp\phi
\let\phi\varphi
\let\varphi\temp

\DeclarePairedDelimiterX{\inp}[2]{\langle}{\rangle}{#1, #2}

\newcommand{\RR}{\mathbb{R}}

\newcommand{\NN}{\mathbb{N}}

\newcommand{\ZZ}{\mathbb{Z}}
\newcommand{\PP}{\mathbb{P}}
\newcommand{\EE}{\mathbb{E}}
\newcommand{\HH}{\mathbb{H}}
\newcommand{\de}{\partial}

\newcommand{\mcl}[1]{\mathcal{#1}}
\newcommand{\mfk}[1]{\mathfrak{#1}}
\newcommand{\mrm}[1]{\mathrm{#1}}
\newcommand{\mbb}[1]{\mathbb{#1}}
\newcommand{\mbf}[1]{\mathbf{#1}}
\newcommand{\ie}{\emph{i.e.}\ }
\newcommand{\eg}{\emph{e.g.}\ }
\newcommand{\cf}{\emph{cf.}\ }
\newcommand{\clop}[1]{\left[#1 \right)}
\newcommand{\opcl}[1]{\left(#1 \right]}

\newcommand{\unit}{\mathds{1}}
\newcommand{\wo}{{\setminus}}

\newcommand{\eva}[1]{\langle #1 \rangle}

\usepackage{authblk}

\title{$\mathbb{H}^{2|2}$-model and Vertex-Reinforced Jump Process on Regular Trees: Infinite-Order Transition and an Intermediate Phase}

\author{Rémy Poudevigne--Auboiron}
\affil{University of Cambridge, Statistical Laboratory, DPMMS; {\upshape \texttt{rp698@cam.ac.uk}}}

\author{Peter Wildemann}
\affil{University of Cambridge, Statistical Laboratory, DPMMS; {\upshape\texttt{pw477@cam.ac.uk}}}

\begin{document}
\maketitle
\begin{abstract}\noindent
  \textbf{Abstract:}
  We explore the supercritical phase of the vertex-reinforced jump process (VRJP) and the $\HH^{2|2}$-model on rooted regular trees.
  The VRJP is a random walk, which is more likely to jump to vertices on which it has previously spent a lot of time.
  The $\HH^{2|2}$-model is a supersymmetric lattice spin model, originally introduced as a toy model for the Anderson transition.

  \indent
  On infinite rooted regular trees, the VRJP undergoes a recurrence/transience transition controlled by an inverse temperature parameter $\beta > 0$.
  Approaching the critical point from the transient regime, $\beta \searrow \beta_{\mrm{c}}$, we show that the expected total time spent at the starting vertex diverges as $\sim \exp(c/\sqrt{\beta - \beta_{\mrm{c}}})$.
  Moreover, on large \emph{finite} trees we show that the VRJP exhibits an additional intermediate regime for parameter values $\beta_{\mrm{c}} < \beta < \beta_{\mrm{c}}^{\mrm{erg}}$.
  In this regime, despite being transient in infinite volume, the VRJP on finite trees spends an unusually long time at the starting vertex with high probability.

  \indent
  We provide analogous results for correlation functions of the $\HH^{2|2}$-model.
  Our proofs rely on the application of branching random walk methods to a horospherical marginal of the $\HH^{2|2}$-model.  
\end{abstract}

\tableofcontents

\newpage
\section{Introduction and Main Results}

\subsection{History and Introduction}
Our work will focus on two distinct but related models:
The {$\HH^{2|2}$-model}, a lattice spin model which is related to the Anderson transition, and the \emph{vertex-reinforced jump process} (VRJP), a random walk on graphs which is more likely to jump to vertices on which it has already spent a lot of time.

\indent
The $\HH^{2|2}$-model was initially introduced by Zirnbauer \cite{zirnbauer_fourier_1991} as a toy model for studying the Anderson transition.
Formally, it is a lattice spin model taking values in the \emph{hyperbolic superplane} $\HH^{2|2}$, a supersymmetric analogue of hyperbolic space.
Independently, the VRJP was introduced by Davis and Volkov \cite{davis_continuous_2002} as a natural example of a reinforced (and consequently non-Markovian) continuous-time random walk.
Somewhat surprisingly, Sabot and Tarrès \cite{sabot_edge-reinforced_2015} observed that these two models are intimately related.
Namely, the time the VRJP asymptotically spends on vertices can be expressed in terms of the $\HH^{2|2}$-model.
This has been used to see the VRJP as a random walk in random environment, with the environment being given by the $\HH^{2|2}$-model.
Furthermore, the two models are linked by a Dynkin-type isomorphism theorem due to Bauerschmidt, Helmuth and Swan \cite{bauerschmidt_dynkin_2019, bauerschmidt_geometry_2020}, analogous to the connection between simple random walk and the Gaussian free field \cite{lyons_probability_2016}.

\indent
Both models are parametrised by an inverse temperature $\beta > 0$ and, depending on the background geometry of the graph under consideration, may exhibit a phase transition at some critical parameter $\beta_{\mrm{c}} \in \opcl{0,\infty}$.
For the $\HH^{2|2}$-model the expected transition is between a disordered high-temperature phase ($\beta < \beta_{\mrm{c}}$) and a symmetry-broken low-temperature phase ($\beta > \beta_{\mrm{c}}$) exhibiting long-range order.
For the VRJP the transition is between a recurrent phase due to strong reinforcement effects and a transient phase due to low reinforcement effects.

\indent
On $\ZZ^{D}$ a fair bit is known about the phase diagram of the two models.
In dimension $D\leq 2$ both models are never delocalised (\ie they are always disordered and recurrent, respectively)
\cite{davis_continuous_2002, merkl_recurrence_2009, sabot_edge-reinforced_2015, bauerschmidt_dynkin_2019, poudevigne-auboiron_monotonicity_2022, sabot_polynomial_2021}.
In dimensions $D\geq 3$, however, they exhibit a phase transition from a localised to a delocalised phase at a unique $\beta_{\mrm{c}} \in (0,\infty)$ \cite{disertori_quasi-diffusion_2010, disertori_anderson_2010, angel_localization_2012, disertori_transience_2015, sabot_edge-reinforced_2015,poudevigne-auboiron_monotonicity_2022, collevecchio_note_2021}.

\begin{wrapfigure}{r}{0.4\textwidth}
  \centering
  \includegraphics[scale=0.7]{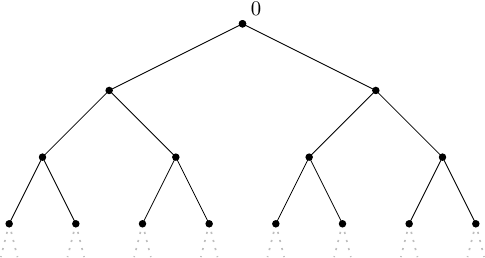}
  \caption{The rooted $(d+1)$-regular tree $\mbb{T}_{d}$ for ${d=2}$ shown up to its third generation, with the root vertex denoted as $0$.}
  \label{fig:binary-tree}
\end{wrapfigure}%
\indent
In this article we consider both models on the geometry of a rooted $(d+1)$-regular tree $\mbb{T}_{d}$ with $d\geq 2$ (see Figure~\ref{fig:binary-tree}).
For the VRJP this setting was previously explored by various authors \cite{davis_vertex-reinforced_2004, collevecchio_limit_2009, basdevant_continuous-time_2012, chen_speed_2018,rapenne_about_2023}.
In particular, Basdevant and Singh \cite{basdevant_continuous-time_2012} showed that the VRJP on Galton-Watson trees with mean offspring $m>1$  has a phase transition from recurrence to transience at some explicitly characterised $\beta_{\mrm{c}} \in (0,\infty)$.
For simplicity, we focus on the ``deterministic case'', but our results should translate to Galton-Watson trees as well (up to some technical restrictions on the offspring distribution).

\indent
The main goal of this work is to provide new information on the supercritical phase ($\beta > \beta_{\mrm{c}}$) including the near-critical regime.
Roughly speaking, we show that on the infinite rooted $(d+1)$-regular tree $\mbb{T}_{d}$ the order parameters of the VRJP and the $\HH^{2|2}$-model diverge as $\exp(c/\sqrt{\beta-\beta_{\mrm{c}}})$ as one approaches the critical point from the supercritical regime, $\beta \searrow \beta_{\mrm{c}}$ (see Theorem~\ref{thm:vrjp-near-crit} and \ref{thm:h22-asymptotics}, respectively).
Such behaviour has previously been predicted by Zirnbauer for Efetov's model \cite{zirnbauer_localization_1986}.
This ``infinite-order'' behaviour towards the critical point is rather surprising, as it conflicts with usual scaling hypotheses in statistical mechanics, which predict algebraic singularities as one approaches the critical points.
Moreover, we show that on \emph{finite} rooted $(d+1)$-regular trees, the VRJP and the $\HH^{2|2}$-model exhibit an additional \emph{mulifractal} intermediate regime for $\beta \in (\beta_{\mrm{c}}, \beta_{\mrm{c}}^{\mrm{erg}})$ (see Theorem~\ref{thm:int-phase-vrjp}, \ref{thm:vrjp-multifractality}, and \ref{thm:int-phase-h22}).
An illustration of some of our results for the VRJP is given in Figure~\ref{fig:phase-diagram}.

\begin{figure}
  \centering
  \includegraphics{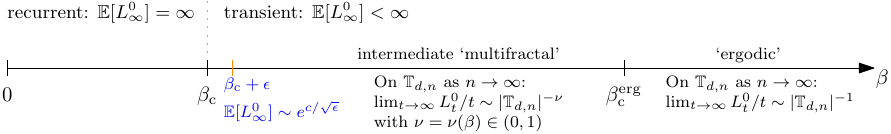}
  \caption{Sketch of the phase diagram for the VRJP on $\mbb{T}_{d}$ with $d\geq 2$. The recurrence/transience transition at $\beta_{\mrm{c}}$ is phrased in terms of $\EE[L^{0}_{\infty}]$, \ie the expected total time the walk (on the infinite rooted $(d+1)$-regular tree $\mbb{T}_{d}$) spends at the starting vertex. In this article, we obtain precise asymptotics for $\EE[L^{0}_{\infty}]$ as $\beta \searrow \beta_{\mrm{c}}$. Second, we show that there is an additional transition point $\beta_{\mrm{c}}^{\mrm{erg}} > \beta_{\mrm{c}}$. It is phrased in terms of the volume-scaling of the fraction of total time, $\lim_{t\to\infty} L^{0}_{t}/t$, the VRJP on the \emph{finite} tree $\mbb{T}_{d,n}$ spends at the origin. Here, the symbol ``$\sim$'' is understood loosely, and we refer to the text for precise error terms.}
  \label{fig:phase-diagram}
\end{figure}

\paragraph{Connection to the Anderson Transition and Efetov's Model.}
Inspiration for our work originates from predictions in the physics literature on \emph{Efetov's model} \cite{efetov_anderson_1987,zirnbauer_localization_1986,mirlin_localization_1991,de_luca_anderson_2014,tikhonov_fractality_2016,sonner_multifractality_2017}.
The latter is a supersymmetric lattice sigma model that is considered to capture the Anderson transition \cite{efetov_supersymmetry_1983, efetov1985anderson}.
To be more precise, Efetov's model can be derived from a \emph{granular limit} (similar to a Griffiths-Simon construction \cite{simon_phi_42_1973}) of the random band matrix model, followed by a sigma model approximation \cite{zirnbauer_supersymmetry_2004,rivasseau_quantum_2012}.
The connection to our work is due to Zirnbauer, who introduced the $\HH^{2|2}$-model as a simplification of Efetov's model \cite{zirnbauer_fourier_1991}.
Namely, in Efetov's model spins take value in the symmetric superspace $\mrm{U}(1,1|2)/[\mrm{U}(1|1)\otimes \mrm{U}(1|1)]$.
According to Zirnbauer, the essential features of this target space are its hyperbolic symmetry and its supersymmetry%
\footnote{Also referred to as ``perfect grading''. Roughly speaking, this refers to the fact that the space has the same number of bosonic and fermionic degrees of freedom (in this case four each), while these are also ``exchangeable'' under a symmetry of the space.}.
In this sense, $\HH^{2|2}$ is the simplest target space with these two properties.
Study of the $\HH^{2|2}$-model may guide the analysis of supersymmetric field theories more closely related to the Anderson transition.

\indent
Moreover, the $\HH^{2|2}$-model and the VRJP are directly and rigorously related to an Anderson-type model, which we refer to as the \emph{STZ-Anderson model} (see Definition~\ref{def:stz-anderson}).
This fact was already hinted at by Disertori, Spencer and Zirnbauer \cite{disertori_quasi-diffusion_2010}, but only fully appreciated by Sabot, Tarrès and Zeng \cite{sabot_vertex_2017,sabot_random_2018}, who exploited the relationship to gain new insights on the VRJP.
It is an interesting open problem to better understand the spectral properties of this model and how it relates to the VRJP and the $\HH^{2|2}$-model.

\indent
Notably, the phase diagram of the $\HH^{2|2}$-model is better understood than that of Efetov's model or the Anderson model on a lattice.
For example, for the $\HH^{2|2}$-model there is proven absence of long-range order in 2D \cite{bauerschmidt_dynkin_2019} as well as proven existence of a phase transition in 3D \cite{disertori_quasi-diffusion_2010, disertori_anderson_2010}.
For the Anderson model on $\ZZ^{D}$, the existence of a phase transition in $D\geq 3$ and the absence of one in $D=2$ are arguably among the most prominent open problems in mathematical physics.
A good example of the Anderson model's intricacies is given by the work of Aizenman and Warzel \cite{aizenman_resonant_2013-1,aizenman_random_2015-2}.
Despite many previous efforts, they were the first to gain a somewhat complete understanding of the model's spectral properties on the regular tree.
However, many questions are still open, in particular there are no rigorous results on the Anderson model's (near-)critical behaviour.
In this sense one might (somewhat generously) interpret this article as a step towards better understanding of the near-critical behaviour for a model in the ``Anderson universality class''.

\indent
We would also like to comment on the methods used in the physics literature on Efetov's model.
The analysis of the model on a regular tree, initiated by Efetov and Zirnbauer \cite{zirnbauer_localization_1986,efetov_anderson_1987}, relies on a recursion/consistency relation that is specific to the tree setting.
Using this approach, Zirnbauer predicted the divergence of the order parameter (relevant for the symmetry-breaking transition of Efetov's model) for $\beta \searrow \beta_{\mrm{c}}$.
We should mention that Mirlin and Gruzberg \cite{gruzberg_phase_1996} argued that this analysis should essentially carry through for the $\HH^{2|2}$-model.
In our case, we take a different path, exploiting a branching random walk structure in the ``horospherical marginal'' of the $\HH^{2|2}$-model (the $t$-field).

\indent
After completion of this work, we were made aware by Martin Zirnbauer of recent numerical investigations for the Anderson transition on random tree-like graphs \cite{garcia-mata_critical_2022,sierant_universality_2023}.
The observed scaling behaviour near the transition point might suggest the need for a field-theoretic description beyond the supersymmetric approach of Efetov (also see \cite{arenz_wegner_2023,zirnbauer_wegner_2023}).
At this point, there does not seem to exist a consensus on the theoretical description of near-critical scaling for the Anderson transition of tree-like graphs and rigorous results would be of great value.

\paragraph{Notation:}
In multi-line estimates, we occasionally use ``running constants'' $c,C > 0$ whose precise value may vary from line to line.
We denote by $[n] = {1,\ldots,n}$ the range of positive integers up to $n$.
For a graph $G = (V,E)$ an unoriented edge $\{x,y\} \in E$ will be denoted by the juxtaposition $xy$, whereas an oriented edge is denoted by a tuple $(x,y)$, which is oriented from $x$ to $y$. Write $\vec{E}$ for the set of oriented edges.
For a vertex $x$ in a rooted tree (or a particle of a branching random walk), we denote its \emph{generation} (i.e.\ distance from the origin) by $\abs{x}$.
We use the short-hand $\sum_{\abs{x} = n}\ldots$ to denote summation over all vertices/particles at generation $n$.
Variants of this convention will be used and the meaning should be clear from context.
When our results concern the $(d+1)$-regular rooted tree $\mbb{T}_{d}$, we assume $d\geq 2$ will typically suppress the $d$-dependence of all involved constants, unless specified otherwise.
Mentions of $\beta_{\mrm{c}}$ implicity refer to the critical parameter $\beta_{\mrm{c}} = \beta_{\mrm{c}}(d)$ as given by Proposition~\ref{prop:basdevant-singh}.

\paragraph{Acknowledgements:}
The authors would like to thank Roland Bauerschmidt for suggesting this line of research, for his valuable feedback and stimulating suggestions.
We would also like to thank Martin Zirnbauer for stimulating discussions on the current theoretical understanding of the Anderson transition.
Finally, we thank the reviewers for their thorough reading of the manuscript.
This work was supported by the European Research Council under the European Union's Horizon 2020 research and innovation programme (grant agreement No.~851682 SPINRG).

\subsection{Model Definitions and Results}
\begin{figure}[t]
  \centering
  \includegraphics[]{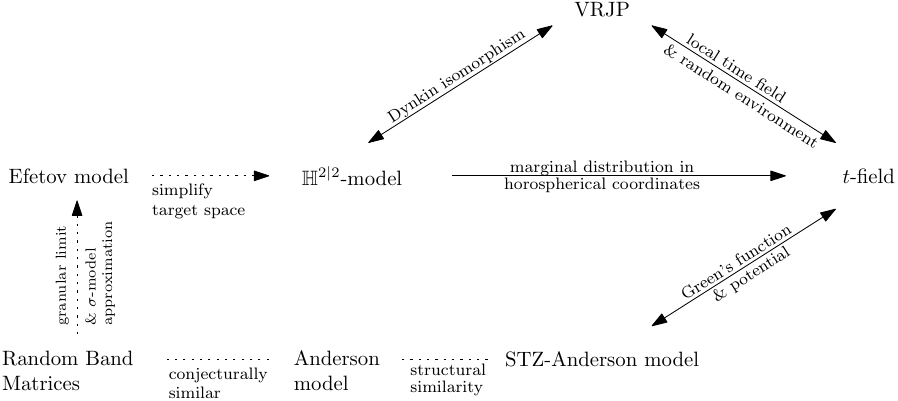}
  \caption{An illustration of various interconnected models, that we touch on. Solid lines denote rigorous connections, \ie relevant quantities in one model can be expressed in terms of the other. Dashed lines signify conceptual/heuristic connections.}
  \label{fig:model-relations}
\end{figure}

In this section, we define the VRJP, the $\HH^{2|2}$-model, the $t$-field and the STZ-Anderson model.
We are aware that spin systems with fermionic degrees of freedom, such as the $\HH^{2|2}$-model, might be foreign to some readers.
However, understanding this model is not necessary for the main results on the VRJP, and the reader can feel comfortable to skip references to the $\HH^{2|2}$-model on a first reading.
We also note that all models that we introduce are intimately related (as illustrated in Figure~\ref{fig:model-relations}) and Section~\ref{sec:model-background} will illuminate some of these connections.

\subsubsection{Vertex-Reinforced Jump Process.}

\begin{definition}
  Let $G = (V,E)$ be a locally finite graph equipped with positive edge-weights $(\beta_e)_{e\in E}$, and a starting vertex $i_0\in V$.
  The VRJP $(X_t)_{t\geq 0}$ starting at $X_{0} = i_0$ is the continuous-time jump process that at time $t$ jumps from a vertex $X_{t} = x$ to a neighbour $y$ at rate
  \begin{equation}
    \label{eq:138}
    \beta_{xy}[1+L_{t}^{y}] \qq{with}L_{t}^{y}(t):=\int_{0}^t 1_{X_s=y}\dd s.
  \end{equation}
  We refer to $L_{t}^{y}$ as the \emph{local time} at $y$ up to time $t$.
\end{definition}

Unless specified otherwise, the VRJP on a graph $G$ refers to the case of constants weights $\beta_{e} \equiv \beta$ and the dependency on the weight $\beta$ is specified by a subscript, as in $\EE_{\beta}$ or $\PP_{\beta}$.
By a slight abuse of language, we refer to $\beta$ as an \emph{inverse temperature}.

\paragraph{Results for the VRJP.}
Note that Figure~\ref{fig:phase-diagram} gives a rough picture of our statements for the VRJP.
In the following we provide the exact results.

\indent
In the following, $\beta_{\mrm{c}} = \beta_{\mrm{c}}(d)$ will denote the critical inverse temperature for the recurrence/transience transition of the VRJP on the infinite rooted $(d+1)$-regular tree $\mbb{T}_{d}$ with $d\geq 2$.
By Basdevant and Singh \cite{basdevant_continuous-time_2012} this inverse temperature is well-defined and finite: $\beta_{\mrm{c}} \in (0,\infty)$ (\cf Proposition~\ref{prop:basdevant-singh}).
Alternatively, $\beta_{\mrm{c}}$ is characterised in terms of divergence of the expected total local time at the origin: $\beta_{\mrm{c}} = \inf \{\beta > 0: \EE_{\beta}[L^{0}_{\infty}] < \infty\}$.
The following theorem provides information about the divergence of $\EE_{\beta}[L^{0}_{\infty}]$ as we approach the critical point from the transient regime.

\begin{theorem}[Local-Time Asymptotics as $\beta \searrow \beta_{\mrm{c}}$ for the VRJP on $\mbb{T}_{d}$]
  \label{thm:vrjp-near-crit}
  Consider the VRJP, started at the root $0$ of the infinite rooted $(d+1)$-regular tree $\mbb{T}_{d}$ with $d\geq 2$.
  Let $\beta_{\mrm{c}} = \beta_{\mrm{c}}(d) \in (0,\infty)$ be as in Proposition~\ref{prop:basdevant-singh}.
  Let $L^{0}_{\infty} = \lim_{t\to\infty}L^{0}_{t}$ denote the total time the VRJP spends at the root.
  There are constants $c,C>0$ such that for sufficiently small $\epsilon > 0$:
  \begin{equation}
    \label{eq:140}
    \exp(c/\sqrt{\epsilon})\leq \EE_{\beta_{\mrm{c}}+\epsilon}[L^{0}_{\infty}]\leq \exp(C/\sqrt{\epsilon}).
  \end{equation}
\end{theorem}

The above result concerned the \emph{infinite} rooted $(d+1)$-regular tree $\mbb{T}_{d}$.
On a finite rooted $(d+1)$-regular tree $\mbb{T}_{d,n}$ the total local time at the origin always diverges, but we may consider the fraction of time the walk spends at the starting vertex.
In terms of this quantity we can identify both the recurrence/transience transition point $\beta_{\mrm{c}}$ as well as an additional intermediate phase inside the transient regime.

\begin{theorem}[Intermediate Phase for VRJP on Finite Trees]\label{thm:int-phase-vrjp}
  Consider the VRJP started at the root of the rooted $(d+1)$-regular tree of depth $n$, $\mbb{T}_{d,n}$, with $d\geq 2$.
  Let $L_{t}^{0}$ denote the total time the walk spent at the root up until time $t$.
  We have
  \begin{equation}
    \label{eq:27}\textstyle
    \lim_{t\to\infty} \tfrac{L^{0}_{t}}{t}
    = \abs{\mbb{T}_{d,n}}^{-\nu(\beta) + o(1)}
    \qquad \text{w.h.p.\ as } n\to\infty
  \end{equation}
  with $\beta \mapsto \nu(\beta)$ continuous and non-decreasing such that
  \begin{equation}
    \label{eq:37}
    \nu(\beta)
    \begin{cases}
      = 0 &\text{ for } \beta\leq\beta_{\mrm{c}}\\
      \in (0,1) &\text{ for } \beta_{\mrm{c}} < \beta < \beta_{\mrm{c}}^{\mrm{erg}}\\
      =1 &\text{ for } \beta > \beta_{\mrm{c}}^{\mrm{erg}},
    \end{cases}
  \end{equation}
  for some $\beta_{\mrm{c}}^{\mrm{erg}} = \beta_{\mrm{c}}^{\mrm{erg}}(d) > \beta_{\mrm{c}}$.
  More precisely, we have
  \begin{equation}
    \label{eq:137}
    \nu(\beta) =  \max\!\Big(0, \inf_{\eta \in \opcl{0,1}} \frac{\psi_{\beta}(\eta)}{\eta \log d}\Big)
  \end{equation}
  with $\psi_{\beta}(\eta)$ given in \eqref{eq:162}.
\end{theorem}

Moreover, in the intermediate phase the inverse fraction of time at the origin shows a \emph{multifractal} scaling behaviour:

\begin{theorem}(Multifractality in the Intermediate Phase)\label{thm:vrjp-multifractality}
  Consider the setup of Theorem~\ref{thm:int-phase-vrjp} and suppose $\beta \in (\beta_{\mrm{c}}, \beta_{\mrm{c}}^{\mrm{erg}})$.
  For $\eta \in (0,1)$ we have
  \begin{equation}\textstyle
    \label{eq:81}
    \EE_{\beta}[(\lim_{t \to \infty}\tfrac{L_{t}^{0}}{t})^{-\eta}]
    \sim \abs{\mbb{T}_{d,n}}^{\tau_{\beta}(\eta) + o(1)} \qq{as} n\to \infty,
  \end{equation}
  where
  \begin{equation}
    \label{eq:82}
    \tau_{\beta}(\eta) =
    \begin{cases}
      \frac{\eta}{\eta_{\beta}} \frac{\psi_{\beta}(\eta_{\beta})}{\log d} &\text{ for } \eta \leq \eta_{\beta}\\
      \frac{\psi_{\beta}(\eta)}{\log d} &\text{ for } \eta \geq \eta_{\beta},
    \end{cases}
  \end{equation}
  where $\psi_{\beta}$ is given in \eqref{eq:162} and $\eta_{\beta} = \mrm{argmin}_{\eta >0} \psi_{\beta}(\eta)/\eta \in (0,1)$.
\end{theorem}

\subsubsection{The $\HH^{2|2}$-model.}\label{sec:h22-intro}

\paragraph{Definition of the $\HH^{2|2}$-Model.}
We start by writing down the formal expressions defining the $\HH^{2|2}$-model, and then make sense out of it afterwards.
Conceptually, we think of the \emph{hyperbolic superplane} $\HH^{2|2}$ as the set of vectors $\mbf{u} = (z,x,y,\xi,\eta)$, satisfying
\begin{equation}
  \label{eq:144}
  -1 = \mbf{u}\cdot \mbf{u} \coloneqq -z^{2} + x^{2} + y^{2} - 2\xi\eta.
\end{equation}
Here, $z,x,y$ are \emph{even/bosonic} coordinates and $\xi,\eta$ are \emph{odd/fermionic}, a notion that will be explained shortly.
For two vectors $\mbf{u}_{i} = (z_{i},x_{i},y_{i},\xi_{i},\eta_{i})$ and $\mbf{u}_{j} = (z_{j},x_{j},y_{j},\xi_{j},\eta_{j})$, we define the inner product
\begin{equation}
  \label{eq:149}
  \mbf{u}_{i} \cdot \mbf{u}_{j} \coloneqq -z_{i}z_{j} + x_{i}x_{j} + y_{i}y_{j} + \eta_{i}\xi_{j} - \xi_{i}\eta_{j}.
\end{equation}
In other words, this pairing is of hyperbolic type in the even variables and of symplectic type in the odd variables.

\indent
Consider a finite graph $G = (V,E)$ with non-negative edge weights $(\beta_{e})_{e \in E}$ and magnetic field $h > 0$.
Morally, we think of the $\HH^{2|2}$-model on $G$ as a probability measure on \emph{spin configurations} $\underline{\mbf{u}} = (\mbf{u}_{i})_{i\in V} \in (\HH^{2|2})^{V}$, such that the formal expectation of a functional $F \in C^{\infty}((\HH^{2|2})^{V})$ is given by

\begin{equation}
  \label{eq:148}
  \ev{F(\underline{\mbf{u}})}_{\beta, h}
  \coloneqq
  \int\limits_{(\HH^{2|2})^{V}} \prod_{i\in V}\dd{\mbf{u}_{i}} F(\underline{\mbf{u}}) \, e^{\sum_{ij \in E} \beta_{ij}(\mbf{u}_{i}\cdot \mbf{u}_{j} + 1) - h\sum_{i\in V} (z_{i} - 1)},
\end{equation}

with $\dd{\mbf{u}}$ denoting the Haar measure over $\HH^{2|2}$.
In other words, formally everything is analogous to the definition of spin/sigma models with ``usual'' target spaces, such as spheres $S^{n}$ or hyperbolic spaces $\HH^{n}$.
The only subtlety is that we still need to understand what a functional such as $F \in C^{\infty}((\HH^{2|2})^{V})$ means and how to interpret the integral above.

\bigskip
Rigorously, the space $\HH^{2|2}$ is not understood as a set of points, but rather is defined in a dual sense by directly specifying its set of smooth functions to be
\begin{equation}
  \label{eq:145}
  C^{\infty}(\HH^{2|2}) \coloneqq C^{\infty}(\RR^{2}) \otimes \Lambda(\RR^{2})
\end{equation}
In other words, this is the exterior algebra in two generators with coefficients in $C^{\infty}(\RR^{2})$ (which is the same as $C^{\infty}(\RR^{2|2})$, analogous to the fact that $\HH^{2} \cong \RR^{2}$ as smooth manifolds.).
Note that this set naturally carries the structure of a graded-commutative algebra.
More concretely, any \emph{superfunction} $f\in C^{\infty}(\HH^{2|2})$ can we written as
\begin{equation}
  \label{eq:142}
  f = f_{0}(x,y) + f_{\xi}(x,y) \xi + f_{\eta}(x,y) \eta + f_{\xi\eta}(x,y) \xi\eta
\end{equation}
with smooth functions $f_{0},f_{\xi},f_{\eta},f_{\xi\eta} \in C^{\infty}(\RR^{2})$ and $\xi,\eta$ generating a Grassmann algebra, \ie they satisfy the algebraic relations $\xi\eta = -\eta\xi$ and $\xi^{2} = \eta^{2} = 0$.
We think of such $f$ as a smooth function in the variables $x,y,\xi,\eta$ and write $f = f(x,y,\xi,\eta)$.
In particular, the \emph{coordinate functions} $x,y,\xi,\eta$ are themselves superfunctions.
In light of \eqref{eq:144}, we \emph{define} the $z$-coordinate to be the (even) superfunction
\begin{equation}
  \label{eq:143}
  z \coloneqq (1 + x^{2} + y^{2} - 2\xi\eta)^{1/2}
  \coloneqq (1 + x^{2} + y^{2})^{1/2} - \frac{\xi\eta}{(1+x^{2}+y^{2})^{1/2}} \in C^{\infty}(\HH^{2|2}).
\end{equation}
In this sense the coordinate vector $\mbf{u} = (z,x,y,\xi,\eta)$ satisfies $\mbf{u}\cdot \mbf{u} = -1$.
By abuse of notation we write $\mbf{u} \in \HH^{2|2}$, but more correctly one might say that $\mbf{u}$ \emph{parametrises} $\HH^{2|2}$.
For a superfunction $f \in C^{\infty}(\HH^{2|2})$ we write $f(\mbf{u}) = f(x,y,\xi,\eta) = f$ and in line with physics terminology we might say that $f$ is a function of the \emph{even/bosonic} variables $z,x,y$ and the \emph{odd/fermionic} variables $\xi,\eta$.

\indent
The definition of $z$ in \eqref{eq:143} shows a particular example of a more general principle:
The composition of an ordinary function (the square root in the example) with a superfunction (in the example that is $1 + x^{2} + y^{2} - 2\xi\eta$) is defined by formal Taylor expansion in the Grassmann variables.
Due to nilpotency of the Grassmann variables this is well-defined.

\indent
Next we would like to introduce a notion of integrating a superfunction $f(\mbf{u})$ over $\HH^{2|2}$.
Expressing $f$ as in \eqref{eq:142}, we define the derivations $\de_{\xi}, \de_{\eta}$ acting via
\begin{equation}
  \label{eq:146}
  \de_{\xi}f = f_{\xi}(x,y) + f_{\xi\eta}(x,y) \eta \qq{and} \de_{\eta}f = f_{\eta}(x,y) - f_{\xi\eta}(x,y) \xi.
\end{equation}
In particular, note that these derivations are \emph{odd}: they anticommute, $\de_{\xi}\de_{\eta} = -\de_{\eta}\de_{\xi}$, and satisfy a graded Leibniz rule.
The $\HH^{2|2}$-integral of $f \in C^{\infty}(\HH^{2|2})$ is then defined to be the linear functional
\begin{equation}
  \label{eq:147}
  \int_{\HH^{2|2}}\dd{\mbf{u}}f(\mbf{u}) \coloneqq \int_{\RR^{2}}\dd{x}\dd{y} \de_{\eta}\de_{\xi} [\frac{1}{z} f].
\end{equation}
The factor $\tfrac{1}{z}$ plays the role of a $\HH^{2|2}$-volume element in the coordinates $x,y,\xi,\eta$.
Note that this integral evaluates to a real number.

\indent
In a final step to formalise \eqref{eq:148} we define multivariate superfunctions over $\HH^{2|2}$
\begin{equation}
  \label{eq:150}
  C^{\infty}((\HH^{2|2})^{V}) \coloneqq \bigotimes\limits_{i\in V} C^{\infty}(\HH^{2|2}) \cong C^{\infty}(\RR^{2\abs{V}}) \otimes \Lambda(\RR^{2 \abs{V}}),
\end{equation}
that is the Grassmann algebra in $2\abs{V}$ generators $\{\xi_{i}, \eta_{i}\}_{i\in V}$ with coefficients in $C^{\infty}(\RR^{2\abs{V}})$.
An element of this algebra is considered a functional over spin configurations $\underline{\mbf{u}} = \{\mbf{u}_{i}\}_{i\in V}$ and we write $F = F(\underline{\mbf{u}})$.
Any superfunction $F \in C^{\infty}((\HH^{2|2})^{V})$ can be expressed, analogously to \eqref{eq:142}, as
\begin{equation}
  \label{eq:151}
  \sum_{I,J \subseteq V} f_{I,J}(\{x_{i},y_{i}\}_{i\in V}) \prod_{i\in I} \xi_{i} \prod_{j\in J} \eta_{j}.
\end{equation}
The integral of such $F$ over $(\HH^{2|2})^{V}$ is defined as
\begin{equation}
  \label{eq:152}\hspace*{-1em}
  \int\limits_{(\HH^{2|2})^{V}} \dd{\underline{\mbf{u}}} F(\underline{\mbf{u}}) \coloneqq
  \int\limits_{(\HH^{2|2})^{V}} \prod_{i\in V}\dd{\mbf{u}_{i}} F(\underline{\mbf{u}}) \coloneqq
  \int\limits_{\RR^{2\abs{V}}} \prod_{i\in V}\dd{x_{i}}\dd{y_{i}} \prod_{i\in V}\de_{\eta_{i}}\de_{\xi_{i}} [{(\textstyle\prod_{i\in V}\tfrac{1}{z_{i}})} F(\underline{\mbf{u}})].
\end{equation}
With this notion of integration, the definition of the $\HH^{2|2}$-model in \eqref{eq:148} can be understood in a rigorous sense: The ``Gibbs factor'' is the composition of a regular function (exponential) with a superfunction (the exponent). As such it is defined by expansion in the Grassmann variables.

\paragraph{Results for the $\HH^{2|2}$-Model.}
In the following we will simply rephrase above theorems in terms of the $\HH^{2|2}$-model.

\begin{theorem}[Asymptotics as $\beta \searrow \beta_{\mrm{c}}$ for the $\HH^{2|2}$-model on $\mbb{T}_{d}$]
  \label{thm:h22-asymptotics}
  Consider the $\HH^{2|2}$-model on $\mbb{T}_{d,n}$.
  Suppose $\beta_{\mrm{c}} = \beta_{\mrm{c}}(d) \in (0,\infty)$ is as in Proposition~\ref{prop:basdevant-singh}.
  The quantity
  \begin{equation}
    \label{eq:47}
    \eva{x_{0}^{2}}_{\beta_{\mrm{c}} + \epsilon}^{+} \coloneqq \lim_{h\searrow 0} \lim_{n\to\infty} \eva{x_{0}^{2}}_{\beta_{\mrm{c}} + \epsilon;h,\mbb{T}_{d,n}}
  \end{equation}
  is well-defined and finite for any $\epsilon > 0$.
  There exist constants $c,C > 0$ such that for sufficiently small $\epsilon > 0$
  \begin{equation}
    \label{eq:136}
    \exp(c/\sqrt{\epsilon}) \leq \eva{x_{0}^{2}}_{\beta_{\mrm{c}} + \epsilon}^{+} \leq \exp(C/\sqrt{\epsilon}).
  \end{equation}
\end{theorem}

The above statement considered the \emph{infinite-volume limit}, \ie taking $n\to\infty$ before removing the magnetic field $h \searrow 0$.
One may also consider a \emph{finite-volume limit} (also referred to as \emph{inverse-order thermodynamic limit} \cite{kravtsov_non-ergodic_2018}):
In that case, we consider scaling limits of observable as $h\searrow 0$ before taking $n\to \infty$.
In this limit, we also demonstrate an intermediate multifractal regime for the $\HH^{2|2}$-model.

\begin{theorem}[Intermediate Phase for the $\HH^{2|2}$-Model on $\mbb{T}_{d,n}$]
  \label{thm:int-phase-h22}
  There exist $0 < \beta_{\mrm{c}} < \beta_{\mrm{c}}^{\mrm{erg}} < \infty$ as in Theorem~\ref{thm:int-phase-vrjp}, such that for $\beta_{\mrm{c}} < \beta < \beta_{\mrm{c}}^{\mrm{erg}}$ we have for $\eta\in(0,1)$
  \begin{equation}\textstyle
    \label{eq:161}
    \lim_{h\searrow 0} h^{-\eta}\eva{z_{0} \abs{x_{0}}^{-\eta}}_{\beta,h;\mbb{T}_{d,n}} \sim \abs{\mbb{T}_{d,n}}^{\tau_{\beta}(\eta) + o(1)} \qq{as} n\to\infty
  \end{equation}
  with $\tau_{\beta}(\eta)$ as given in \eqref{eq:82}.
\end{theorem}

At first glance, the observable in \eqref{eq:161} might seem somewhat obscure.
However, in the physics literature on Efetov's model and the Anderson transition, analogous quantities are predicted to encode disorder-averaged (fractional) moments of eigenstates at a given vertex and energy level, see for example \cite[Equation (6)]{sonner_multifractality_2017}.
The volume-scaling of these quantities provides information about the (de)localisation behaviour of the eigenstates.

\subsubsection{The $t$-field.}
Despite the inconspicuous name, the $t$-field is the most relevant object for our analysis.
It is directly related to both the VRJP, encoding the time the VRJP asymptotically spends on each vertex, as well as the $\HH^{2|2}$-model, arising as a marginal in horospherical coordinates (see Section~\ref{sec:model-background} for details).

\begin{definition}[$t$-field Distribution]\label{def:t-field}
  Consider a finite graph $G = (V,E)$, a vertex $i_{0} \in V$ and non-negative edge-weights $(\beta_e)_{e\in E}$.
  The law of the $t$-field, with weights $(\beta_e)_{e\in E}$, \emph{pinned} at $i_{0}$, is a probability measure on configurations $\mbf{t} = \{t_{i}\}_{i\in V} \in \RR^{V}$ given by
  \begin{equation}
    \label{eq:156}
    \mcl{Q}^{(i_{0})}_{\beta}(\dd{\mbf{t}})
    \coloneqq
    e^{-\sum_{ij \in E}\beta_{ij}[\cosh(t_{i} - t_{j}) - 1]} D_{\beta}(\mbf{t})^{1/2}\; \delta(t_{i_{0}})\prod_{i \in V \wo \{i_{0}\}} \frac{\dd{t}_{i}}{\sqrt{2\pi/\beta}},
  \end{equation}
  with the determinantal term
  \begin{equation}
    \label{eq:165}
    D_{\beta}(\mbf{t}) \coloneqq
    \sum_{T \in \vec{\mcl{T}}^{(i_{0})}} \prod_{(i,j) \in T} \beta_{ij}e^{t_{i} - t_{j}},
  \end{equation}
  where $\vec{\mcl{T}}^{(i_{0})}$ is the set of spanning trees in $G$ oriented away from $i_{0}$.
\end{definition}

Alternatively, one can write $D_{\beta}(\mbf{t}) = \prod_{i\in V\setminus\{i_{0}\}} e^{-2t_{i}} \det_{i_{0}}(-\Delta_{\beta(\mbf{t})})$, where $\det_{i_{0}}$ denotes the principal minor with respect to $i_{0}$ and $-\Delta_{\beta(\mbf{t})}$ is the discrete Laplacian for edge-weights $\beta(\mbf{t}) = (\beta_{ij}e^{t_{i}+t_{j}})_{ij}$.

\indent
In general the determinantal term renders the law $\mcl{Q}_{\beta}^{(i_{0})}$ highly non-local.
However, in case the underlying graph $G$ is a tree, only a single summand contributes to~\eqref{eq:165} and the measure factorises in terms of the oriented edge-increments $\{t_{i}-t_{j}\}_{(i,j)}$.
This simplification is essential for this article and gives us the possibility to analyse the $t$-field on rooted $(d+1)$-regular trees in terms of a branching random walk.

\subsubsection{STZ-Anderson Model.}
The following introduces a random Schrödinger operator, which is related to the previously introduced models.
It will only be required for translating our results on the intermediate phase to the $\HH^{2|2}$-model (Section~\ref{sec:h22-intermediate}), so the reader may skip this definition on a first reading.
As Sabot, Tarrès and Zeng \cite{sabot_vertex_2017, sabot_random_2018} were the first to study this system in detail, we refer to it as the \emph{STZ-Anderson model}.

\begin{definition}[STZ-Anderson Model]\label{def:stz-anderson}
  Consider a locally finite graph $G = (V,E)$, equipped with non-negative edge-weights $(\beta_e)_{e\in E}$.
  For $B = (B_{i})_{i\in\Lambda} \subseteq \RR_{+}^{\Lambda}$ define the Schrödinger-type operator
  \begin{equation}
    \label{eq:167}
    H_{B} \coloneqq -\Delta_{\beta} + V(B) \qq{with} [V(B)]_{i} = B_{i} - {\textstyle\sum_{j}\beta_{ij}}.
  \end{equation}
  Define a probability distribution $\nu_{\beta}$ over configurations $B = (B_{i})_{i\in\Lambda}$ by specifying the Laplace transforms of its finite-dimensional marginals:
  For any vector $(\lambda_{i})_{i\in V} \in \clop{0,\infty}^{V}$ with only finitely many non-zero entries, we have
  \begin{equation}
    \label{eq:209}
    \int e^{-(\lambda,B)}\nu_{\beta}(\dd{B})
    = \frac1{\prod_{i\in V}\sqrt{1+2\lambda_{i}}} \exp[-\sum_{ij \in E}\beta_{ij}(\sqrt{1+2\lambda_{i}} \sqrt{1+2\lambda_{j}} - 1)].
  \end{equation}
  Subject to this distribution, we refer to $B$ as the \emph{STZ-field} and to $H_{B}$ as the STZ-Anderson model.  
\end{definition}

One may note that on finite graphs, the density of $\nu_{\beta}$ is explicit:
\begin{equation}
  \label{eq:168}
  \nu_{\beta}(\dd{B}) \propto \frac{e^{-\tfrac12 \sum_{i}B_{i}}}{\sqrt{\det(H_{B})}} \unit_{H_{B} > 0} \dd{B},
\end{equation}
where $H_{B} > 0$ means that the matrix $H_{B}$ is positive definite.
The definitino via \eqref{eq:209} is convenient, since it allows us to directly consider the infinite-volume limit.
We also note that while the density \eqref{eq:168} seems highly non-local, the Laplace transform in \eqref{eq:209} only involves values of $\lambda$ at adjacent vertices and therefore implies $1$-dependency of the STZ-field.

\indent
In the original literature the STZ-field is denoted by $\beta$ and referred to as the $\beta$-field. In order to be consistent with the statistical physics literature and avoid confusion with the inverse temperature, we introduced this slightly different notation.
To be precise, we used this change of notation to also introduce a slightly more convenient normalisation: one has $B_{i} = 2\beta_{i}$ compared to the normalisation of the $\beta$-field $\{\beta_{i}\}$ used by Sabot, Tarrès and Zeng.

\subsection{Further Comments}\label{sec:further-comments}

\paragraph{Comments on Related Work.}
As noted earlier, the VRJP on tree geometries was already studied by various authors \cite{davis_vertex-reinforced_2004, collevecchio_limit_2009, basdevant_continuous-time_2012, chen_speed_2018, rapenne_about_2023}.
One notable difference to our work is that we do not consider the more general setting of Galton-Watson trees.
While this is mostly to avoid unnecessary notational and technical difficulties, the Galton-Watson setting might be more subtle.
This is due to an ``extra'' phase transition in the transient phase, observed by Chen and Zeng \cite{chen_speed_2018}.
This phase transition depends on the probability of the Galton Watson tree having precisely one offspring.
It is an interesting question how this would interact with our analysis.

\indent
In regard to our results, the recent work by Rapenne \cite{rapenne_about_2023} is of particular interest.
He provides precise quantitative information on the (sub-)critical phase $\beta \leq \beta_{\mrm{c}}$.
The results are phrased in terms of a certain martingale, associated with the STZ-Anderson model, but they can be formulated in terms of the $\HH^{2|2}$-model with wired boundary conditions (or analogously the VRJP started from the boundary) on a rooted $(d+1)$-regular tree of finite depth.
In this sense, Rapenne's article can be considered as complementary to our work.

\indent
Another curious connection to our work is given by the Derrida-Retaux model \cite{collet_study_1984, derrida_depinning_2014, hu_free_2018, chen_max-type_2019, chen_critical_2020, hu_exactly_2020-1, derrida_results_2020, chen_derridaretaux_2021-1}.
The latter is a toy model for a hierarchical renormalisation procedure related to the depinning transition.
It has recently been shown \cite{chen_derridaretaux_2021-1} that the free energy of this model may diverge as $\sim \exp(-c/\sqrt{p - p_{\mrm{c}}})$ approaching the critical point from the supercritical phase, $p\searrow p_{\mrm{c}}$.
There are further formal similarities between their analysis and the present article.
It would be of interest to shed further light on the universality of this type of behaviour.

\paragraph{Debate on Intermediate Phase}

We would like to highlight that the presence/absence of such an intermediate phase for the Anderson transition%
\footnote{This may refer to the Anderson model, Efetov's model, or certain sparse random matrix models (such as random band matrices), all of which are largely considered equivalent in the theoretical physics community.}
on tree-geometries has been a recent topic of debate in the physics literature (see \cite{biroli_delocalization_2018,kravtsov_non-ergodic_2018} and references therein).
In short, the debate concerns the question of whether the intermediate phase only arises due to finite-volume and boundary effects on the tree.

\indent
While the presence of a \emph{non-ergodic delocalised} phase on finite regular trees has been established in recent years \cite{monthus_anderson_2011, tikhonov_fractality_2016,sonner_multifractality_2017}, it was not clear if this behaviour persists in the absence of a large ``free'' boundary.
To study this, one can consider a system on a large random regular graphs (RRGs) as a ``tree without boundary'' (alternatively one could consider trees with wired boundary conditions).
For the Anderson transition on RRGs, early numerical simulations \cite{biroli_difference_2012, de_luca_anderson_2014, altshuler_nonergodic_2016} suggested existence of an intermediate phase, in conflict with existing theoretical predictions \cite{fyodorov_localization_1991,mirlin_localization_1991, mirlin_universality_1991,fyodorov_novel_1992}.
Shortly afterwards, it was argued that the discrepancy was due to finite-size effects that vanish at very large system sizes \cite{tikhonov_anderson_2016,tikhonov_fractality_2016, biroli_delocalization_2018}, even though this does not seem to be the consensus%
\footnote{To our understanding, the cited sources consider an \emph{inverse-order thermodynamic limit}, in which they remove the level-broadening (resp.\ magnetic field) before taking the system size to infinity. This corresponds to a \emph{finite-volume limit}, as opposed to the reversed limit order considered in other treatments of the Anderson transition. In this sense, the different statements are not directly comparable.}
\cite{kravtsov_non-ergodic_2018, altshuler_nonergodic_2016}.

\indent
We should note that Aizenman and Warzel \cite{aizenman_resonant_2013-1,aizenman_extended_2011} have shown the existence of an energy-regime of ``resonant delocalisation'' for the Anderson model on regular trees.
It would be interesting to understand if/how this phenomenon is related to the intermediate phase discussed here.

\indent
In accordance with the physics literature, we refer to the intermediate phase ($\beta_{\mrm{c}} < \beta < \beta_{\mrm{c}}^{\mrm{erg}}$) as \emph{multifractal} as opposed to the \emph{ergodic} phase ($\beta > \beta_{\mrm{c}}^{\mrm{erg}}$).

\subsection{Structure of this Article}
In \textbf{Section~\ref{sec:model-background}} we provide details on the connections between the various models and recall previously known results for the VRJP.
In particular, we recall that the VRJP can be seen as a random walk in random conductances given in terms of a $t$-field (referred to as the \emph{$t$-field environment}).
On the tree, the $t$-field can be seen as a branching random walk (BRW) and we recall various facts from the BRW literature.
In \textbf{Section~\ref{sec:near-crit-t-field}} we apply BRW techniques to establish a statement on effective conductances in random environments given in terms of \emph{critical} BRWs (Theorem~\ref{thm:critical-eff-cond}).
With Theorem~\ref{thm:effective-conductance-nearcrit} we prove a result on effective conductances in the \emph{near-critical} $t$-field environment.
We close the section by showing how the result on effective conductances implies Theorem~\ref{thm:vrjp-near-crit} on expected local times for the VRJP.
In \textbf{Section~\ref{sec:intermediate-phase}} we continue to use BRW techniques for the $t$-field to establish Theorem~\ref{thm:int-phase-vrjp} on the intermediate phase for the VRJP.
We also prove Theorem~\ref{thm:vrjp-multifractality} on the multifractality in the intermediate phase. Moreover, we argue that Rapenne's recent work \cite{rapenne_about_2023} implies the absence of such an intermediate phase on trees with wired boundary conditions.
In \textbf{Section~\ref{sec:h22-results}} we show how to establish the results for the $\HH^{2|2}$-model.
For the near-critical asymptotics (Theorem~\ref{thm:h22-asymptotics}) this is an easy consequence of a Dynkin isomorphism between the $\HH^{2|2}$-model and the VRJP.
For Theorem~\ref{thm:int-phase-h22} on the intermediate phase, we make use of the STZ-field to connect the observable for the $\HH^{2|2}$-model with the observable $\lim_{t\to\infty} L^{0}_{t}/t$ that we study for the VRJP.

\newpage
\section{Additional Background}
\label{sec:model-background}

\subsection{Dynkin Isomorphism for the VRJP and the $\HH^{2|2}$-Model}
Analogous to the connection between the Gaussian free field and the (continuous-time) simple random walk, there is a Dynkin-type isomorphism theorem relating correlation functions of the $\HH^{2|2}$-model with the local time of a VRJP.

\begin{theorem}[{\cite[Theorem~5.6]{bauerschmidt_geometry_2020}}]
  \label{thm:h22-dynkin}
  Suppose $G = (V,E)$ is a finite graph with positive edge-weights $\{\beta_{ij}\}_{ij\in E}$.
  Let $\langle\cdot\rangle_{\beta,h}$ denote the expectation of the $\HH^{2|2}$-model and suppose that under $\EE_{i}$, the process $(X_{t})_{t\geq 0}$ denotes a VRJP started from $i$.
  Suppose $g\colon \RR^{V} \to \RR$ is a smooth bounded function.
  Then, for any $i,j\in V$
  \begin{equation}
    \label{eq:77}
    \eva{x_{i}x_{j} g(\mbf{z}-1)}_{\beta,h} = \int_{0}^{\infty} \EE_{i}[g(\mbf{L}_{t})\unit_{X_{t} = j}]e^{-ht}\dd{t},
  \end{equation}
  where $\mbf{L}_{t} = (L_{t}^{x})_{x\in V}$ denotes the VRJP's local time field.
\end{theorem}

This result will be key to deduce Theorem~\ref{thm:h22-asymptotics} from Theorem~\ref{thm:vrjp-near-crit}.

\subsection{VRJP as Random Walk in a $t$-Field Environment}
As a continuous-time process, there is some freedom in the time-parametrisation of the VRJP.
While the definition in \eqref{eq:138} (the \emph{linearly reinforced} timescale) is the ``usual'' parametrisation, we also make use of the \emph{exchangeable timescale} VRJP $(\tilde{X}_t)_{t\in \clop{0,+\infty}}$:
\begin{equation}\textstyle
  \label{eq:14}
  \tilde{X}_t \coloneqq X_{A^{-1}(t)} \qq{with} A(t) \coloneqq \int_{0}^{t} 2(1+L_{s}^{X_{s}})\dd{s} = \sum_{x\in V} [(1+L_{t}^{x})^{2} - 1]
\end{equation}
Writing $\tilde{L}_{t}^{x} = \int_{0}^{t}\unit\{\tilde{X}_{s} = x\}\dd{s}$, the local times in the two timescales are related by
\begin{equation}
  \label{eq:6}
  L_{t}^{x} = \sqrt{1+\tilde{L}_{t}^{x}} - 1.
\end{equation}
Above reparametrisation is motivated by the following result of Sabot and Tarrès \cite{sabot_edge-reinforced_2015}, showing that the VRJP in exchangeable timescale can be seen as a (Markovian) random walk in random conductances given in terms of the $t$-field.

\begin{theorem}[VRJP as Random Walk in Random Environment \cite{sabot_edge-reinforced_2015}]
  \label{thm:vrjp-random-walk-random-env}
  Consider a finite graph $G = (V,E)$, a starting vertex $i_{0} \in V$ and edge-weights $(\beta_e)_{e\in E}$.
  The exchangeable timescale VRJP, started at $i_{0}$, equals in law an (annealed) continuous-time Markov jump process, with jump rates between from $i$ to $j$ given by
  \begin{equation}
    \label{eq:141}
    \tfrac{1}{2}\beta_{ij}e^{T_j-T_{i}},
  \end{equation}
  where $\mbf{T} = (T_x)_{x\in V}$ are random variables distributed according to the law of the $t$-field \eqref{eq:156} pinned at $i_{0}$.
\end{theorem}

As a consequence of Theorem~\ref{thm:vrjp-random-walk-random-env}, the $t$-field can be recovered from the VRJP's asymptotic local time:

\begin{corollary}[$t$-field from Asymptotic Local Time \cite{sabot_vertex_2017}]\label{corr:t-field-as-local-time}
  Consider the setting of Theorem~\ref{thm:vrjp-random-walk-random-env}.
  Let $(L_{t}^{x})_{x\in V}$ and $(\tilde{L}_{t}^{x})_{x\in V}$ denote the local time field of the VRJP in linearly reinforced and exchangeable timescale, respectively.
  Then
  \begin{equation}\label{eq:85}
    \begin{aligned}
      T_{i} &\coloneqq \lim\limits_{t\rightarrow \infty} \log\left( L^{i}_{t}/L^{i_{0}}_{t}\right) \qquad (i \in V)\\
      \tilde{T}_{i} &\coloneqq \tfrac12 \lim\limits_{t\rightarrow \infty} \log\left( \tilde{L}^{i}_{t}/\tilde{L}^{i_{0}}_{t}\right) \qquad (i \in V)
    \end{aligned}
  \end{equation}
  exist and follow the law $\mcl{Q}^{(i_{0})}_{\beta}$ of the $t$-field in \eqref{eq:156}.
\end{corollary}
\begin{proof}
  For the exchangeable timescale, Sabot, Tarrès and Zeng \cite[Theorem~2]{sabot_vertex_2017} provide a proof.
  The statement for the usual (linearly reinforced) VRJP then follows by the time change formula for local times \eqref{eq:6}.
\end{proof}

Considering the VRJP as a random walk in random environment enables us to study its local time properties with the tools of random conductance networks.
For a $t$-field $\mbf{T} = (T_x)_{x\in V}$ pinned at $i_{0}$, we refer to the collection of random edge weights (or \emph{conductances})
\begin{equation}
  \label{eq:10}
  \{\beta_{ij}e^{T_{i} + T_{j}}\}_{ij \in E}
\end{equation}
as the \emph{$t$-field environment}.
This should be thought of as a symmetrised version of the VRJP's random environment \eqref{eq:141}.
It is easier to study a random walk with symmetric jump rates, since its amenable to the methods of conductance networks.
The following lemma relates local times in the $t$-field environment with the local times in the environment of the exchangeable timescale VRJP:

\begin{lemma}\label{lem:local-times-symm-exch}
  Consider the setting of Theorem~\ref{thm:vrjp-random-walk-random-env}.
  Let $(\tilde{X}_{t})_{t\geq 0}$ and $(Y_{t})_{t\geq 0}$ denote two continuous-time Markov jump processes started from $i_{0}$ with rates given by \eqref{eq:141} and \eqref{eq:10}, respectively.
  We write $\tilde{L}_{t}^{x}$ and $l_{t}^{x}$ for their respective local time fields.
  Let $B \subseteq V$ and write $\tilde{\mcl{T}_{B}}$ and $\mcl{T}_{B}$ for the respective hitting times of $B$.
  Then
  \begin{equation}
    \label{eq:49}
    L_{\tilde{\mcl{T}}_{B}}^{x} \stackrel{\tiny\text{law}}{=} 2 e^{T_{x}} l_{\mcl{T}_{B}}^{x},
  \end{equation}
  for $x \in V$.
  In particular, $L_{\tilde{\mcl{T}}_{B}}^{i_{0}} \stackrel{\tiny\text{law}}{=} 2 l_{\mcl{T}_{B}}^{i_{0}}$.
\end{lemma}

\begin{proof}
  The discrete-time processes associated to $(\tilde{X}_{t})_{t\geq 0}$ and $(Y_{t})_{t\geq 0}$ apparently agree.
  In particular, they both visit a vertex $x$ the same number of times, before hitting $B$.
  Every time $\tilde{X}_{t}$ visits the vertex $x$, it spends an $\mrm{Exp}(\sum_{y}\tfrac12\beta_{xy}e^{T_{y} - T_{x}})$-distributed time there, before jumping to another vertex.
  $Y_{t}$ on the other hand will spend time distributed as $\mrm{Exp}(\sum_{y}\beta_{xy}e^{T_{x} + T_{y}}) = \tfrac12 e^{-2T_{x}} \mrm{Exp}(\sum_{y}\tfrac12\beta_{xy}e^{T_{y} - T_{x}})$.
  This concludes the proof.
\end{proof}

\subsection{Effective Conductance}
Our approach to proving Theorem~\ref{thm:vrjp-near-crit} will rely on establishing asymptotics for the \emph{effective conductance} in the $t$-field environment (Theorem~\ref{thm:effective-conductance-nearcrit}).

\begin{definition}\label{def:effective-conductance}
  Consider a locally finite graph $G = (V,E)$ with edge weights (or \emph{conductances}) $\{w_{ij}\}_{ij \in E}$.
  For two disjoint sets $A,B \subseteq V$, the \emph{effective conductance} between them is defined as
  \begin{equation}
    \label{eq:154}
    C^{\mrm{eff}}(A,B) \coloneqq \inf\limits_{\substack{U\colon V\to \RR \\ U\vert_{A} \equiv 0,\, U\vert_{B} \equiv 1}} \sum_{ij \in E} w_{ij}\, (U(i) - U(j))^{2}.
  \end{equation}
\end{definition}

The variational definition \eqref{eq:154} makes it easy to deduce monotonicity and boundedness properties:

\begin{lemma}\label{lem:eff-cond-properties}
  Consider the situation of Definition~\ref{def:effective-conductance}.
  Suppose $S \subseteq E$ is a edge-cutset separating $A,B$.
  Then
  \begin{equation}
    \label{eq:153}
    C^{\mrm{eff}}(A,B) \leq \sum_{ij \in S} w_{ij}.
  \end{equation}
  Alternatively, suppose $C \subseteq V$ is a vertex-cutset separating $A,B$.
  Then
  \begin{equation}
    \label{eq:66}
    C^{\mrm{eff}}(A,B) \leq C^{\mrm{eff}}(A,C).
  \end{equation}
\end{lemma}

\begin{proof}
  For the first statement, consider \eqref{eq:154} for the function $U\colon V \to \RR$ that is constant zero (resp.\ one) in the component of $A$ (resp.\ $B$) in $V\wo S$.
  For the second statement, note that for any funcion $U \colon V \to \RR$ with $U\vert_{A} \equiv 0$ and $U\vert_{C} \equiv 1$ we can define a function $\tilde{U}$ that agrees with $U$ on $C$ and the connected compenent of $V\setminus C$ containing $A$, and is constant equal to one on the component of $B$ in $V\setminus V$.
  Then, $\tilde{U}\vert_{A} \equiv 0$ and $\tilde{U}\vert_{B} \equiv 1$ and $\sum_{ij\in E} w_{ij} (U(i) - U(j))^{2} \leq \sum_{ij\in E} w_{ij} (\tilde{U}(i) - \tilde{U}(j))^{2}$, which proves the claim.
\end{proof}

The monotoniciy in \eqref{eq:66} makes it possible to define an effective conductance \emph{to infinity}.
For an increasing exhaustion $V_{1} \subseteq V_{2} \subseteq \cdots$ of the vertex set $V = \bigcup_{n}V_{n}$ and a given finite set $A\subseteq V$, we define the \emph{effective conductance from $A$ to infinity} by
\begin{equation}
  \label{eq:68}
  C^{\mrm{eff}}_{\infty}(A) = \lim_{n\to\infty} C^{\mrm{eff}}(A,V\setminus V_{n}).
\end{equation}
One may check that this is independent from the choice of exhaustion.
For us, the main use of effective conductances stems from their relation to \emph{escape times}:

\begin{lemma}\label{lem:escape-times-conductance}
  Consider a locally finite graph $G = (V,E)$ with edge weights (or \emph{conductances}) $\{w_{ij}\}_{ij \in E}$.
  Let $C^{\mrm{eff}}(i_{0},B)$ denote the effective conductance between the singleton $\{i_{0}\}$ and a disjoint set $B$.
  Consider a continuous-time random walk $(X_{t})_{t\geq 0}$ on $G$, starting at $X_{0} = i_{0}$ and jumping from $X_{t} = i$ to $j$ at rate $w_{ij}$.
  Let $L_{\mrm{esc}}(i_{0},B)$ denote the total time the walk spends at $i_{0}$ before visiting $B$ for the first time.
  Then $L_{\mrm{esc}}(i_{0},B)$ is distributed as an $\mrm{Exp}(1/C^{\mrm{eff}}(i_{0},B))$-random variable.

  \indent
  For an infinite graph $G$, the above conclusions also hold for $B$ ``at infinity'': We let $L_{\mrm{esc},\infty}(i_{0})$ denote the total time spent at $i_{0}$ and understand $C_{\infty}^{\mrm{eff}}(i_{0})$ as in \eqref{eq:68}. Then $L_{\mrm{esc},\infty}(i_{0}) \sim \mrm{Exp}(1/C^{\mrm{eff}}_{\infty}(i_{0}))$.
\end{lemma}

\begin{proof}
  According to \cite[Section~2.2]{lyons_probability_2016}, the walk's number of visits at $i_{0}$ before hitting $B$ is a geometric random variable $N\sim \mrm{Geo}(p_{\mrm{esc}})$ with the \emph{escape probability} $p_{\mrm{esc}} = C^{\mrm{eff}}(i_{0},B)/(\sum_{j\sim i_{0}} w_{i_{0}j})$.
  Moreover, for the continuous-time process, every time we visit $i_{0}$ we spend an $\mrm{Exp}(\sum_{j\sim i_{0}} w_{i_{0}j})$-distributed time there, before jumping to a neighbour.
  Hence, $L_{\mrm{esc}}(i_{0},B)$ is distributed as the sum of $N$ independent $\mrm{Exp}(\sum_{j\sim i_{0}} w_{i_{0}j})$-distributed random variables.
  By standard results for the exponential distribution (easily checked via its moment-generating function), this implies the claim.
  Note that this argument also holds true for $B$ ``at infinity'', in which case $N\sim \mrm{Geo}(p_{\mrm{esc}})$ with $p_{\mrm{esc}} = C^{\mrm{eff}}_{\infty}(i_{0})/(\sum_{j\sim i_{0}} w_{i_{0}j})$ will simply denote the total number of visits at $i_{0}$ (see \cite[Section~2.2]{lyons_probability_2016} for more details).
\end{proof}

\subsection{The $t$-Field from the $\HH^{2|2}$- and STZ-Anderson Model}
\label{sec:t-field-connections}

\paragraph{$t$-Field as a Horospherical Marginal of the $\HH^{2|2}$-model}
First we introduce \emph{horospherical coordinates} on $\HH^{2|2}$.
In these coordinates, $\mbf{u} \in \HH^{2|2}$ is parametrised by $(t,s,\bar{\psi},\psi)$, with $t,s \in \RR$ and Grassmann variables $\bar{\psi},\psi$ via
\begin{equation}
  \label{eq:36}
  \mqty(z\\ x\\ y\\ \xi \\ \eta)
  = \mqty(\cosh(t) + e^{t} (\tfrac12 s^{2} + \bar{\psi}\psi) \\ \sinh(t) - e^{t} (\tfrac12 s^{2} + \bar{\psi}\psi) \\ e^{t}s \\ e^{t}\bar{\psi} \\ e^{t}\psi).
\end{equation}
A particular consequence of this is that $e^{t} = z + x$.
By rewriting the Gibbs measure for the $\HH^{2|2}$-model, defined in \eqref{eq:148}, in terms of horospherical coordinates and integrating out the fermionic variables $\psi, \bar{\psi}$, one obtains a marginal density in $\underline{t} = \{t_{x}\}_{x\in V}$ and $\underline{s} = \{s_{x}\}_{x\in V}$, which can be interpreted probabilistically:

\begin{lemma}[Horospherical Marginal of the $\HH^{2|2}$-Model {\cite{disertori_anderson_2010,disertori_quasi-diffusion_2010,bauerschmidt_dynkin_2019}.}]
  \label{lem:h22-horospherical-marginal}
  Consider a finite graph $G = (V,E)$, a vertex $i_{0} \in V$, and non-negative edge-weights $(\beta_{ij})_{ij\in E}$.
  There exist random variables $\underline{T} = \{T_{x}\}_{x\in V} \in \RR^{V}$ and $\underline{S} = \{S_{x}\}_{x \in V} \in \RR^{V}$, such that for any $F \in C^{\infty}_{\mrm{c}}(\RR^{V}\times\RR^{V})$
  \begin{equation}
    \label{eq:72}
    \eva{F(\underline{t}, \underline{s})}_{\beta} = \EE[F(\underline{T}, \underline{S})].
  \end{equation}
  The law of $\underline{T}$ is given by the $t$-field pinned at $i_{0}$ (see Definition~\ref{def:t-field}).
  Moreover, conditionally on $\underline{T}$, the \emph{$s$-field} follows the law of a Gaussian free field in conductances $\{\beta_{ij}e^{T_{i} + T_{j}}\}_{ij \in E}$, pinned at $i_{0}$, $S_{i_{0}} = 0$.
\end{lemma}

\paragraph{$t$-Field and the STZ-Anderson Model.}
It turns out that the (zero-energy) Green's function of the STZ-Anderson model is directly related to the $t$-field:

\begin{proposition}[\cite{sabot_vertex_2017}]
  \label{prop:stz-t-field-coupling}
  For $H_{B}$ denoting the STZ-Anderson model as in Definition~\ref{def:stz-anderson} define the Green's function $G_{B}(i,j) = [H_{B}^{-1}]_{i,j}$.
  For a vertex $i_{0} \in V$, define $\{T_{i}\}_{i\in \Lambda}$ via
  \begin{equation}
    \label{eq:169}
    e^{T_{i}} \coloneqq G_{B}(i_{0}, i)/G_{B}(i_{0},i_{0}).
  \end{equation}
  Then $\{T_{i}\}$ follows the law $\mcl{Q}_{\beta}^{(i_{0})}$ of the $t$-field, pinned at $i_{0}$.
  Moreover, with $\{T_{i}\}$ as above we have $B_{i} = \sum_{j\sim i}\beta_{ij}e^{T_{j} - T_{i}}$ for all $i \in V\setminus\{i_{0}\}$.
\end{proposition}

This provides a way of coupling the STZ-field with the $t$-field, as well as a coupling of $t$-fields pinned at different vertices.

\begin{remark}[Natural Coupling]\label{rmk:natural-coupling}
  Lemma~\ref{lem:h22-horospherical-marginal} and Proposition~\ref{prop:stz-t-field-coupling} give us a way to define a \emph{natural coupling} of STZ-field, $t$-field and $s$-field as follows:
  Fix some pinning vertex $i_{0} \in V$.
  Sample an STZ-Anderson model $H_{B}$ with respect to edge weights $\{\beta_{ij}\}_{ij \in E}$.
  Then define the $t$-field $\{T_{i}\}_{i\in V}$, pinned at $i_{0}$ via \eqref{eq:169}.
  Then, conditionally on the $t$-field, sample the $s$-field $\{S_{i}\}_{i\in V}$ as a Gaussian free field in conductances $\{\beta_{ij}e^{T_{i} + T_{j}}\}_{ij \in E}$, pinned at $i_{0}$, $S_{i_{0}} = 0$.

\end{remark}

\subsection{Monotonicity Properties of the $t$-Field}
A rather surprising property of the $t$-field, proved by the first author, is the monotonicity of various expectation values with respect to the edge-weights.
The following is a restatement of \cite[Theorem~6]{poudevigne-auboiron_monotonicity_2022} after applying Proposition~\ref{prop:stz-t-field-coupling}:

\begin{theorem}[{\cite[Theorem~6]{poudevigne-auboiron_monotonicity_2022}}]
  \label{thm:monotonicity-t-field}
  Consider a finite graph $G = (V,E)$ and fix some vertex $i_{0} \in V$.
  Under $\EE_{\pmb{\beta}}$, we let $\mbb{T} = \{T_{i}\}_{i\in V}$ denote a $t$-field pinned at $i_{0}$ with respect to non-negative edge weights $\pmb{\beta} = \{\beta_{e}\}_{e\in E}$.
  Then, for any convex $f\colon \clop{0,\infty} \to \RR$ and non-negative $\{\lambda_{i}\}_{i\in V}$, the map
  \begin{equation}\textstyle
    \label{eq:83}
    \pmb{\beta} \mapsto \EE_{\pmb{\beta}}[f(\sum_{i}\lambda_{i} e^{T_{i}})]
  \end{equation}
  is decreasing.
\end{theorem}

A direct corollary of the above is that expectations of the form $\EE_{\beta}[e^{\eta T_{x}}]$ are increasing in $\beta$ for $\eta \leq [0,1]$ and are decreasing for $\eta \geq 1$.
This will be the extent to which we make use of the result.

\subsection{The $t$-Field on $\mbb{T}_{d}$}
Consider the $t$-field measure \eqref{eq:156} on $\mbb{T}_{d,n} = (V_{d,n}, E_{d,n})$, the rooted $(d+1)$-regular tree of depth $n$, pinned at the root $i_{0} = 0$.
Only one term contributes to the determinantal term \eqref{eq:165}, namely the term corresponding to $\mbb{T}_{d,n}$ itself, oriented away from the root:
\begin{equation}
  \label{eq:155}
  \mcl{Q}_{\beta; \mbb{T}_{d,n}}^{(0)}(\dd{\mbf{t}})
  = e^{-\sum_{(i,j) \in \vec{E}_{d,n}}[\beta\, (\cosh(t_{j} - t_{i}) - 1) + \tfrac12 (t_{j} - t_{i})]} \delta(t_{0}) \prod_{i \in V_{d,n} \wo 0} \frac{\dd{t}_{i}}{\sqrt{2\pi/\beta}},
\end{equation}
where $\vec{E}_{d,n}$ is the set of edges in $\mbb{T}_{d,n}$ oriented away from the root.
In other words, the increments of the $t$-field along outgoing edges are i.i.d.\ and distributed according to the following:

\begin{definition}[$t$-field Increment Measure]\label{def:t-field-inc}
  For $\beta > 0$ define the probability distribution
  \begin{equation}
    \label{eq:182}
    \mcl{Q}^{\mrm{inc}}_{\beta}(\dd{t}) = e^{-\beta [\cosh(t) - 1] - t/2} \frac{\dd{t}}{\sqrt{2\pi/\beta}} \qq{with} t \in \RR.
  \end{equation}
  We refer to this as the \emph{$t$-field increment distribution} and if not specified otherwise, $T$ will always denote a random variable with distribution $\mcl{Q}^{\mrm{inc}}_{\beta}$.
  The dependence on $\beta$ is either implicit or denoted by a subscript, such as in $\EE_{\beta}$ or $\PP_{\beta}$.
\end{definition}

The density \eqref{eq:182} implies that
\begin{equation}
  \label{eq:186}
  e^{T} \sim \mrm{IG}(1,\beta) \qq{and} e^{-T} \sim \mrm{RIG}(1,\beta),
\end{equation}
where IG (RIG) denotes the (reciprocal) inverse Gaussian distribution (\cf \eqref{eq:111}).
Note that changing variables to $t\mapsto e^{t}$ and comparing to the density of the inverse Gaussian, we see that \eqref{eq:182} is normalised.

\begin{definition}[Free Infinite Volume $t$-field on $\mbb{T}_{d}$]
  \label{def:infinite-vol-t-field}
  For $\beta > 0$, associate to every edge $e$ of the infinite rooted $(d+1)$-regular tree $\mbb{T}_{d}$ a $t$-field increment $\tilde{T}_{e}$, distributed according to \eqref{eq:182}.
  For every vertex $x \in \mbb{T}_{d}$ let $\gamma_{x}$ denote the unique self-avoiding path from $0$ to $x$ and define $T_{x} \coloneqq \sum_{e\in \gamma_{x}} \tilde{T}_{e}$.
  The random field $\{T_{x}\}_{x\in \mbb{T}_{d}}$ is the \emph{free infinite volume $t$-field} on $\mbb{T}_{d}$ at inverse temperature $\beta > 0$.
  In particular, its restriction $\{T_{x}\}_{x \in \mbb{T}_{d,n}}$ onto vertices up to generation $n$ follows the law $\mcl{Q}^{(0)}_{\beta;\mbb{T}_{d,n}}$.
\end{definition}

By construction, $\{T_{x}\}_{x\in \mbb{T}_{d}}$ can be considered a branching random walk (BRW) with a deterministic number of offsprings (every particle gives rise to $d$ new particles in the next generation).
In Section~\ref{sec:brw-background} we will elaborate on this perspective.

\subsection{Previous Results for VRJP on Trees.}
As we have already noted in the introduction, the VRJP on tree graphs has received quite some attention \cite{davis_vertex-reinforced_2004, collevecchio_limit_2009, basdevant_continuous-time_2012, chen_speed_2018, rapenne_about_2023}.
In particular, Basdevant and Singh \cite{basdevant_continuous-time_2012} studied the VRJP on Galton-Watson trees with general offspring distribution, and exactly located the recurrence/transience phase transition:

\begin{proposition}[Basdevant-Singh \cite{basdevant_continuous-time_2012}]\label{prop:basdevant-singh}
  Let $\mcl{T}$ denote a Galton-Watson tree with mean offspring $b > 1$.
  Consider the VRJP started from the root of $\mcl{T}$, conditionally on non-extinction of the tree.
  There exists a critical parameter $\beta_{\mrm{c}} = \beta_{\mrm{c}}(b)$, such that the VRJP is
  \begin{itemize}[nosep]
    \item recurrent for $\beta\leq \beta_\mrm{c}$,
    \item transient for $\beta> \beta_\mrm{c}$.
  \end{itemize}
  Moreover, $\beta_{\mrm{c}}$ is characterised as the unique positive solution to
  \begin{equation}
    \label{eq:139}
    \frac{1}{b} = \sqrt{\frac{\beta_{\mrm{c}}}{2\pi}}\int_{-\infty}^{+\infty}\dd{t} e^{-\beta_{\mrm{c}} (\cosh(t) - 1)}.
  \end{equation}
\end{proposition}

We also take the opportunity to highlight Rapenne's recent results \cite{rapenne_about_2023} concerning the (sub)critical phase, $\beta \leq \beta_{\mrm{c}}$.
His statements can be seen to complement our results, which focus on the supercritical phase $\beta > \beta_{\mrm{c}}$.

\subsection{Background on Branching Random Walks}
\label{sec:brw-background}
Let's quickly recall some basic results from the theory of branching random walks. For a more comprehensive treatment we refer to Shi's monograph \cite{shi_branching_2015}.

\indent
A \emph{branching random walk} (BRW) with offspring distribution $\mu \in \mrm{Prob}(\NN_{0})$ and increment distribution $\nu$ is constructed as follows:
We start with a ``root'' particle $x = 0$ at generation $\abs{0} = 0$ and starting position $V(0) = v_{0}$.
We sample its number of offsprings according to $\mu$.
They constitute the particles at generation one, $\{\abs{x} = 1\}$.
Every such particle is assigned a position $v_{0} + \delta V_{x}$ with $\{\delta V_{x}\}_{\abs{x} = 1}$ being i.i.d.\ according to the increment distribution $\nu$.
This process is repeated recursively and we end up with a random collection of particles $\{x\}$, each equipped with a position $V(x) \in \RR$, a \emph{generation} $\abs{x} \in \NN_{0}$ and a \emph{history} $0=x_{0}, x_{1}, \ldots, x_{\abs{x}} = x$ of predecessors.
Unless otherwise stated, we assume from now on that a BRW always starts from the origin, $v_{0} = 0$.

\indent
A particularly useful quantity for the study of BRWs is the $\log$-Laplace transform of the offspring process:
\begin{equation}
  \label{eq:157}
  \psi(\eta) \coloneqq \log\EE \Big[ \sum_{\abs{x} = 1} e^{- \eta V(x)} \Big],
\end{equation}
where the sum goes over all particles in the first generation.
A priori, we have $\psi(\eta) \in [0,\infty]$, but we typically assume $\psi(0) > 0$ and $\inf_{\eta >0} \psi(\eta) < \infty$.
The first assumption corresponds to supercriticality of the offspring distribution%
\footnote{Here we mean supercriticality in the sense of Galton-Watson trees. In other words, with positive probability the BRW consists of infinitely many particles. We also say that the BRW does not go extinct.},
whereas the second assumption enables us to study the average over histories of the BRW in terms of single random walk:

\begin{proposition}[Many-To-One Formula]\label{prop:many-to-one}
  Consider a BRW with log-Laplace transform $\psi(\eta)$.
  Choose $\eta > 0$ such that $\psi(\eta) < \infty$ and define a random walk $0 = S_{0}, S_{1}, \ldots$ with i.i.d.\ increments such that for any measurable $h\colon \RR \to \RR$
  \begin{equation}\textstyle
    \label{eq:158}
    \EE[h(S_{1})] = \EE \left[ \sum_{\abs{x} = 1}e^{-\eta V(x)} h(V(x)) \right] \Big/ \EE \left[ \sum_{\abs{x} = 1} e^{-\eta V(x)} \right].
  \end{equation}
  Then, for all $n\geq 1$ and $g\colon \RR^{n} \to \clop{0,\infty}$ measurable we have
  \begin{equation}\textstyle
    \label{eq:159}
    \EE \left[ \sum_{\abs{x} = n} g(V(x_{1}), \ldots, V(x_{n}))\right] = \EE \left[ e^{n\psi(\eta) + \eta S_{n}} g(S_{1}, \ldots, S_{n}) \right].
  \end{equation}
\end{proposition}

For a proof we refer to Shi's lecture notes \cite[Theorem~1.1]{shi_branching_2015}.
An application of the many-to-one formula is the following statement about the velocity of extremal particles (\cf \cite[Theorem~1.3]{shi_branching_2015}).

\begin{proposition}[Asymptotic Velocity of Extremal Particles]\label{prop:asymp-velocity}
  Suppose $\psi(0) > 0$ and $\inf\limits_{\eta > 0} \psi(\eta) < \infty$.
  Then, almost surely under the event of non-extinction, we have
  \begin{equation}
    \label{eq:160}
    \lim_{n\to\infty} \frac{1}{n} \inf_{\abs{x} = n} V(x) = -\inf_{\eta > 0} \psi(\eta)/\eta.
  \end{equation}
\end{proposition}

\paragraph{Critical Branching Random Walks.}
A common assumption, under which BRWs exhibit various universal properties, is $\psi(1) = \psi'(1) = 0$.
While not common terminology in the literature, we will refer to this as \emph{criticality}:
\begin{equation}\textstyle
  \label{eq:187}
  \text{BRW with } \psi(\eta) = \log \EE[\sum_{\abs{x} = 1} e^{-\eta V(x)}] \text{ is \emph{critical}}
  \quad\overset{\tiny\text{def}}{\Longleftrightarrow}\quad \psi(1) = \psi'(1) = 0
\end{equation}
This definition can be motivated by considering the many-to-one formula (Proposition~\ref{prop:many-to-one}) applied to a critical BRW for $\eta = 1$:
In that case, the random walk $S_{i}$ has mean zero increments, $\mbb{E}[S_{1}] = -\psi'(1) = 0$, and the exponential drift in \eqref{eq:159} vanishes, $e^{n\psi(1)} = 1$.
Consequently, as far as the many-to-one formula is concerned, critical BRWs inherit some of the universality of mean zero random walks (\eg Donsker's theorem, say under an additional second moment assumption).
Moreover, the notion of criticality is particularly useful, since in many cases we can reduce a BRW to the critical case by a simple rescaling/drift transformation:

\begin{lemma}[Critical Rescaling of a BRW]\label{lem:critical-rescaling}
  Consider a BRW with log-Laplace transform $\psi(\eta) = \log\EE[\sum_{\abs{x} = 1} e^{-\eta V(x)}]$. Suppose there exists $\eta^{\ast} > 0$ solving the equation
  \begin{equation}
    \label{eq:176}
    \psi(\eta^{\ast}) = \eta^{\ast} \psi'(\eta^{\ast}).
  \end{equation}
  Equivalently, $\eta^{\ast}$ is a critical point for $\eta \to \psi(\eta)/\eta$.
  Define a BRW with the same particles $\{x\}$ and rescaled positions
  \begin{equation}
    \label{eq:178}
    V^{\ast}(x) = \eta^{\ast} V(x) + \psi(\eta^{\ast}) \abs{x}.
  \end{equation}
  The resulting BRW is critical.
\end{lemma}
\begin{proof}
  Write $\psi^{\ast}(\gamma) = \log \EE \sum_{\abs{x} = 1} e^{-\gamma V^{\ast}(x)}$ for the log-Laplace transform of the rescaled BRW.
  We easily check
  \begin{equation}
    \label{eq:78}
    \begin{aligned}
      \psi^{\ast}(1)
      = \log \EE\sum_{\abs{x}=1}e^{-\eta^{\ast}V(x) - \psi(\eta^{\ast})}
      &=  -\psi(\eta^{\ast}) + \log \EE\sum_{\abs{x}=1}e^{-\eta^{\ast}V(x)}\\
      &= -\psi(\eta^{\ast}) + \psi(\eta^{\ast})
      = 0.
    \end{aligned}
  \end{equation}
  Equivalently, $1 = \EE\sum_{\abs{x}=1}e^{-\eta^{\ast}V(x) - \psi(\eta^{\ast})}$, which together with \eqref{eq:176} yields
  \begin{equation}
    \label{eq:79}
    \begin{aligned}
      (\psi^{\ast})'(1)
      &= -\frac{\EE\sum_{\abs{x}=1}(\eta^{\ast}V(x) + \psi(\eta^{\ast}))e^{-\eta^{\ast}V(x) - \psi(\eta^{\ast})}}{\EE\sum_{\abs{x}=1}e^{-\eta^{\ast}V(x) - \psi(\eta^{\ast})}}\\
      &= -\eta^{\ast} \EE\sum_{\abs{x}=1}V(x)e^{-\eta^{\ast} V(x)} - \psi(\eta^{\ast})\\
      &= \eta^{\ast}\psi'(\eta^{\ast}) - \psi(\eta^{\ast}) = 0,
    \end{aligned}
  \end{equation}
  which concludes the proof.
\end{proof}

\newpage
\section{VRJP and the $t$-Field as $\beta \searrow \beta_{\mrm{c}}$}
\label{sec:near-crit-t-field}
The main goal of this section is to prove Theorem~\ref{thm:vrjp-near-crit} on the asymptotic escape time of the VRJP as $\beta \searrow \beta_{\mrm{c}}$.
The main work will be in establishing the following result on the effective conductance in a $t$-field environment:

\begin{theorem}[Near-Critical Effective Conductance]\label{thm:effective-conductance-nearcrit}
  Let $\{T_{x}\}_{x\in \mbb{T}_{d}}$ denote the (free) $t$-field on $\mbb{T}_{d}$, pinned at the origin.
  Let $C^{\mrm{eff}}_{\infty}$ denote the effective conductance from the origin to infinity in the network given by conductances $\{\beta e^{T_{i} + T_{j}} \unit_{i\sim j}\}_{i,j\in \mbb{T}_{d}}$.
  There exist constants $c,C > 0$ such that
  \begin{equation}
    \label{eq:9}
    \exp[-(C+o(1))/\sqrt{\epsilon}] \leq \EE_{\beta_{\mrm{c}} + \epsilon}[C^{\mrm{eff}}_{\infty}] \leq \exp[-(c+o(1))/\sqrt{\epsilon}],
  \end{equation}
  as $\epsilon \searrow 0$, where $\beta_{\mrm{c}} = \beta_{\mrm{c}}(d) > 0$ is given by Proposition~\ref{eq:139}.
\end{theorem}

For establishing this result, the BRW perspective onto the $t$-field is essential.
The lower bound will follow from a mild modification of a result by Gantert, Hu and Shi \cite{gantert_asymptotics_2011} (see Theorem~\ref{thm:uniform-ghs}).
For the upper bound we will consider the critical rescaling of the near-critical $t$-field (\cf Lemma~\ref{lem:critical-rescaling}).
The bound will then follow by a perturbative argument applied to a result on effective conductances in a \emph{critical} BRW environment.
The latter we prove in a more general form, for which it is convenient to introduce some additional notions.

\indent
For a random variable $V$ and a fixed offspring degree $d$ we write
\begin{equation}
  \label{eq:216}
  \psi_{V}(\eta) \coloneqq \log (d\, \EE[e^{-\eta V}]).
\end{equation}
Analogous to Definition~\ref{def:infinite-vol-t-field}, for an increment distribution given by $V$, we define a random field $\{V_{x}\}_{x\in\mbb{T}_{d}}$ and refer to it as the \emph{BRW with increments $V$}.
We say that $V$ is a \emph{critical increment} if $\{V_{x}\}_{x \in \mbb{T}_{d}}$ is critical, \ie $\psi_{V}(1) = \psi_{V}'(1) = 0$.
Note that this implicitly depends on our choice of $d \geq 2$, but we choose to suppress this dependency.
For a critical increment $V$ we write
\begin{equation}
  \label{eq:215}
  \sigma_{V}^{2} \coloneqq \psi''_{V}(1) = d\,\EE[V^{2} e^{-V}].
\end{equation}
Note that this is the variance of the (mean-zero) increments of the random walk $(S_{i})_{i\geq 0}$ given by the many-to-one formula (Proposition~\ref{prop:many-to-one} for $\eta = 1$).

\begin{theorem}\label{thm:critical-eff-cond}
  Fix some offspring degree $d \geq 2$ and consider a critical increment $V$ with $\sigma_{V}^{2} < \infty$ and $\psi_{V}(1+2a) < \infty$ for some constant $a > 0$.
  Write $\{V_{x}\}_{x \in \mbb{T}_{d}}$ for the BRW with increments $V$ and define the conductances $\{e^{-\gamma(V_{x} + V_{y})}\}_{xy}$.
  Let $C_{n,\gamma}^{\mrm{eff}}$ denote the effective conductance between the origin $0$ and the vertices in the $n$-th generation.
  Then, for $\gamma \in (1/2, 1/2 + a)$, we have
  \begin{equation}
    \label{eq:60}
    \EE[C_{n,\gamma}^{\mrm{eff}}]
    \leq \exp\Big[-\big[ \min(\tfrac14,\gamma - \tfrac12)\, (\pi^{2}\sigma_{V}^{2})^{1/3} +o(1)\big]n^{1/3}\Big] \qq{as} n\to\infty.
  \end{equation}
  Moreover, this is uniform with respect to $\gamma$, $\sigma_{V}^{2}$ and $\psi_{V}(1+2a)$ in the following sense:
  Suppose there is a family $V^{(k)}$, $k\in \NN$, of critical increments and define $C^{\mrm{eff}}_{n,\gamma;k}$ as above.
  Further assume $0 < \inf_{k}\sigma^{2}_{V^{(k)}} \leq \sup_{k}\sigma^{2}_{V^{(k)}} < \infty$ and $\sup_{k} \psi_{V^{(k)}}(1+2a) < \infty$.
  Then we have
  \begin{equation}
    \label{eq:218}
    \limsup_{n\to\infty}\, \sup_{k} \sup_{\frac12 < \gamma < \frac12 + a}\, \Bigg(
    n^{-1/3}\log\EE[C^{\mrm{eff}}_{n,\gamma;k}] +
    \min(\tfrac14,\gamma - \tfrac12)\, (\pi^{2}\sigma_{V^{(k)}}^{2})^{1/3}\Bigg) \leq 0.
  \end{equation}
\end{theorem}

We note that random walk in (critical) multiplicative environments on trees has previously been studied, see for example \cite{lyons_random_1992, pemantle_critical_1995, menshikov_random_2001, hu_slow_2007, hu_minimal_2009, faraud_almost_2012}.
In particular, Hu and Shi \cite[Theorem~2.1]{hu_slow_2007} established bounds analogous to \eqref{eq:60} for escape probabilities, instead of effective conductances.
While the quantities are related, bounds on the expected escape probability do not directly translate into bounds for the expected effective conductance.
Moreover, their setup for the random environment does not directly apply to our setting%
\footnote{Roughly speaking, they are working with weights $\{e^{-\gamma V_{x}}\}_{(x,y) \in \vec{E}(\mbb{T}_{d})}$ while we consider the ``symmetrised'' variant $\{e^{-\gamma(V_{x} + V_{y})}\}_{xy \in E(\mbb{T}_{d})}$.}.
Last but not least, for our applications, we require additional uniformity of the bounds with respect to the underlying BRW.

\subsection{The $t$-Field as a Branching Random Walk}
\label{sec:t-field-brw}
\begin{figure}[t]
  \centering
  \includegraphics[width=0.6\linewidth]{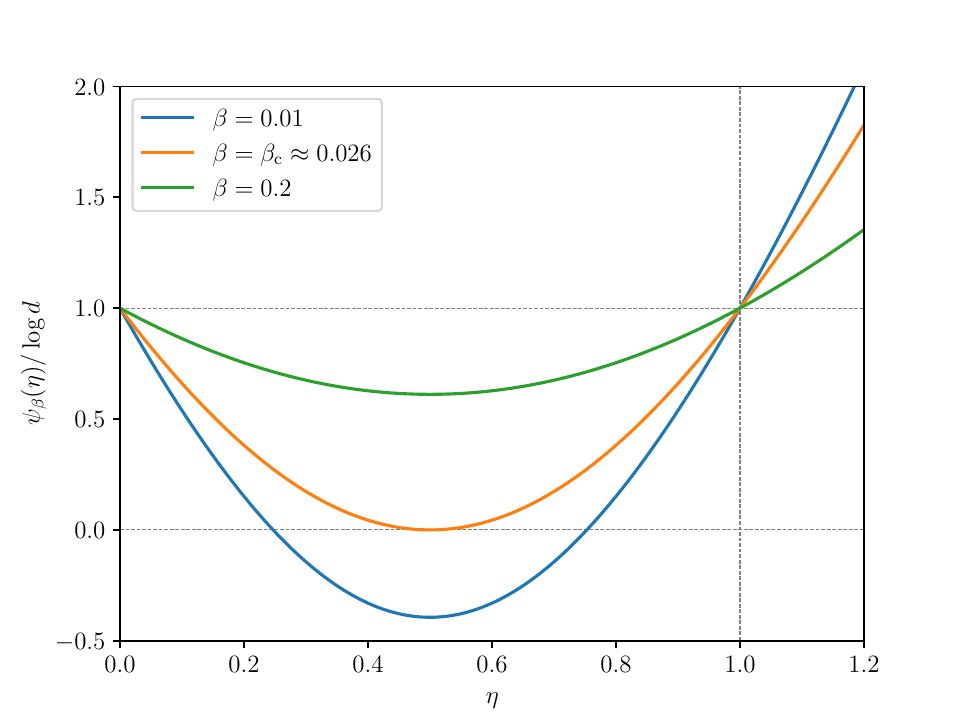}
  \caption{Illustration of $\psi_{\beta}(\eta)/\log(d)$ for $d=2$ at different values of $\beta$. Its minimum is always at $\eta = 1/2$, and the value of this minimum is increasing with $\beta$. It is equal to zero at $\beta = \beta_{\mrm{c}}$.}
  \label{fig:psi_regimes}
\end{figure}

Considered as a BRW, the $t$-field $\{T_{x}\}_{x\in \mbb{T}_{d}}$ on the rooted $(d+1)$-regular tree $\mbb{T}_{d}$ (or more precisely the negative $t$-field) has a log-Laplace transform given by
\begin{equation}
  \label{eq:179}
  \psi_{\beta}(\eta) \coloneqq \log\EE[\sum_{\abs{x} = 1} e^{\eta T_{x}}] = \log (d\, \EE_{\beta}[e^{\eta T}]) \qquad (\eta > 0),
\end{equation}
where $T$ denotes the $t$-field increment as introduced in Definition~\ref{def:t-field-inc}.
One can check easily that $\psi_{\beta}(0) = \psi_{\beta}(1) = \log d$.
More generally, using the density for $T$ we have
\begin{equation}
  \label{eq:162}
  \psi_{\beta}(\eta)
  = \log \Big(d \int\frac{\dd{t}}{\sqrt{2\pi/\beta}} e^{- \beta\, [\cosh(t) - 1] - (\tfrac12 - \eta)\, t}\Big)
  = \log \Big(\frac{ d\sqrt{2\beta}e^{\beta} }{\sqrt{\pi}} K_{\eta - \tfrac12}(\beta)\Big)
\end{equation}
where $K_{\alpha}$ denotes the modified Bessel function of second kind.
An illustration of $\psi_{\beta}$ for different values of $\beta$ is given in Figure~\ref{fig:psi_regimes}.
In particular, it's a smooth function in $\beta,\eta > 0$ and one may check that it's strictly convex since
\begin{equation}
  \label{eq:43}
  \psi_{\beta}^{\prime\prime}(\eta) = \frac{\EE_{\beta}[T^{2}e^{\eta T}]}{\EE_{\beta}[e^{\eta T}]} - \frac{\EE_{\beta}[T e^{\eta T}]^{2}}{\EE_{\beta}[e^{\eta T}]^{2}} > 0
\end{equation}
equals the variance of a non-deterministic random variable.
Moreover, by the symmetry and monotonicity properties of the Bessel function ($K_{\alpha} = K_{-\alpha}$ and $K_{\alpha} \leq K_{\alpha'}$ for $0 \leq \alpha \leq \alpha'$), the infimum of $\psi_{\beta}(\eta)$ is attained at $\eta = 1/2$:
\begin{equation}
\label{eq:163}
\inf_{\eta > 0} \psi_{\beta}(\eta) = \psi_{\beta}(1/2) = \log (d \, \EE_{\beta} [ e^{T/2}]) = \log (\frac{ \sqrt{2\beta}e^{\beta} d}{\sqrt{\pi}} K_{0}(\beta))
\end{equation}
The critical inverse temperature $\beta_{\mrm{c}} = \beta_{\mrm{c}}(d) > 0$, as given in Proposition~\ref{prop:basdevant-singh}, is equivalently characterised by the vanishing of this infimum:
\begin{equation}
\label{eq:164}
\psi_{\beta_{\mrm{c}}}(1/2) = \inf_{\eta > 0} \psi_{\beta_{\mrm{c}}}(\eta) = 0.
\end{equation}

In particular, by Lemma~\ref{lem:critical-rescaling}, this implies that $\{-\tfrac12 T_{x}\}_{x\in\mbb{T}_{\mrm{d}}}$ is a critical BRW at $\beta=\beta_{\mrm{c}}$.
More generally, it will be useful to consider critical rescalings of $\{T_{x}\}$ for general $\beta > 0$.
For this we write
\begin{equation}
  \label{eq:17}
  \eta_{\beta} \coloneqq \mrm{argmin}_{\eta > 0} \frac{\psi_{\beta}(\eta)}{\eta} \qq{and} \gamma_{\beta} \coloneqq  \inf_{\eta > 0} \frac{\psi_{\beta}(\eta)}{\eta} = \frac{\psi_{\beta}(\eta_{\beta})}{\eta_{\beta}}.
\end{equation}
An illustration of these quantities is given in Figure~\ref{fig:crit_parameter_regimes}.
If $\eta_{\beta}$ as above is well-defined, then it satisfies \eqref{eq:176} and hence by Lemma~\ref{lem:critical-rescaling} the rescaled field
\begin{equation}
  \label{eq:180}
  \tau^{\beta}_{x} = - \eta_{\beta} T_{x} + \psi_{\beta}(\eta_{\beta}) \abs{x}
\end{equation}
defines a critical BRW.
The following lemma lends rigour to this:

\begin{lemma}\label{lem:crit-exp-behaviour}
  $\eta_{\beta}$ as given in \eqref{eq:17} is well-defined and the unique positive root of the strictly increasing map $\eta \mapsto \eta \psi_{\beta}^{\prime}(\eta) - \psi_{\beta}(\eta)$.
  Consequently, the maps $\beta \mapsto \eta_{\beta}$ and $\beta \mapsto \gamma_{\beta}$ are continuously differentiable.
\end{lemma}

\begin{proof}
  Recall the Bessel function asymptotics $K_{\alpha}(\beta) \sim \tfrac12 (2/\beta)^{\alpha} \Gamma(\alpha)$ as $\alpha \to \infty$, hence by \eqref{eq:162} we have $\psi_{\beta}(\eta) \sim \eta \log\eta$ for $\eta \to \infty$.
  Consequently, $\psi_{\beta}(\eta)/\eta$ diverges as $\eta \to\infty$ (and it also diverges as $\eta \searrow 0$).
  Hence it attains its infimum at some finite value.
  We claim that there is a unique minimiser $\eta_{\beta}$.
  Since $\psi_{\beta}(\eta)/\eta$ is continuously differentiable in $\eta > 0$, at any minimum it will have vanishing derivative $\de_{\eta} (\psi_{\beta}(\eta)/\eta) = [\eta \psi_{\beta}^{\prime}(\eta) - \psi_{\beta}(\eta)]/\eta^{2}$.
  And in fact the map $\eta \mapsto \eta \psi_{\beta}^{\prime}(\eta) - \psi_{\beta}(\eta)$ is strictly increasing, since its derivative equals $\eta \psi_{\beta}^{\prime\prime}(\eta) > 0$, see \eqref{eq:43}, and as such has at most one root.
  This implies that $\eta_{\beta}$ as in \eqref{eq:17} is well-defined and the unique root of $\eta \psi_{\beta}^{\prime}(\eta) - \psi_{\beta}(\eta)$.

  \indent
  Continuous differentiability of  $\beta \mapsto \eta_{\beta}$ follows from the implicit function theorem applied to $f(\eta, \beta) \coloneqq \eta \psi_{\beta}'(\eta) - \psi_{\beta}(\eta)$, noting that $\de_{\eta}f(\eta,\beta) = \eta \psi_{\beta}^{\prime\prime}(\eta) > 0$.
  This directly implies continuous differentiability of $\beta \mapsto \gamma_{\beta} = \psi_{\beta}(\eta_{\beta})/\eta_{\beta}$

\end{proof}

\begin{figure}[t]
  \centering
  \begin{subfigure}[b]{0.49\linewidth}
    \includegraphics[width=\textwidth]{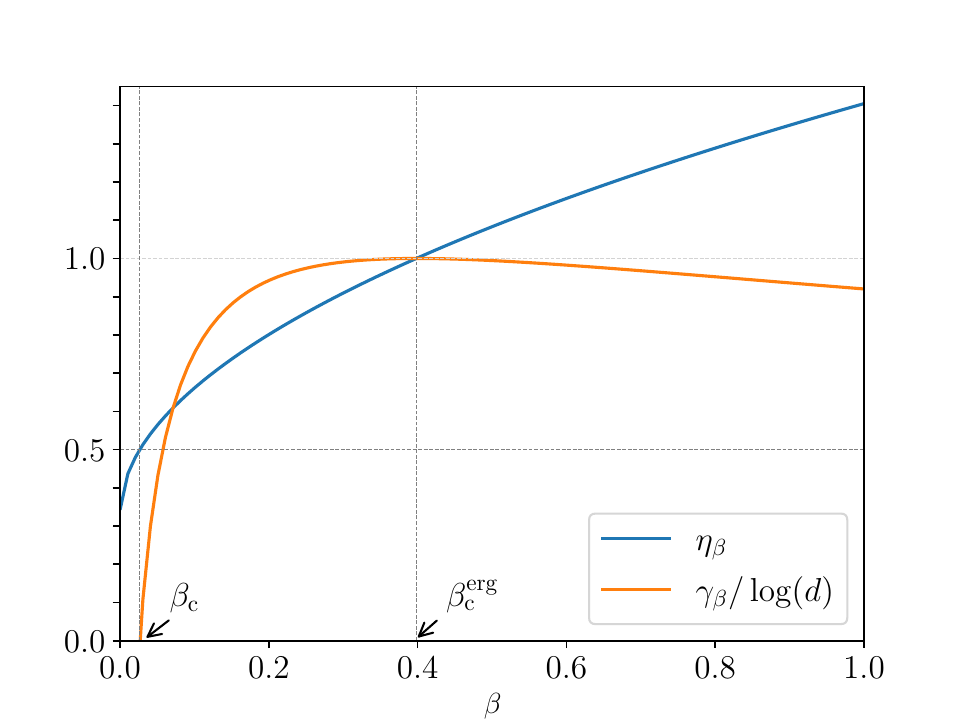}
  \end{subfigure}
  \begin{subfigure}[b]{0.49\linewidth}
    \includegraphics[width=\textwidth]{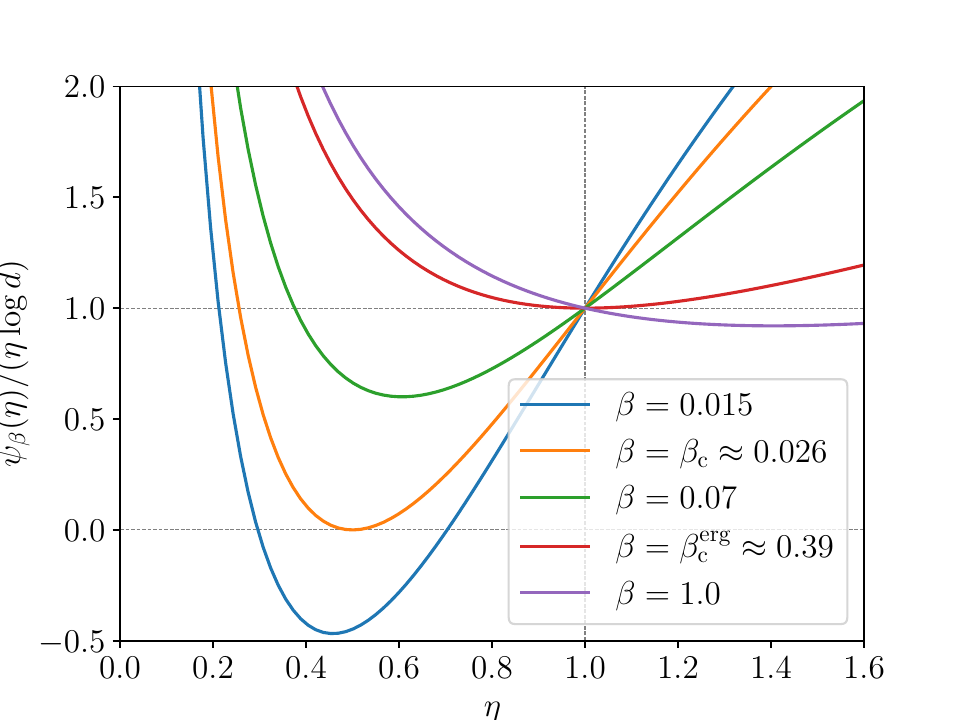}
  \end{subfigure}
  \caption{Illustration of $\eta_{\beta}$, $\gamma_{\beta}/\log d$ and $\psi_{\beta}(\eta)/(\eta \log d)$ for $d=2$. For the figure on the left, note that $\gamma_{\beta}$ is positive for $\beta > \beta_{\mrm{c}}$ and attains its maximum at $\beta_{\mrm{c}}^{\mrm{erg}}$, at the same point at which $\eta_{\beta} = 1$. The right figure illustrates the same point: The minima of $\psi_{\beta}(\eta)/\eta$ move to the right with increasing $\beta$ and attain their highest value at $\beta = \beta_{\mrm{c}}^{\mrm{erg}}$.}
  \label{fig:crit_parameter_regimes}
\end{figure}

Considering the graphs in Figure~\ref{fig:crit_parameter_regimes}, one would conjecture that $\eta_{\beta}$ is strictly increasing in $\beta$.
One can apply the implicit function theorem to $f(\eta, \beta) \coloneqq \eta \psi_{\beta}'(\eta) - \psi_{\beta}(\eta)$ to obtain
\begin{equation}
  \label{eq:45}
  \dv{\eta_{\beta}}{\beta}
  = - \frac{[\de_{\beta}f](\eta_{\beta}, \beta)}{[\de_{\eta}f](\eta_{\beta}, \beta)}
  = \frac{[\de_{\beta}\psi_{\beta}](\eta_{\beta}) - \eta_{\beta} [\de_{\beta} \psi'_{\beta}](\eta_{\beta})}{\eta_{\beta} \psi_{\beta}''(\eta_{\beta})}.
\end{equation}
The denominator is positive by \eqref{eq:43}, but we are not aware how to show non-negativity of the numerator for general $\beta$.
We can however make use of this for the special case $\beta = \beta_{\mrm{c}}$, which will be relevant in Section~\ref{sec:near-crit-conductance}, in order to prove Theorem~\ref{thm:effective-conductance-nearcrit}.

\begin{proposition}\label{prop:near-crit-t-field}
  Let $\psi_{\beta}(\eta)$ and $\eta_{\beta}$ be as in \eqref{eq:162} and \eqref{eq:17}, for some $d\geq 2$.
  For $\beta_{\mrm{c}} = \beta_{\mrm{c}}(d) > 0$, as given in Proposition~\ref{prop:basdevant-singh}, we have $\eta_{\beta_{\mrm{c}}} = 1/2$ and
  \begin{equation}
    \label{eq:69}
    \dv{\beta}\Big\vert_{\beta=\beta_{\mrm{c}}} \eta_{\beta} > 0
    \qq{and}
    \dv{\beta}\Big\vert_{\beta=\beta_{\mrm{c}}} \psi_{\beta}(\eta_{\beta}) > 0
  \end{equation}
\end{proposition}
\begin{proof}
  By \eqref{eq:164} we have $\tfrac12 \psi_{\beta_{\mrm{c}}}^{\prime}(\tfrac12) - \psi_{\beta_{\mrm{c}}}(\tfrac12) = -\psi_{\beta_{\mrm{c}}}(\tfrac12) = 0$.
  Lemma~\ref{lem:crit-exp-behaviour} therefore implies $\eta_{\beta_{\mrm{c}}} = 1/2$.
  Applying \eqref{eq:45} and recalling $\psi_{\beta}'(\tfrac12) = 0$, we get
  \begin{equation}
    \label{eq:50}
    \dv{\eta_{\beta}}{\beta}\Big\vert_{\beta = \beta_{\mrm{c}}}
    =
    \frac{\de_{\beta}\vert_{\beta=\beta_{\mrm{c}}}\psi_{\beta}(\tfrac12)}{\tfrac12 \psi_{\beta}''(\eta_{\beta})}.
  \end{equation}
  The denominator is positive by \eqref{eq:43}.
  As for the numerator, we recall \eqref{eq:162} for $\eta = 1/2$:
  \begin{equation}
    \label{eq:11}
    \begin{aligned}
      \psi_{\beta}(\tfrac12)
      = \log \left(d \int \sqrt{\frac{\beta}{2\pi}}e^{-\beta(\cosh(t)-1)}\dd{t}\right).
    \end{aligned}
  \end{equation}
  To see monotonicity of the integral in $\beta$ it is convenient to apply the change of variables.
  \begin{equation}
    \label{eq:71}
    \begin{aligned}
      u &= e^{t/2} - e^{-t/2} = 2\sinh(t/2) \Longleftrightarrow t = 2 \operatorname{arsinh}(u/2)\\
      \dv{u}{t} &= \frac12 (e^{t/2} + e^{-t/2}) = \sqrt{1 + u^{2}/4}
    \end{aligned}
  \end{equation}
  Note that $u^{2}/2 = \tfrac12(e^{t} + e^{-t}) - 1 = \cosh(t)-1$, hence
  \begin{equation}
    \label{eq:70}
    \begin{aligned}
      \int \sqrt{\frac{\beta}{2\pi}}e^{-\beta(\cosh(t)-1)}\dd{t}
      =&\int \sqrt{\frac{\beta}{2\pi}}e^{-\frac{\beta}{2}u^2 } \frac{2}{\sqrt{u^2+4}}\dd{u} \\
      =& \int \sqrt{\frac{1}{2\pi}}e^{-\frac{1}{2}s^2 } \frac{2}{\sqrt{s^2/\beta+4}}\dd{s}.
    \end{aligned}
  \end{equation}
  Clearly, the integrand in the last line is strictly increasing in $\beta$, hence $\de_{\beta}\psi_{\beta}(\tfrac12)>0$.
  This implies the first statement in \eqref{eq:69}.
  For the second statement note that $\psi_{\beta_{\mrm{c}}}^{\prime}(\tfrac12) = 0$.
  Hence, $\de_{\beta}\vert_{\beta=\beta_{\mrm{c}}}\psi_{\beta}(\eta_{\beta}) = \de_{\beta}\vert_{\beta=\beta_{\mrm{c}}}\psi_{\beta}(\tfrac12) > 0$.
\end{proof}

As already suggested in Figure~\ref{fig:crit_parameter_regimes}, there is a second natural transition point $\beta_{\mrm{c}}^{\mrm{erg}} > \beta_{\mrm{c}}$, which is ``special'' due to $\gamma_{\beta}$ attaining its maximum there.
This transition point will be relevant for the study of the intermediate phase in Section~\ref{sec:intermediate-phase}.

\begin{proposition}[Characterisation of $\beta_{\mrm{c}}^{\mrm{erg}}$]
  \label{prop:beta-erg-characterisation}
  Let $\psi_{\beta}(\eta)$ and $\eta_{\beta}$ be as in \eqref{eq:162} and \eqref{eq:17}, for some $d\geq 2$.
  The map $\beta \mapsto \psi_{\beta}^{\prime}(1) - \psi_{\beta}(1)$ is strictly decreasing and there exists a unique $\beta_{\mrm{c}}^{\mrm{erg}} = \beta_{\mrm{c}}^{\mrm{erg}}(d) > 0$, such that
  \begin{equation}
    \label{eq:51}
    \psi_{\beta_{\mrm{c}}^{\mrm{erg}}}(1) = \psi^{\prime}_{\beta_{\mrm{c}}^{\mrm{erg}}}(1).
  \end{equation}
  Equivalently, $\beta_{\mrm{c}}^{\mrm{erg}} > 0$ is characterised by any of the following conditions:
  \begin{equation}
    \label{eq:42}
    \EE_{\beta_{\mrm{c}}^{\mrm{erg}}}[T] = - \log d
    \quad\Longleftrightarrow\quad
    \eta_{\beta_{\mrm{c}}^{\mrm{erg}}} = 1
    \quad\Longleftrightarrow\quad
    \gamma_{\beta_{\mrm{c}}^{\mrm{erg}}} = \sup_{\beta > 0} \gamma_{\beta} = \log d.
  \end{equation}
  Moreover, for $\beta < \beta_{\mrm{c}}^{\mrm{erg}}$ we have that $\eta_{\beta} < 1$ and that $\beta \mapsto \gamma_{\beta}$ is increasing, while for $\beta > \beta_{\mrm{c}}^{\mrm{erg}}$ one has $\eta_{\beta} > 1$ and $\beta \mapsto \gamma_{\beta}$ is decreasing.
\end{proposition}

\begin{proof}
  By definition of $\psi_{\beta}$ and the $t$-field increment measure we have
  \begin{equation}
    \label{eq:52}
    \psi_{\beta}^{\prime}(1) - \psi_{\beta}(1)
    = \EE_{\beta}[T e^{T}] - \log d
    = -\EE_{\beta}[T] - \log d.
  \end{equation}
  We claim that $\beta \mapsto \EE_{\beta}[T]$ is strictly increasing.
  In fact, using the change of variables in \eqref{eq:71} and noting that $e^{-t/2} = \cosh(t/2) - \sinh(t/2) = \sqrt{1+(u/2)^{2}} - u/2$, we have
  \begin{equation}
    \label{eq:76}
    \begin{aligned}
      \EE_{\beta}[T]
      =& \int \sqrt{\frac{\beta}{2\pi}}e^{-\beta(\cosh(t)-1)}e^{-t/2}t\dd{t}\\
      =& \int \sqrt{\frac{\beta}{2\pi}}e^{-\frac{\beta}{2}u^2 } \,\frac{2\operatorname{arsinh}(u/2)(\sqrt{1+(u/2)^{2}} - u/2)}{\sqrt{1+(u/2)^{2}}}\dd{u}\\
      =& -2 \int \sqrt{\frac{\beta}{2\pi}}e^{-\frac{\beta}{2}u^2 }
         \,\frac{u}{2}\frac{\operatorname{arsinh}(u/2)}{\sqrt{1+(u/2)^{2}}}\dd{u}.
    \end{aligned}
  \end{equation}
  It is easy to check that $x \operatorname{arsinh}(x)/\sqrt{1+x^{2}}$ is strictly increasing in $\abs{x}$.
  Consequently, rescaling $u = s/\sqrt{\beta}$ as in \eqref{eq:70}, we see that above integral is strictly increasing in $\beta$.
  Moreover, one also observes that that $\EE_{\beta}[T] \to -\infty$ for $\beta \searrow 0$, whereas $\EE_{\beta}[T] \to 0$ for $\beta \to\infty$.
  Hence by \eqref{eq:52}, there exists a unique $\beta_{\mrm{c}}^{\mrm{erg}} > 0$, such that $\psi_{\beta_{\mrm{c}}^{\mrm{erg}}}^{\prime}(1) = \psi_{\beta_{\mrm{c}}^{\mrm{erg}}}(1)$.
  In particular, $\eta_{\beta_{\mrm{c}}^{\mrm{erg}}} = 1$.

  \indent
  The first two alternative characterisations in \eqref{eq:42} follow from \eqref{eq:52} and our previous considerations.
  Also, by Theorem~\ref{thm:monotonicity-t-field}, we have
  \begin{equation}
    \label{eq:53}
    \psi_{\beta}(1) \lessgtr \psi_{\beta}^{\prime}(1) \qq{for} \beta \lessgtr \beta_{\mrm{c}}^{\mrm{erg}},
  \end{equation}
  which by Lemma~\ref{lem:crit-exp-behaviour} implies that $\eta_{\beta} \lessgtr 1$ for $\beta \lessgtr \beta_{\mrm{c}}^{\mrm{erg}}$.

  \indent
  To show the last characterisation in \eqref{eq:42}, we calculate the derivative of $\beta \mapsto \gamma_{\beta} = \psi_{\beta}(\eta_{\beta})/\eta_{\beta}$:
  \begin{equation}
    \label{eq:23}
    \begin{aligned}
      \de_{\beta}\gamma_{\beta}
      &= \de_{\beta}[\frac{\psi_{\beta}(\eta_{\beta})}{\eta_{\beta}}]\\
      &= \tfrac{1}{\eta_{\beta}}[\de_{\beta}\psi_{\beta}](\eta_{\beta}) + \tfrac{1}{\eta_{\beta}} [\de_{\beta}\eta_{\beta}] \psi_{\beta}^{\prime}(\eta_{\beta}) - \tfrac{1}{\eta_{\beta}^{2}} [\de_{\beta}\eta_{\beta}] \psi_{\beta}(\eta_{\beta})\\
      &= \tfrac{1}{\eta_{\beta}}[\de_{\beta}\psi_{\beta}](\eta_{\beta}),
    \end{aligned}
  \end{equation}
  where in the last line we used that $\eta_{\beta}\psi^{\prime}_{\beta}(\eta_{\beta}) - \psi_{\beta}(\eta_{\beta}) = 0$.
  By Theorem~\ref{thm:monotonicity-t-field}, the last line in \eqref{eq:23} is non-negative if $\eta_{\beta} \leq 1$ and non-positive for $\eta_{\beta} \geq 1$.
  Since $\eta_{\beta} \lessgtr 1$ for $\beta \lessgtr \beta_{\mrm{c}}^{\mrm{erg}}$ this implies the last statement in \eqref{eq:42} as well as the stated monotonicity behaviour of $\beta \mapsto \gamma_{\beta}$.
\end{proof}

\subsection{Effective Conductance in a Critical Environment (Proof of Theorem~\ref{thm:critical-eff-cond})}
First we recall some results on small deviation of random walks.
To be precise, we use an extension of Mogulskii's Lemma \cite{mogulskii_small_1975}, due to Gantert, Hu and Shi \cite{gantert_asymptotics_2011}.

\begin{lemma}[Triangular Mogulskii's Lemma {\cite[Lemma 2.1]{gantert_asymptotics_2011}}]\label{lem:mogulski}
  For each $n\geq 1$, let $X_i^{(n)}$, $1\leq i \leq n$, be i.i.d.\ real-valued random variables. Let $g_1<g_2$ be continuous functions on $[0,1]$ with $g_1(0)<0<g_2(0)$. Let $(a_n)$ be a sequence of positive numbers such that $a_n \to \infty$ and $a^{2}_n/n \to 0$ as $n \to \infty$. Assume that there exist constants $\eta>0$ and $\sigma^2>0$ such that:
  \begin{equation}
    \label{eq:170}
    \sup_{n\geq 1} \EE\left[|X_1^{(n)}|^{2+\eta}\right]<\infty,\qquad  \EE\left[X_1^{(n)}\right]=o\bigg(\frac{a_n}{n}\bigg), \qquad \text{Var}\big[X_1^{(n)}\big]\rightarrow \sigma^2.
  \end{equation}
  Consider the measurable event
  \begin{equation}
    \label{eq:171}
    E_n:=\bigg\{ g_1\left(\frac{i}{n}\right)\leq \frac{S_i^{(n)}}{a_n}\leq g_2\left(\frac{i}{n}\right)\; \forall i\in [n]\bigg\},
  \end{equation}
  where $S_i^{(n)}:= X_1^{(n)}+\cdots+X_i^{(n)},\ 1\leq i \leq n$. We have
  \begin{equation}
    \label{eq:172}
    \frac{a_n^2}{n}\log\left(\PP[E_n]\right)\xrightarrow[n\rightarrow\infty]{} -\frac{\pi^2\sigma^2}{2} \int_0^1 \frac{1}{(g_2(t)-g_1(t))^2}\dd{t}.
  \end{equation}
\end{lemma}

\begin{lemma}\label{lem:small-deviation-uniform}
  For each $k\geq 1$, let $X_{i}^{(k)}$, $i\in \NN$, be i.i.d.\ real-valued random variables with $\EE[X_{i}^{(k)}] = 0$ and $\sigma_{k}^{2} \coloneqq \EE[(X_{i}^{(k)})^{2}]$.
  Suppose that $0 < \inf_{k}\sigma_{k}^{2} \leq \sup_{k}\sigma_{k}^{2} < \infty$.
  Write $S_{i}^{k} = X_{1}^{(k)} + \cdots + X_{i}^{(k)}$.
  For $\gamma > 0$ and $\nu \in (0,\tfrac12)$, define the events
  \begin{equation}
    \label{eq:212}
    E^{(k)}_{n} \coloneqq \{ |S_{i}| \leq \gamma n^{\nu},\; \forall i\in [n]\}.
  \end{equation}
  then we have
  \begin{equation}
    \label{eq:213}
    \lim_{n\to\infty} \sup_{k\in\NN} \abs{n^{1-2\nu}\log\PP[E^{(k)}_{n}] + \Big(\frac{\pi \sigma_{k}}{2\gamma}\Big)^{2}} = 0.
  \end{equation}
\end{lemma}

\begin{proof}
  We proceed by contradiction.
  Write $b_{n}^{(k)} \coloneqq -n^{1-2\nu}\log\PP[E^{(k)}_{n}]$ and $b_{\infty}^{(k)} \coloneqq \big(\frac{\pi \sigma_{k}}{2\gamma}\big)^{2}$ and suppose \eqref{eq:213} does not hold.
  Then there exists $\epsilon > 0$, $(k_{n})_{n\in \NN}$, and a subsequence $\mcl{N}_{0} \subseteq \NN$
  \begin{equation}
    \label{eq:214}
    \forall n\in \mcl{N}_{0}\colon \abs{b_{n}^{(k_{n})} - b_{\infty}^{(k_{n})}} > \epsilon.
  \end{equation}
  Since the $\sigma_{k}^{2}$ are bounded, we can refine to a subsequence $\mcl{N}_{1} \subseteq \mcl{N}_{0} \subseteq \NN$, such that $\sigma^{2}_{k_{n}} \to \tilde{\sigma} > 0$ along $\mcl{N}_{1}$.
  But by Lemma~\ref{lem:mogulski} (with $a_{n} = n^{\nu}$, $g_{1} = -\gamma$, and $g_{2} = +\gamma$) we have $b_{n}^{(k_{n})} \to - \big(\frac{\pi\tilde{\sigma}}{2\gamma}\big)^{2}$ along $\mcl{N}_{1}$, in contradiction with \eqref{eq:214}.
\end{proof}

\begin{proof}[Proof of Theorem~\ref{thm:critical-eff-cond}]
  Recall the notation in Theorem~\ref{thm:critical-eff-cond}.
  We proceed by proving the statement for an individual increment $V$, but indicate at which steps care has to be taken to establish the uniformity \eqref{eq:218}.

  \indent
  Write $\de\Lambda_{n} \coloneqq \{x \in \mbb{T}_{d} \colon \abs{x} = n\}$ for the vertices at distance $n$ from the origin.
  Set $\alpha:=\frac{1}{2}(\pi^{2}\sigma_{V}^{2})^{1/3}$.
  Define the \emph{stopping lines} of $\{V_{x}\}_{x\in\mbb{T}_{d}}$ at level $\alpha n^{1/3}$:
  \begin{equation}
    \label{eq:173}
    \mcl{L}^{(n)} \coloneqq \{(x,y) \in \vec{E}\colon V_{y}\geq \alpha n^{1/3},\;  \forall z \prec y : V_{z} < \alpha n^{1/3}\},
  \end{equation}
  where we write $\vec{E}$ for the set or edges oriented away from the origin and ``$a \prec b$'' means that $a$ is an ancestor of $b$.
  Let $A_{n}$ denote the event that $\mcl{L}^{(n)}$ is a cut-set between the origin and $\de\Lambda_{n}$.
  By \eqref{eq:153}, conditionally on the event $A_{n}$ we have the point-wise bound
  \begin{equation}
    \label{eq:174}
    C^{\mrm{eff}}_{n,\gamma}
    \leq \sum_{xy \in \mcl{L}^{(n)}}e^{-\gamma(V_{x} + V_{y})}.
  \end{equation}
  We thus have:
  \begin{equation}
    \label{eq:189}
    \EE\Big[C^{\mrm{eff}}_{n,\gamma}\Big]\leq \EE\Big[\sum_{xy \in \mcl{L}^{(n)}}e^{-\gamma(V_{x} + V_{y})}\Big]+\EE\Big[C^{\mrm{eff}}_{n,\gamma}\unit_{A_{n}^{\mrm{c}}}\Big]
  \end{equation}
  \emph{Bounding the second summand.}
  Clearly, we have
  \begin{equation}
    \label{eq:220}
    \begin{aligned}
      \PP[A_{n}^{\mrm{c}}]
      &\leq
      \PP[\exists |x|=n,\text{ such that } \forall y\prec x, |V_{y}|\leq \alpha n^{1/3}]\\
      &\hspace*{3em}+
        \PP[\exists |x|\leq n,\text{ such that } V_{x} \leq -\alpha n^{1/3}].
    \end{aligned}
  \end{equation}
  To bound the first summand on the right hand side, we apply the many-to-one formula (Proposition~\ref{prop:many-to-one}) with $\eta = 1$, and get a
  random walk $(S_{i})_{i\geq 0}$, such that
  \begin{equation}
    \begin{aligned}
      &\hspace*{-8em}
        \PP[\exists |x|=n,\text{ such that } \forall y\prec x, |V_{y}|\leq \alpha n^{1/3}]\\
      &\leq \EE\Big[{\textstyle \sum_{|x|=n}} \unit \{\forall y\prec x, |V_{y}|\leq \alpha n^{1/3}\}\Big]\\
      &= \EE[e^{S_n}\unit_{\forall i\in [n], |S_i|\leq \alpha n^{1/3}}]\\
      &\leq e^{\alpha n^{1/3}} \PP[\forall i\in [n], |S_i|\leq \alpha n^{1/3}].
    \end{aligned}
  \end{equation}
  In the third line we used that $\psi_{V}(1) = 0$.
  We recall that (since $\psi(1)_{V} = \psi_{V}'(1) = 0$) we have $\EE[S_{1}] = 0$ and $\EE[S_{1}^{2}] = \sigma_{V}^{2}$.
  Applying Lemma~\ref{lem:small-deviation-uniform} (with $\gamma = \alpha$ and $\nu = 1/3$) yields
  \begin{equation}
    \label{eq:217}
    \PP[\forall i\in [n], |S_i|\leq \alpha n^{1/3}] = e^{-[2\alpha+o(1)] n^{1/3}},
  \end{equation}
  where we used that $(\tfrac{\pi\sigma_{V}}{2\alpha})^{2} = 2 \alpha$.
  Moreover, Lemma~\ref{lem:small-deviation-uniform} states that the convergence in \eqref{eq:217} is uniform over a family $V^{(k)}$, $k\in \NN$, of critical increments given that $0 < \inf_{k}\sigma^{2}_{V^{(k)}} \leq \sup_{k}\sigma^{2}_{V^{(k)}} < \infty$.
  In conclusion we have
  \begin{equation}
    \label{eq:219}
    \PP[\exists |x|=n,\text{ such that } \forall y\prec x, |V_{y}|\leq \alpha n^{1/3}]
    \leq e^{-[\alpha + o(1)]n^{1/3}}.
  \end{equation}
  Fot the second summand in \eqref{eq:220} we have
  \begin{equation}\label{eq:221}
    \begin{aligned}
      \PP[\exists |x|\leq n,\text{ such that } V_{x} \leq -\alpha n^{1/3}]
      \leq & \sum_{i=1}^n\EE\Big[\sum_{|x|=i} \unit_{V_{x}\leq -\alpha n^{1/3}} \Big]\\
      = & \sum_{i=1}^n \sum_{|x|=i} \EE[e^{-V_{x}} \,e^{V_{x}}\unit_{V_{x}\leq -\alpha n^{1/3}}]\\
      \leq & \sum_{i=1}^n \sum_{|x|=i} \EE[e^{-V_{x}}]e^{-\alpha n^{1/3}}\\
      = & \sum_{i=1}^n e^{i \psi_{V}(1)} e^{-\alpha n^{1/3}}\\
      = & \sum_{i=1}^n e^{-\alpha n^{1/3}}\\
      =& n e^{-\alpha n^{1/3}}.
    \end{aligned}
  \end{equation}
  Where we used that $e^{i \psi_{V}(\eta)} = \sum_{\abs{x} = i} \EE[e^{-\eta V_{x}}]$, which one may check inductively.
  In conclusion, \eqref{eq:220}, \eqref{eq:219} and \eqref{eq:221} yield $\PP(A_{n}^{\mrm{c}}) \leq e^{-(\alpha+o(1)) n^{1/3}}$.
  We proceed by controlling the second summand in \eqref{eq:189} using Cauchy-Schwarz and properties of the effective conductance (Lemma~\ref{lem:eff-cond-properties}):
  \begin{equation}
    \label{eq:175}
    \begin{aligned}
      \EE[C^{\mrm{eff}}_{n,\gamma} \unit_{A_{n}^{\mrm{c}}}]
      \leq \sqrt{\EE[(C_{n,\gamma}^{\mrm{eff}})^{2}]}\; e^{-\frac{\alpha}{2}\,[n^{1/3} + o(1)]} 
    \end{aligned}
  \end{equation}
  To bound the first factor on the right hand side note that $C_{n,\gamma}^{\mrm{eff}} \leq {\textstyle\sum_{\abs{x}=1}}e^{-\gamma V_{x}}$ by Lemma~\ref{lem:eff-cond-properties}.
  By Jensen's and Hölder's inequality
  \begin{equation}
    \label{eq:223}
    \begin{aligned}
      \EE[({\textstyle\sum_{\abs{x}=1}}e^{-\gamma V_{x}})^{2}]
      &\leq d\, \EE[{\textstyle\sum_{\abs{x}=1}}e^{-2\gamma V_{x}}]\\
      &= d^{2}\, \EE[e^{-2\gamma V}]\\
      &\leq d^{2} \EE[e^{-V}]^{2\gamma (1-\frac{2\gamma -1}{2a})} \EE[e^{-(1+2a)V}]^{\frac{2\gamma}{1+2a}\frac{2\gamma -1}{2a}}\\
      &\leq d^{2 - 2\gamma (1-\frac{2\gamma -1}{2a})} [\tfrac1{d} e^{\psi_{V}(1+2a)}]^{\frac{2\gamma}{1+2a}\frac{2\gamma -1}{2a}},
    \end{aligned}
  \end{equation}
  where we used $1 = e^{\psi_{V}(1)} = d\, \EE[e^{-V}]$.
  The last line in \eqref{eq:223} is continuous in $\gamma \in \RR$, hence uniformly bounded for $\gamma \in (1/2, 1/2 + a)$.
  In conclusion, we have
  \begin{equation}
    \label{eq:225}
    \sup_{1/2 < \gamma < 1/2 + a} \EE[C^{\mrm{eff}}_{n,\gamma} \unit_{A_{n}^{\mrm{c}}}]
    \leq C(\psi_{V}(1+2a))\, e^{-[\frac{\alpha}{2} + o(1)] n^{1/3}},
  \end{equation}
  for a constant $C(\psi_{V}(1+2a)) > 0$ depending continuously on $\psi_{V}(1+2a)$.
  In particular, this yields a uniform bound over a family of critical increments $V^{(k)}$ with $0 < \inf_{k}\sigma^{2}_{V^{(k)}} \leq \sup_{k}\sigma^{2}_{V^{(k)}} < \infty$ and $\sup_{k}\psi_{V^{(k)}}(1+2a) \infty$.

  \medskip
  \emph{Bounding the first summand.}
  For a vertex $x \in \de\Lambda_{n}$ we write $(x_{k})_{k=0,\ldots,n}$ for its sequence of predecessors ($x_{0} = 0, x_{n} = x$).
  For a walk $X = (X_{i})_{i\geq 0}$, analogously to our stopping lines, we introduce the stopping time at level $\alpha n^{1/3}$:
  \begin{equation}
    \label{eq:197}
    T^{(n)}_{X} = \inf\{i\geq 0\colon X_{i} \geq \alpha n^{1/3}\}
  \end{equation}
  Note that on the event $A_{n}$, we know for every $x\in\de\Lambda_{n}$ that the sequence $(V_{x_{i}})_{i=0,\ldots,n}$ crosses level $\alpha n^{1/3}$.
  In other words, $T^{(n)}_{(V_{x_{i}})} \leq n$.

  Consequently, the first summand in \eqref{eq:189} is bounded via
  \begin{equation}
    \label{eq:198}
    \begin{aligned}
      \EE \Big[ \sum_{xy \in \mcl{L}^{(n)}}e^{-\gamma(V_{x} + V_{y})} \Big]
      &\leq \sum_{k=1}^{n} \EE \Big[ \sum_{\abs{x} = k} \unit\{T^{(n)}_{(V_{x_{i}})} = k\} e^{-\gamma(V_{x_{k-1}} + V_{x_{k}})} \Big].
    \end{aligned}
  \end{equation}
  The last line is amenable to the many-to-one formula (Theorem~\ref{prop:many-to-one}).
  Write $(S_{i})_{i\geq 0}$ for the associated random walk (choosing $\eta = 1$), then the last line in \eqref{eq:198} is equal to
  \begin{equation}
    \label{eq:132}
    \begin{aligned}
      \sum_{k=1}^{n} \EE \Big[ \unit\{T^{(n)}_{S} = k\} e^{S_{k}} e^{-\gamma(S_{k-1} + S_{k})} \Big]
      =
      \sum_{k=1}^{n} \EE \Big[ \unit\{T^{(n)}_{S} = k\} e^{-(2\gamma-1)S_{k-1}} e^{(1-\gamma) (S_{k} - S_{k-1})} \Big].
    \end{aligned}
  \end{equation}
  Now, since $S_{k}\geq \alpha n^{1/3}$ for $T^{(n)}_{S} = k$, and since $\gamma > 1/2$ by assumption, we can bound the right hand side and obtain
  \begin{equation}
    \label{eq:133}
    \begin{aligned}
      \EE \Big[ \sum_{xy \in \mcl{L}^{(n)}}e^{-\gamma(V_{x} + V_{y})} \unit_{A_{n}} \Big]
      &\leq e^{-(2\gamma - 1)\alpha n^{1/3}} \times \sum_{k=1}^{n} \EE \Big[ \unit\{T^{(n)}_{S} = k\} e^{(1-\gamma)(S_{k} - S_{k-1})} \Big]\\
      &\leq e^{-(2\gamma - 1)\alpha n^{1/3}} \times n\EE \Big[  e^{(1-\gamma) S_{1}} \Big]\\
    \end{aligned}
  \end{equation}

  Now by using the definition of $(S_{i})$ in \eqref{eq:158} we have
  \begin{equation}
    \label{eq:226}
    \EE[e^{(1-\gamma)S_{1}}]
    = d\,\EE[e^{-\gamma V}]
    \leq d\, \EE[e^{-(1+2a)V}]^{\frac{\gamma}{1+2a}}
    \leq d\, [\tfrac1{d} e^{\psi_{V}(1+2a)}]^{\frac{\gamma}{1+2a}}
    \leq C(\psi_{V}(1+2a)),
  \end{equation}
  for a constant $C(\psi_{V}(1+2a)) > 0$ that is independent of $\gamma \in (1/2, 1/2 + a)$ and continuous with respect to $\psi_{V}(1+2a)$.
  Hence,  
  \begin{equation}\label{eq:227}
    \EE \Big[ \sum_{xy \in \mcl{L}^{(n)}}e^{-\gamma(V_{x} + V_{y})} \Big]
     \leq e^{-[(2\gamma - 1)\alpha +o(1)] n^{1/3}},
  \end{equation}
  and this bound holds uniformly with respect to $\gamma \in (1/2, 1/2 + a)$ and over family of critical increments $V^{(k)}$, given that $\sup_{k} \psi_{V^{(k)}}(1+2a) < \infty$.
  In conclusion \eqref{eq:174}, \eqref{eq:225} and \eqref{eq:227} yield
  \begin{equation}
    \label{eq:228}
    \begin{aligned}
      \EE[C_{n,\gamma}^{\mrm{eff}}]
      &\leq e^{-[\alpha/2 + o(1)]n^{1/3}} + e^{-[(2\gamma - 1)\alpha + o(1)]n^{1/3}}\\
      &\leq e^{-[\min(\tfrac12,2\gamma - 1) \alpha + o(1)]n^{1/3}}\\
      &= e^{-[ \min(\tfrac14,\gamma - \tfrac12) (\pi^{2}\sigma_{V}^{2})^{1/3} +o(1)]n^{1/3}}
    \end{aligned}
  \end{equation}
  uniformly over $\gamma \in (1/2, 1/2 + a)$ as $n\to\infty$.
  And as noted, this bound is also uniform over a family of critical increments $V^{(k)}$, given the assumptions in the theorem.
  This concludes the proof.
\end{proof}

\subsection{Near-Critical Effective Conductance (Proof of Theorem~\ref{thm:effective-conductance-nearcrit})}
\label{sec:near-crit-conductance}

The upper bound in Theorem~\ref{thm:effective-conductance-nearcrit} will follow from Theorem~\ref{thm:critical-eff-cond} and a perturbative argument.
For the lower bound, we will apply a modification of a result due to Gantert, Hu and Shi \cite{gantert_asymptotics_2011}.
In their work they give the asymptotics for the probability that some trajectory of a critical branching random walk stays below a slope $\delta |i|$ when $\delta \searrow 0$.
We are interested in this result applied to the critical rescaling of $t$-field $\{\tau_{x}^{\beta}\}_{x\in\mbb{T}_{d}}$ as given in \eqref{eq:180}.
Comparing to Gantert, Hu and Shi's result, we will require additional uniformity in $\beta$:

\begin{theorem}\label{thm:uniform-ghs}
  Let $\{\tau_{x}^{\beta}\}_{x\in\mbb{T}_{d}}$ be as in \eqref{eq:180}.
  For any $a>0$ small enough, there exists a constant $C>0$ such that for all $\beta\in[\beta_c,\beta_c+a]$, for $\delta$ small enough:
  \[
    \PP_{\beta}[\exists \text{a path } \gamma\colon0\rightarrow\infty \text{ s.t.\ } \forall i\in\NN,\ \tau^{\beta}_{\gamma_i}\leq \delta i]\geq e^{-C/\sqrt{\delta}}.
  \]
\end{theorem}

This theorem will be proven in Appendix~\ref{sec:uniform-ghs}, as it closely follows the arguments of Gantert, Hu and Shi, while taking some extra care to ensure the required uniformity.

\begin{proof}[Proof of Theorem~\ref{thm:effective-conductance-nearcrit}]
  The main idea is to consider, for $\beta = \beta_{\mrm{c}} + \epsilon$, the critical rescaling of the $t$-field (see Lemma~\ref{lem:critical-rescaling}, \eqref{eq:17} and Lemma~\ref{lem:crit-exp-behaviour})
  \begin{equation}
    \label{eq:61}
    \tau^{\beta}_{i} = - \eta_{\beta} T_{i} + \psi_{\beta}(\eta_{\beta}) \abs{i}.
  \end{equation}
  We remind the reader of the definition of the rescaled field
  with the following near-critical behaviour for the constants (Proposition~\ref{prop:near-crit-t-field}):
  \begin{equation}
    \label{eq:62}
    \begin{aligned}
      \eta_{\beta_{\mrm{c}} + \epsilon}
      &= \tfrac12 + c_{\eta} \epsilon + O(\epsilon^{2})\qq{with} c_{\eta} > 0\\
      \psi_{\beta_{\mrm{c}} + \epsilon}(\eta_{\beta_{\mrm{c}} + \epsilon})
      &= c_{\psi}\epsilon + O(\epsilon^{2}) \qq{with} c_{\psi} > 0.
    \end{aligned}
  \end{equation}
  Together with these asymptotics, application Theorem~\ref{thm:uniform-ghs} and Theorem~\ref{thm:critical-eff-cond} to $\{\tau^{\beta}_{i}\}_{i\in\mbb{T}_{d}}$, will yield the lower and upper bound, respectively.

  \medskip
  \emph{Lower Bound:}
  According to Theorem~\ref{thm:uniform-ghs} we have that there exist constants $a, C>0$, such that for all sufficiently small $\delta>0$:
  \begin{equation}
    \label{eq:192}
    \inf_{\beta_{\mrm{c}} < \beta < \beta_{\mrm{c}} + a}\PP_{\beta}[\exists \text{a path } \gamma\colon0\rightarrow\infty \text{ s.t.\ } \forall i\in\NN,\ \tau^{\beta}_{\gamma_i}\leq \delta i]\geq e^{-C/\sqrt{\delta}}.
  \end{equation}
  Note that $\tau_{\gamma_{i}} \leq \delta i$ is equivalent to $T_{\gamma_{i}} \geq \eta_{\beta}^{-1} [\psi_{\beta}(\eta_{\beta}) - \delta] i$.
  Choosing $\delta(\epsilon) = \tfrac12 c_{\psi} \epsilon$, we have $\eta_{\beta_{\mrm{c}}+\epsilon}^{-1} [\psi_{\beta_{\mrm{c}}+\epsilon}(\eta_{\beta_{\mrm{c}}+\epsilon}) - \delta(\epsilon)] = c_{\psi}\epsilon + O(\epsilon^{2})$.
  Hence, for $\epsilon>0$ small enough
  \begin{equation}
    \label{eq:193}
    \PP_{\beta_{\mrm{c}} + \epsilon}[\exists \text{a path } \gamma\colon0\rightarrow\infty \text{ s.t.\ } \forall i\in\NN,\ T_{\gamma_i}\geq \tfrac12 c_{\psi}\epsilon i]  \geq e^{-C/\sqrt{\epsilon}}.
  \end{equation}
  Write $A_{\epsilon}$ for the event in brackets.
  Conditionally on this event, we can bound $C_{\infty}^{\mrm{eff}}$ from below by the conductance along the path $\gamma$ (which is given by Kirchhoff's rule for conductors in series):
  \begin{equation}
    \label{eq:194}
    \text{On } A_{\epsilon}\colon\quad
    C_{\infty}^{\mrm{eff}}
    \geq \Big[\sum_{i=0}^{\infty} \frac{1}{\beta} e^{-2\tfrac{1}{2}c_{\psi}\epsilon\, i}\Big]^{-1}
    = \beta (1 - e^{- c_{\psi}\epsilon}).
  \end{equation}
  Consequently, \eqref{eq:193} and \eqref{eq:194} yield
  \begin{equation}
    \label{eq:196}
    \EE_{\beta_{\mrm{c}}+\epsilon}[C_{\infty}^{\mrm{eff}}] \geq (\beta_{\mrm{c}} + \epsilon) (1 - e^{-c_{\psi}\epsilon}) e^{-C/\sqrt{\epsilon}} = e^{-[C + o(1)]/\sqrt{\epsilon}} \text{ as } \epsilon \rightarrow 0.
  \end{equation}
  This concludes the proof of the lower bound in \eqref{eq:9}.\\

  \emph{Upper Bound:} Recalling the definition \eqref{eq:61}, we have for any $i,j \in \mbb{T}_{d,n} \subseteq \mbb{T}_{d}$ that
  \begin{equation}
    \label{eq:229}
    e^{T_{i} + T_{j}}
    = e^{(\abs{i} + \abs{j})\, \psi_{\beta}(\eta_{\beta})/\eta_{\beta}} e^{-\eta^{-1}_{\beta}(\tau^{\beta}_{i} + \tau^{\beta}_{j})} 
    \leq e^{2n\, \psi_{\beta}(\eta_{\beta})/\eta_{\beta}} e^{-\eta^{-1}_{\beta}(\tau^{\beta}_{i} + \tau^{\beta}_{j})}.
  \end{equation}
  Hence, if we write $\tilde{C}^{\mrm{eff}}_{n}$ for the effective conductance between the origin and $\de\Lambda_{n} = \{x\in\mbb{T}_{d}\colon \abs{x} = n\}$ in the electrical network with conductances $\{e^{-\eta^{-1}_{\beta}(\tau^{\beta}_{i} + \tau^{\beta}_{j})}\}_{ij \in E}$,
  we have
  \begin{equation}
    \label{eq:230}
    \EE_{\beta}[C_{n}^{\mrm{eff}}] \leq e^{2n\, \psi_{\beta}(\eta_{\beta})/\eta_{\beta}}\, \EE_{\beta}[\tilde{C}^{\mrm{eff}}_{n}].
  \end{equation}
  For any $\beta > 0$, the field $\tau^{\beta}_{i}$ is the BRW for the critical increment $\tau^{\beta} \coloneqq -\eta_{\beta} T + \psi_{\beta}(\eta_{\beta})$, with $T$ is distributed as a $t$-field increment (at inverse temperature $\beta$).
  Hence, Theorem~\ref{thm:critical-eff-cond} implies
  \begin{equation}
    \label{eq:63}
    \EE_{\beta}[\tilde{C}^{\mrm{eff}}_{n}] \leq \exp[-\big[\min(\tfrac14,\eta^{-1}_{\beta} - 1/2)\, (\pi^{2}\sigma_{\tau^{\beta}}^{2})^{1/3} +o(1)\big]n^{1/3}] \qq{as} n\to \infty,
  \end{equation}
  and moreover this holds uniformly as $\beta \searrow \beta_{\mrm{c}}$.
  Note that by \eqref{eq:62} we have $\min(\tfrac14,\eta^{-1}_{\beta} - 1/2) = \tfrac14$ for $\beta$ sufficiently close to $\beta_{\mrm{c}}$.
  In the following write $\beta = \beta_{\mrm{c}} + \epsilon$.
  By \eqref{eq:62} we have $\psi_{\beta_{\mrm{c} + \epsilon}}(\eta_{\beta_{\mrm{c}} + \epsilon})/\eta_{\beta_{\mrm{c}} + \epsilon} \sim 2 c_{\psi} \epsilon$ as $\epsilon \searrow 0$.
  Hence, choosing $n = n(\epsilon) =  c' \epsilon^{-3/2}$ we have
  \begin{equation}
    \label{eq:231}
    2n(\epsilon)\, \psi_{\beta_{\mrm{c} + \epsilon}}(\eta_{\beta_{\mrm{c}} + \epsilon})/\eta_{\beta_{\mrm{c}} + \epsilon} \sim 4 c_{\psi} c' \epsilon^{-1/2} \qq{and} n(\epsilon)^{1/3} = c'^{1/3} \epsilon^{-1/2},
  \end{equation}
  consequently for $c' > 0$ sufficiently small, \eqref{eq:230} and \eqref{eq:63} together with Lemma~\ref{lem:eff-cond-properties} yield
  \begin{equation}
    \label{eq:191}\
    \EE_{\beta_{\mrm{c}} + \epsilon}[C^{\mrm{eff}}_{\infty}] \leq \EE_{\beta_{\mrm{c}} + \epsilon}[C^{\mrm{eff}}_{n(\epsilon)}] \leq e^{-(C+o(1))\,\epsilon^{-1/2}} \qq{as} \epsilon\searrow 0,
  \end{equation}
  for some constant $C > 0$.
\end{proof}

A corollary of the proof above, in particular \eqref{eq:193}, \eqref{eq:194} is the following
\begin{lemma}\label{lem:near-crit-large-conductance}
  In the setting of Theorem~\ref{thm:effective-conductance-nearcrit} one has, for some constants $c, C > 0$
  \begin{equation}
    \label{eq:204}
    \PP_{\beta_{\mrm{c}} + \epsilon} [C_{\infty}^{\mrm{eff}} > c \epsilon] \geq \exp[-(C + o(1))/\sqrt{\epsilon}],
  \end{equation}
  as $\epsilon \searrow 0$.
\end{lemma}

\subsection{Average Escape Time of the VRJP as $\beta \searrow \beta_{\mrm{c}}$ (Proof of Theorem~\ref{thm:vrjp-near-crit})}
\begin{lemma}[Local Time and Effective Conductance]\label{lem:timescales}
Let $L^{0}_{\infty}$ denote the time the VRJP spends at the origin.
Let $C_{\infty}^{\mrm{eff}}$ be the effective conductance between the origin and infinity in the $t$-field environment.
Also suppose $Z$ is an independent exponential random variable of unit mean.
Then we have
\begin{equation}
  \label{eq:16}
  L^{0}_{\infty} \stackrel{\tiny\text{law}}{=} \sqrt{1+2 Z/C_{\infty}^{\mrm{eff}}}\, -1.
\end{equation}
\end{lemma}

\begin{proof}
  Write $\tilde{L}^{0}_{\infty}$ for the total time the exchangeable timescale VRJP spends at the origin.
  By the time change formula for the local times \eqref{eq:6}, we have:
  \begin{equation}
    \label{eq:15}
    L^{0}_{\infty}=\sqrt{1+\tilde{L}^{0}_{\infty}}-1.
  \end{equation}
  By Theorem~\ref{thm:vrjp-random-walk-random-env}, Lemma~\ref{lem:local-times-symm-exch}, and Lemma~\ref{lem:escape-times-conductance},  $\tilde{L}^{0}_{\infty}$ is $\mrm{Exp}(2/C_{\infty}^{\mrm{eff}})$-distributed.
\end{proof}

\begin{lemma}\label{lem:resistance-tail-bound}
  Let $C^\mrm{eff}_{\infty}$ be as in Theorem~\ref{thm:effective-conductance-nearcrit}.
  For any $\alpha>0$, there exists a constant $c = c(d,\alpha) > 0$, such that for $\epsilon > 0$ small enough and $x \geq e^{c/\sqrt{\epsilon}}$
  \begin{equation}
    \label{eq:130}
    \PP_{\beta_{\mrm{c}}+ \epsilon}[\tfrac{1}{C^\mrm{eff}_{\infty}} > x]
    \leq x^{-\alpha}.
  \end{equation}
  In particular, there exists a constant $C > 0$ such that
  \begin{equation}
    \EE_{\beta_{\mrm{c}} + \epsilon}\Big[\frac{1}{C^{\mrm{eff}}_\infty}\Big]\leq e^{\frac{C}{\sqrt{\epsilon}}}\label{eq:91}
  \end{equation}
\end{lemma}
\begin{proof}
  Recall that the $t$-field environment is given by edge-weights $\{\beta_{ij}e^{T_{i} + T_{j}}\}_{ij \in E(\mbb{T}_{d})}$, where the $t$-field $T_{i}$ has independent increments along outgoing edges and is defined to equal $0$ at the origin.
  In particular, the environment on the subtree emanating from $x$ (which is isomorphic to $\mbb{T}_{d}$) is distributed as a $t$-field environment on $\mbb{T}_{d}$ multiplied by $e^{2T_{x}}$ (which is the same as requiring that the $t$-field equals $T_{x}$ at the ``origin'' $x$).
  For any $n\in \NN$, and a vertex $x$ at generation $n$, write $\omega_{n,x}$ for the effective conductance from $x$ to infinity.
  By the above we have that $\{e^{-2T_{x}}\omega_{n,x}\}_{\abs{x} = n}$ are independently distributed as $C_{\infty}^{\mrm{eff}}$.
  Also, they are independent from the $t$-field up to generation $n$.

  \indent
  In the following, we replace each of the $d^{n}$ subtrees emanating from the vertices $x$ at generation $n$ by a single edge ``to infinity'' with weight $\omega_{n,x}$.
  The resulting network has the same effective conductance between $0$ and infinity.

  \indent
  Define the event
  \begin{equation}
    \label{eq:2}
    A_{n} \coloneqq \{\exists \abs{x} = n : e^{-2T_{x}}\omega_{n,x} > 2c \epsilon\}.
  \end{equation}
  By Lemma~\ref{lem:near-crit-large-conductance} we have $\PP_{\beta_{\mrm{c}} + \epsilon}[e^{-2T_{x}}\omega_{n,x}>2c\epsilon]\geq e^{-2 C/\sqrt{\epsilon}}$ and hence
  \begin{equation}
    \label{eq:3}
    \PP_{\beta_{\mrm{c}} + \epsilon}[A_{n}^{\mrm{c}}] = 1 - \PP_{\beta_{\mrm{c}} + \epsilon}[A_{n}] \leq (1-e^{-2 C/\sqrt{\epsilon}})^{d^n}\leq e^{-d^n e^{-2 C/\sqrt{\epsilon}}},
  \end{equation}
  which is small for appropriately chosen $n$.

  \indent
  Hence, suppose we are working under the event $A_{n}$, and let $x_0$ be a vertex at generation $n$, such that $e^{-2T_{x_0}}\omega_{n,x_0}>2c\epsilon$.
  The effective conductance on the tree is larger than the effective conductance on the subgraph where we only keep the edges between $0$ and $x_0$, as well as an edge between $x_0$ and infinity with conductance $e^{2T_{x_{0}}} 2c\epsilon < \omega_{n,x_{0}}$.
  Denote the conductance of this reduced graph by $C^{\mrm{red}}$.
  We write $y_0=0,\dots,y_n=x_0$ for the vertices along the path from $0$ to $x_0$.
  The series formula for conductances yields
  \begin{equation}\label{eq:92}
    \frac{1}{C^\mrm{eff}_{\infty}}\leq\frac{1}{C^{\mrm{red}}} = \frac{1}{\beta} \sum\limits_{i=0}^{n-1}e^{-(T_{y_i}+T_{y_{i+1}})} + \frac{1}{2c\epsilon}e^{-2 T_{y_n}}.
  \end{equation}
  We bound $T_{y_{i}} + T_{y_{i+1}} \geq 2 \min(T_{y_{i}}, T_{y_{i+1}})$.
  Recall that $T_{y_{i}} \overset{\tiny\text{law}}{=} \sum_{k=0}^{i} T^{(k)}$ with i.i.d.\ samples $\{T^{(k)}\}_{k\geq 0}$ from the $t$-field increment measure \eqref{eq:182}.
  This yields
  \begin{equation}
    \label{eq:7}
    \frac{1}{C^{\mrm{red}}} \leq (\tfrac{n}{\beta} + \tfrac{1}{2c\epsilon}) e^{-2\min(T_{y_{0}}, \ldots, T_{y_{n}})}.
  \end{equation}
  For fixed $\tau > 0$ we apply a union bound and Chernoff's bound (resp.\ Lemma~\ref{lem:t-field-sum-tail-bound})
  \begin{equation}\textstyle
    \label{eq:8}
    \begin{aligned}
      \PP_{\beta}[\min(T_{y_{0}}, \ldots, T_{y_{n}}) < - n \tau]
      & \leq \sum_{i=0}^{n} \PP[{\textstyle\sum_{k=0}^{i}T^{(k)}} < -n\tau]\\
      & \leq \sum_{i=0}^{n} \exp(-i \Psi_{\beta}^{\ast}(\tfrac{n}{i} \tau)),
    \end{aligned}
  \end{equation}
  where $\Psi^{\ast}_{\beta}(\tau) = \sup_{\lambda \geq 0}(\lambda \tau - \log \EE_{\beta} [e^{-\lambda T}])$ is the Fenchel-Legendre dual of the (negative) $t$-field increment's log-MGF.
  Convexity of $\Psi_{\beta}^{\ast}$ (and $\Psi^{\ast}_{\beta}(0) = 0$) implies $\Psi_{\beta}^{\ast}(\tfrac{n}{i} \tau) \geq \tfrac{n}{i} \Psi_{\beta}^{\ast}(\tau)$.
  Consequently, \eqref{eq:8} yields
  \begin{equation}
    \label{eq:24}
    \PP_{\beta}[\min(T_{y_{0}}, \ldots, T_{y_{n}}) < - n \tau] \leq (n+1) e^{-n\Psi^{\ast}_{\beta}(\tau)} \qq{for} \tau > 0
  \end{equation}
  which by \eqref{eq:92} and \eqref{eq:7} implies
  \begin{equation}
    \label{eq:107}
    \PP_{\beta_{\mrm{c}} + \epsilon}[\tfrac{1}{C^{\mrm{eff}}_{\infty}} > (\tfrac{n}{\beta} + \tfrac{1}{2c\epsilon}) e^{2 n \tau} |A_{n}] \leq (n+1)\exp[-n \Psi^{\ast}_{\beta_{\mrm{c}} + \epsilon}(\tau)],
  \end{equation}
  In Appendix~\ref{sec:tail-bound-t-field} we obtain lower bounds on $\Psi_{\beta}^{\ast}$ (Lemma~\ref{lem:t-field-sum-tail-bound}).
  By \eqref{eq:4}, we have that for fixed $\alpha > 0$ and sufficiently small $\epsilon > 0$, any sufficiently large $\tau > 0$ will satisfy $\Psi^{\ast}_{\beta_{\mrm{c}} + \epsilon}(\tau) \geq 7\alpha\tau$, uniformly as $\epsilon \searrow 0$.
  To conclude, we choose $n\geq N(\epsilon) \coloneqq \frac{4C}{\log(d) \sqrt{\epsilon}}$, such that $\PP[A_{n}] \leq e^{-d^{n/2}}$.
  In conclusion, with above choices, \eqref{eq:3} and \eqref{eq:107} yield
  \begin{equation}
    \label{eq:129}
    \PP_{\beta_{\mrm{c}}+ \epsilon}\big[\tfrac{1}{C^\mrm{eff}_{\infty}} > e^{3n\tau}\big]
    \leq e^{-6n\alpha\tau}+ e^{-d^{n/2}}
  \end{equation}
  This implies the claim.
\end{proof}

\begin{proof}[Proof of Theorem~\ref{thm:vrjp-near-crit}]
  We start with the lower bound:
  By Lemma~\ref{lem:timescales} there exists an exponential random variable $Z$ of expectation $1$ such that:
  \begin{equation}
    \begin{aligned}
      \EE[L^{0}_{\infty}]
      &= \EE\big[\sqrt{1+2 Z/C^{\mrm{eff}}_\infty}\big]-1\\
      &\geq  \EE\big[\sqrt{1+2 Z/\EE(C^{\mrm{eff}}_\infty)}\big]-1 \text{ by cond.\ Jensen inequality}\\
      &\geq  \EE[\sqrt{Z}]/\EE[C^{\mrm{eff}}_\infty]-1 \\
      &\geq  \exp(c/\sqrt{\epsilon}) -1 \text{ by Theorem~\ref{thm:effective-conductance-nearcrit}}.
    \end{aligned}\label{eq:88}
  \end{equation}
  For the upper bound, we start with Jensen's inequality:
  \begin{equation}\label{eq:89}
    \begin{aligned}
      \EE[L^{0}_{\infty}]
      &= \EE\big[\sqrt{1+2Z/C^{\mrm{eff}}_{\infty}}-1\big]\\
      &\leq \sqrt{1+2\EE\big[Z/C^{\mrm{eff}}_{\infty}\big]}-1\\
      &= \sqrt{1+2\EE\big[1/C^{\mrm{eff}}_{\infty}\big]}-1\\
      &\leq \sqrt{2} \sqrt{\EE\big[1/C^{\mrm{eff}}_{\infty}\big]}.
    \end{aligned}
  \end{equation}
  The result now follows by Lemma~\ref{lem:resistance-tail-bound}.
\end{proof}

\newpage
\section{Intermediate Phase of the VRJP}
\label{sec:intermediate-phase}
In this section we show that the VRJP on large finite regular trees exhibits an intermediate phase.
We also argue that Rapenne's recent results \cite{rapenne_about_2023} imply the \emph{absence} of such an intermediate phase on regular trees with \emph{wired} boundary conditions.

\subsection{Existence of an Intermediate Phase on $\mbb{T}_{d,n}$ (Proof of Theorem~\ref{thm:int-phase-vrjp})}
The intermediate phase is characterised by the VRJP, despite being transient, spending ``unusually much'' time at the root.
To be precise, on finite trees the fraction of time spent at the origin scales with the system size as a \emph{fractional power} of the inverse system volume.
At the second transition point the walk then reverts to the behaviour that one expects by comparison with simple random walk, spending time inversely proportional to the tree's volume at the starting vertex.

\indent
We will see that the different scalings will be due to different regimes for the log-Laplace transform of the $t$-field increments, $\psi_{\beta}(\eta) = \log[d \,\EE_{\beta} e^{\eta T}]$, as elaborated in Section~\ref{sec:t-field-brw}.

\indent
Before starting the proof, we show how the observable in Theorem~\ref{thm:int-phase-vrjp} can be rephrased in terms of a $t$-field.
The proof will then proceed by analysing the resulting $t$-field quantity via branching random walk methods.

\begin{lemma}\label{lem:frac-local-time-t-field}
  Consider the situation of Theorem~\ref{thm:int-phase-vrjp}.
  Further consider a $t$-field $\{T_{x}\}$ on $\mbb{T}_{d,n}$, rooted at the origin $0$.
  We then have
  \begin{equation}\textstyle
    \label{eq:38}
    \lim_{t\to\infty} \tfrac{L^{0}_{t}}{t}
    \overset{\text{law}}{=}
    \Big[\sum_{\abs{x} \leq n}e^{T_{x}}\Big]^{-1},
  \end{equation}
\end{lemma}

\begin{proof}
  Trivially one has $t = \sum_{\abs{x} \leq n} L_{t}^{x}$.
  Consequently,
  \begin{equation}
    \label{eq:104}
    \lim_{t\to\infty} \tfrac{L^{0}_{t}}{t} = \lim_{t\to\infty} \Big[\sum_{\abs{x} \leq n}L^{x}_{t}/L^{0}_{t}\Big]^{-1}.
  \end{equation}
  Hence, the claim follows from Corollary~\ref{corr:t-field-as-local-time}.
\end{proof}

\begin{proof}[Proof of Theorem~\ref{thm:int-phase-vrjp}]
  In light of Lemma~\ref{lem:frac-local-time-t-field} we consider a $t$-field $\{T_{x}\}$ on $\mbb{T}_{d}$, rooted at the origin.
  In the following we analyse the asymptotic behaviour of the random variable $\sum_{\abs{x} \leq n}e^{T_{x}}$.\\

  \emph{Case $\beta_{\mrm{c}} < \beta < \beta_{\mrm{c}}^{\mrm{erg}}$:}
  We note that it suffices to show
  \begin{equation}
    \label{eq:59}
    \sum_{\abs{x} \leq n} e^{T_{x}} = e^{n \gamma_{\beta} + o(n)} \qq{a.s. for} n\to\infty \qq{with} \gamma_{\beta} = \inf_{\eta > 0} \psi_{\beta}(\eta)/\eta > 0,
  \end{equation}
  since we have $0 < \gamma_{\beta} < \log(d)$ by Proposition~\ref{prop:beta-erg-characterisation}.
  The lower bound in \eqref{eq:59} follows from Theorem~\ref{prop:asymp-velocity}:
  \begin{equation}
    \label{eq:124}
    \sum_{\abs{x} \leq n} e^{T_{x}} \geq \sum_{\abs{x} = n} e^{T_{x}} \geq e^{\max_{\abs{x} = n} T_{x}} = e^{n\gamma_{\beta} + o(n)}.
  \end{equation}
  For the upper bound in \eqref{eq:59} note that for $\eta \in (0,1)$ and $\epsilon > 0$ we have
  \begin{equation}
    \label{eq:135}
    \begin{aligned}
      \PP[\sum_{\abs{x} \leq n} e^{T_{x}} > e^{n(\gamma_{\beta} + \epsilon)}]
      &\leq e^{-n \eta(\gamma_{\beta} + \epsilon)} \EE[(\sum_{\abs{x} \leq n} e^{T_{x}})^{\eta}]\\
      &\leq e^{-n \eta(\gamma_{\beta} + \epsilon)} \EE[\sum_{\abs{x} \leq n} e^{\eta T_{x}}]\\
      &= e^{-n \eta(\gamma_{\beta} + \epsilon)} \sum_{k=0}^{n} e^{\psi(\eta) k}
    \end{aligned}
  \end{equation}

  Now let $\eta = \eta_{\beta}$ as in Lemma~\ref{lem:crit-exp-behaviour}, \ie such that $\gamma_{\beta} = \psi_{\beta}(\eta_{\beta})/\eta_{\beta} > 0$.
  Note that by Proposition~\ref{prop:beta-erg-characterisation}, we have $\gamma_{\beta} \in (0,\log(d))$.
  With this choice \eqref{eq:135} implies $\limsup_{n\to\infty}\tfrac{1}{n}\log\sum_{\abs{x} \leq n} e^{T_{x}} \leq \gamma_{\beta} + \epsilon$ almost surely for any $\epsilon > 0$.
  This yields the lower bound in \eqref{eq:59}.\\

  \emph{Case $\beta \leq \beta_{\mrm{c}}$:} This proceeds similarly to the previous case.
  For the lower bound we simply use $\sum_{\abs{x}\leq n} e^{T_{x}} \geq e^{T_{0}} = 1$.
  For the lower bound we use \eqref{eq:135} with $\gamma_{\beta} \mapsto 0$ and $\eta = 1/2$, which implies that $\limsup_{n\to\infty}\tfrac{1}{n}\log\sum_{\abs{x} \leq n} e^{T_{x}} \leq \epsilon$. almost surely for any $\epsilon > 0$.\\
  
  \emph{Case $\beta > \beta_{\mrm{c}}^{\mrm{erg}}$:} First note that the quantity $W_{n} \coloneqq d^{-n} \sum_{\abs{x} = n} e^{T_{x}}$ is a martingale.
  In the branching random walk literature this is referred to as the \emph{additive martingale} associated with the BRW $\{T_{x}\}_{x\in \mbb{T}_{d}}$.
  Since $W_{n}$ is non-negative it converges almost surely to a random variable $W_{\infty} = \lim_{n\to\infty} W_{n}$.
  Biggin's martingale convergence theorem \cite[Theorem~3.2]{shi_branching_2015} implies that for $\beta > \beta_{\mrm{c}}^{\mrm{erg}}$ (equivalently $\psi_{\beta}'(1) < \psi_{\beta}(1)$, see Proposition~\ref{prop:beta-erg-characterisation}), the sequence is uniformly integrable and the limit $W_{\infty}$ is almost surely strictly positive.
  Consequently we also get convergence for the weighted average
  \begin{equation}
    \label{eq:39}
    \frac{1}{\abs{\mbb{T}_{d,n}}} \sum_{\abs{x} \leq n}e^{T_{x}} = \frac{1}{\abs{\mbb{T}_{d,n}}}\sum_{k=0}^{n}d^{k}W_{k} \to W_{\infty} > 0 \qq{a.s. for } n\to\infty.
  \end{equation}
  In other words,
  \begin{equation}
    \label{eq:41}
    \sum_{\abs{x} \leq n}e^{T_{x}} \sim \abs{\mbb{T}_{d,n}} W_{\infty} = d^{n + O(1)} \qq{as} n\to\infty,
  \end{equation}
  which implies the claim for $\beta > \beta_{\mrm{c}}^{\mrm{erg}}$.
\end{proof}

\subsection{Multifractality of the Intermediate Phase (Proof of Theorem~\ref{thm:vrjp-multifractality})}

For the proof we will make use of explicit large deviation asymptotics for the maximum of the $t$-field.
These follow (as an easy special case) from results due to Gantert and Höfelsauer on the large deviations of the maximum of a branching random walk \cite[Theorem~3.2]{gantert_large_2018}:

\begin{lemma}\label{lem:t-field-large-deviations}
  Consider the $t$-field $\{T_{x}\}_{x\in\mbb{T}_{d}}$ on $\mbb{T}_{d}$, pinned at the origin $0$.
  Let $\gamma_{\beta} = \inf_{\eta > 0} \psi_{\beta}(\eta)/\eta$ as in \eqref{eq:17}.
  For any $\gamma > \gamma_{\beta}$ we have
  \begin{equation}\textstyle
    \label{eq:201}
    \liminf_{n\to\infty} \tfrac{1}{n} \log \PP[\max_{\abs{x} = n} T_{x} \geq n\gamma]
    =
    - \sup_{\eta \in \RR} [\gamma \eta - \psi_{\beta}(\eta)] < 0.
  \end{equation}
\end{lemma}

\begin{proof}
  As noted, this is a direct consequence of \cite[Theorem~3.2]{gantert_large_2018}.
  To be precise, we consider the special case of a deterministic offspring distribution (instead of Galton-Watson trees) and fluctuations \emph{above} the asymptotic velocity $\gamma_{\beta}$ (corresponding to the case $x > x^{\ast}$ in \cite{gantert_large_2018}).
  In this case, the rate function given by Gantert and Höfelsauer (denoted by $x\mapsto I(x) - \log(m)$ in their article) is equal to
  \begin{equation}
    \label{eq:202}
    \gamma \mapsto \sup_{\eta \in \RR} (\gamma \eta - \log \EE[e^{\eta T}]) - \log d
    =
    \sup_{\eta \in \RR} [\gamma \eta - \psi_{\beta}(\eta)].
  \end{equation}
  This concludes the proof.
\end{proof}

\begin{proof}[Proof of Theorem~\ref{thm:vrjp-multifractality}]
  By Lemma~\ref{lem:frac-local-time-t-field}, we would like to understand fractional moments of
  \begin{equation}
    \label{eq:134}
    [\lim_{t\to\infty}L_{t}^{0}/t]^{-1}
    \overset{\tiny\text{law}}{=} \sum_{\abs{x}\leq n} e^{T_{x}},
  \end{equation}
  where $\{T_{x}\}_{x\in \mbb{T}_{d}}$ denotes a $t$-field on the rooted $(d+1)$-regular tree, pinned at the origin.
  Recall the definition of $\eta_{\beta}$ in \eqref{eq:17} and Lemma~\ref{lem:crit-exp-behaviour}.
  For $\beta \in (\beta_{\mrm{c}}, \beta_{\mrm{c}}^{\mrm{erg}})$ we have $\eta_{\beta} \in (0,1)$ by Proposition~\ref{prop:beta-erg-characterisation}.

  \indent
  \emph{Case $\eta \in \opcl{0,\eta_{\beta}}$:}
  We recall Proposition~\ref{prop:asymp-velocity}, which implies that
  \begin{equation}
    \label{eq:200}
    \lim_{n\to\infty}\frac{1}{n} \max_{\abs{x} = n} T_{x} = \gamma_{\beta} = \psi_{\beta}(\eta_{\beta})/\eta_{\beta}.
  \end{equation}
  By Jensen's inequality and Fatou's lemma we get
  \begin{equation}
    \label{eq:199}
    \begin{aligned}
      \liminf_{n \to \infty} \tfrac{1}{n} \log \EE [(\sum_{\abs{x}\leq n} e^{T_{x}})^{\eta}]
      &\geq \liminf_{n \to \infty} \tfrac{1}{n} \log \EE [e^{\eta \max_{\abs{x}=n}T_{x}}]\\
      &\geq \liminf_{n\to\infty} \frac{\eta}{n} \EE[\max_{\abs{x} = n} T_{x}]\\
      &\geq \eta \psi_{\beta}(\eta_{\beta})/\eta_{\beta}.
    \end{aligned}
  \end{equation}
  On the other hand, since $\eta / \eta_{\beta} \leq 1$
  \begin{equation}
    \label{eq:181}
    \begin{aligned}
      \EE [(\sum_{\abs{x}\leq n} e^{T_{x}})^{\eta}]
      \leq
      \EE [(\sum_{\abs{x}\leq n} e^{T_{x}})^{\eta_{\beta}}]^{\eta/\eta_{\beta}}
    \end{aligned}
  \end{equation}
  For any $\eta \in (0,1)$ and $\beta > \beta_{\mrm{c}}$ we can bound
  \begin{equation}
    \label{eq:184}
    \EE [(\sum_{\abs{x}\leq n} e^{T_{x}})^{\eta}]
    \leq
    \EE [\sum_{\abs{x}\leq n} e^{\eta T_{x}}]
    \leq
    \sum_{k=0}^{n} e^{k \psi_{\beta}(\eta)}
    \leq
    e^{n \psi_{\beta}(\eta) + o(n)},
  \end{equation}
  where we used that $\inf_{\eta> 0} \psi_{\beta}(\eta) = \psi_{\beta}(1/2) > 0$ for $\beta > \beta_{\mrm{c}}$ (\cf \eqref{eq:164}, \eqref{eq:163} and \eqref{eq:11}).
  Applying this to the last line of \eqref{eq:181}, we obtain
  \begin{equation}
    \label{eq:183}
    \EE [(\sum_{\abs{x}\leq n} e^{T_{x}})^{\eta}]
    \leq
    e^{n\, \eta\, \psi_{\beta}(\eta_{\beta})/\eta_{\beta} + o(n)}
  \end{equation}

  \indent
  \emph{Case $\eta \in \clop{\eta_{\beta}, 1}$:}
  The upper bound already follows from \eqref{eq:184}.
  For the lower bound we start with
  \begin{equation}
    \label{eq:185}
    \begin{aligned}
      \EE [(\sum_{\abs{x}\leq n} e^{T_{x}})^{\eta}]
      &\geq
      \EE [e^{\eta \max_{\abs{x} = n} T_{x}}]\\
      &\geq
        e^{n\eta \gamma}\, \PP[\max_{\abs{x} = n} T_{x} \geq n\gamma] \qq{for any} \gamma > 0.
    \end{aligned}
  \end{equation}
  We get that for any $\gamma \in \RR$:
  \begin{equation}
    \label{eq:195}
    \liminf_{n\to\infty} \tfrac{1}{n} \log \EE [(\sum_{\abs{x}\leq n} e^{T_{x}})^{\eta}]
    \geq
    \eta\gamma + \liminf_{n\to\infty} \tfrac{1}{n} \log \PP[\max_{\abs{x} = n} T_{x} \geq n\gamma].
  \end{equation}
  By Lemma~\ref{lem:t-field-large-deviations}, we have
  \begin{equation}
    \label{eq:203}
    \liminf_{n\to\infty} \tfrac{1}{n} \log \EE [(\sum_{\abs{x}\leq n} e^{T_{x}})^{\eta}]
    \geq
    \sup_{\gamma > \gamma_{\beta}} \Big(\eta\gamma - \sup_{\tilde{\eta} \in \RR} [\gamma \tilde{\eta} - \psi_{\beta}(\tilde{\eta})]\Big).
  \end{equation}
  We claim that the right hand side of \eqref{eq:203} is equal to $\psi_{\beta}(\eta)$.
  For the upper bound simply choose $\tilde{\eta} = \eta$.
  For the lower bound first note that the supremum of $\tilde{\eta}\mapsto \gamma\tilde{\eta} - \psi_{\beta}{\tilde{\eta}}$ is attained at the unique $\tilde{\eta}$, such that $\psi_{\beta}^{\prime}(\tilde{\eta}) = \gamma$ (uniqueness follow from convexity of $\eta\mapsto \psi_{\beta}(\eta)$).
  Since we assumed $\eta > \eta_{\beta}$, we may choose $\gamma = \psi_{\beta}^{\prime}(\eta)$, satisfying $\gamma > \gamma_{\beta} = \psi_{\beta}^{\prime}(\eta_{\beta})$.
  Together with previous observation this shows that the right hand side is larger or equal to $\psi_{\beta}(\eta)$.
  This concludes the proof.
\end{proof}

\subsection{On the Intermediate Phase for Wired Boundary Conditions}
We recall that for the Anderson transition it was debated whether an intermediate multifractal phase persists in the infinite volume and on tree-like graphs without free boundary conditions (see Section~\ref{sec:further-comments}).

\indent
We \emph{conjecture} that there is no intermediate phase for the VRJP on regular trees with wired boundary conditions.
In this section, we would like to provide some evidence for this claim, based on recent work by Rapenne \cite{rapenne_about_2023}.

\indent
Let $\overline{\mbb{T}}_{d,n}$ denote the rooted $(d+1)$-regular tree of depth $n$ with \emph{wired} boundary, \ie all vertices at generation $n$ have an outgoing edge to a single \emph{boundary ghost} $\mfk{g}$.
We consider $\mbb{T}_{d,n} \subset \overline{\mbb{T}}_{d,n}$ as a the subgraph induced by the vertices excluding the ghost.
Let $\{\overline{T}^{\mfk{g}}_{x}\}_{x\in \overline{T}_{d,n}}$ denote a $t$-field on the wired tree $\overline{\mbb{T}}_{d,n}$, pinned at the ghost $\mfk{g}$, and at inverse temperature $\beta$.
We define
\begin{equation}
  \label{eq:206}
  \psi_{n}(x) = e^{\overline{T}^{\mfk{g}}_{x}} \text{ for } x\in \mbb{T}_{d,n},
\end{equation}
where we use the index $n$ to make the dependence on the underlying domain $\overline{\mbb{T}}_{d,n}$ more explicit.
This coincides with the (vector) martingale $\{\psi_{n}(x)\}_{x \in \mbb{T}_{d,n}}$ considered by Rapenne (see \cite[Lemma~2]{sabot_random_2018} for a proof that these are in fact the same).
By \cite[Theorem~2]{rapenne_about_2023} we have for $\beta > \beta_{\mrm{c}}$ and $p\in(1,\infty)$
\begin{equation}\textstyle
  \label{eq:208}
  \sup_{n\geq 1} \EE_{\beta}[\psi_{n}(0)^{p}] < \infty.
\end{equation}
Our statement about the absence of an intermediate phase, will be conditional on a (conjectural) extension of this result:
\begin{equation}
  \label{eq:64}
  \text{\itshape Conjecture:  }
  \sup_{n\geq 1} \frac{1}{\abs{\mbb{T}_{d,n}}} \sum_{x \in \mbb{T}_{d,n}} \EE_{\beta}[\psi_{n}(x)^{p}] < \infty \qq{for} p>1 \text{ and } \beta > \beta_{\mrm{c}}.
\end{equation}
We believe this statement to be true due to the following heuristic:
Given that the origin of $\overline{\mbb{T}}_{d,n}$ is furthest away from the ghost $\mfk{g}$, at which the $t$-field in \eqref{eq:206} is pinned, we expect the fluctuations of $\psi_{n}(x)$ to be  largest at $x=0$.
Hence, we expect the moments of $\psi_{n}(x)$ to be comparable with the ones of $\psi_{n}(0)$, in which case \eqref{eq:208} would imply \eqref{eq:64}.

\begin{proposition}
  Consider a VRJP started from the root of $\overline{\mbb{T}}_{d,n}$ and let $L_{t}^{0}$ denote the time it spent at root up until time $t$.
  Assume \eqref{eq:64} holds true.
  Then, for any $\beta > \beta_{\mrm{c}}$
  \begin{equation}
    \label{eq:55}
    \lim_{t\to\infty} \frac{L_{t}^{0}}{t} \leq \abs{\mbb{T}_{d,n}}^{-1 + o(1)} \qq{w.h.p.\ as} n\to\infty.
  \end{equation}
\end{proposition}

This is to be contrasted with the behaviour in Theorem~\ref{thm:int-phase-vrjp}.

\begin{proof}
  Let $\{\overline{T}_{x}\}_{x\in\overline{\mbb{T}}_{d,n}}$ denote the $t$-field on $\overline{\mbb{T}}_{d,n}$, pinned at the origin $0$.
  We stress that this is different from $\overline{T}^{\mfk{g}}_{x}$, as used in \eqref{eq:206}, which is pinned at the ghost $\mfk{g}$.
  However, we can sample the former from the latter:
  First consider an STZ-Anderson operator $H_{B}$ on the infinite graph $\mbb{T}_{d}$, as defined in Definition~\ref{def:stz-anderson}.
  Define $\hat{G}_{n} \coloneqq (H_{B}\vert_{\mbb{T}_{d,n}})^{-1}$ and also define $\{\psi_{n}(x)\}_{x\in\mbb{T}_{d}}$ by
  \begin{equation}
    \label{eq:127}
    (H_{B}\psi_{n})\vert_{\mbb{T}_{d,n}} = 0 \text{ and } \psi\vert_{\mbb{T}_{d} \setminus \mbb{T}_{d,n}} \equiv 1.
  \end{equation}
  By \cite[Lemma~2]{sabot_random_2018}, the $\psi_{n}$ so defined (and restriced to $\mbb{T}_{d,n}$) agree in law with the definition in \eqref{eq:206}.
  Then define $\overline{T}_{x}$ for $x\in\mbb{T}_{d,n}$ via
  \begin{equation}
    \label{eq:54}
    e^{\overline{T}_{x}} = \frac{\hat{G}_{n}(0,x) + \frac{1}{2\gamma} \psi_{n}(0)\psi_{n}(x)}{\hat{G}_{n}(0,0) + \frac{1}{2\gamma} \psi_{n}(0)^{2}},
  \end{equation}
  where $\gamma \sim \mrm{Gamma(\tfrac12, 1)}$ is independent of $H_{B}$.
  By \cite[Proposition~8]{sabot_random_2018}, $\{\overline{T}_{x}\}_{x\in\mbb{T}_{d,n}}$ has the law of a $t$-field on $\overline{\mbb{T}}_{d,n}$, pinned at the origin $0$ (and restricted to $\mbb{T}_{d,n}$).
  Note that $\overline{\mbb{T}}_{\mfk{g}}$ is not defined by \eqref{eq:54}.
  Using the conditional law of the $t$-field on $\overline{\mbb{T}}_{d,n}$ given its values away from the ghost, we can however define it such that $\{\overline{T}_{x}\}_{x\in\overline{\mbb{T}}_{d,n}}$ is the ``full'' $t$-field on $\overline{\mbb{T}}_{d,n}$, pinned at the origin.
  Then, as in \eqref{eq:38}, we have that
  \begin{equation}
    \label{eq:56}
    \lim_{t\to\infty} \frac{L_{t}^{0}}{t}
    \overset{\text{law}}{=}
    \Bigg[\sum_{x \in \overline{\mbb{T}}_{d,n}}e^{\overline{T}_{x}}\Bigg]^{-1}.
  \end{equation}
  By \eqref{eq:54} and positivity of $\hat{G}$ we get
  \begin{equation}
    \label{eq:57}
    \sum_{x \in \overline{\mbb{T}}_{d,n}}e^{\overline{T}_{x}}
    \geq \sum_{x \in \mbb{T}_{d,n}}e^{\overline{T}_{x}}
    \geq \frac{\psi_{n}(0)}{2\gamma \hat{G}_{n}(0,0) + \psi_{n}(0)^{2}} \sum_{x \in \mbb{T}_{d,n}} \psi_{n}(x).
  \end{equation}
  By \cite[Theorem~1]{sabot_random_2018}, for $\beta > \beta_{\mrm{c}}$ the fraction on the right hand side converges a.s.\ to a (random) positive number as $n\to\infty$.
  Hence, the claim in \eqref{eq:55} follows if we show that $\sum_{x \in \mbb{T}_{d,n}} \psi_{n}(x) \geq \abs{\mbb{T}_{d,n}}^{1 - o(1)}$ a.s.\ as $n\to \infty$.
  For any $s>0$ and $q\geq 1$ we have
  \begin{equation}
    \label{eq:58}
    \begin{aligned}
      \PP[\sum_{x \in \mbb{T}_{d,n}} \psi_{n}(x) \leq s \abs{\mbb{T}_{d,n}}]
      &= \PP[(\frac{1}{\abs{\mbb{T}_{d,n}}} \sum_{x \in \mbb{T}_{d,n}} \psi_{n}(x))^{-q} \geq s^{-q}]\\
      &\leq s^{q}\, \EE[(\frac{1}{\abs{\mbb{T}_{d,n}}} \sum_{x \in \mbb{T}_{d,n}} \psi_{n}(x))^{-q}]\\
      &\leq s^{q} \frac{1}{\abs{\mbb{T}_{d,n}}} \sum_{x \in \mbb{T}_{d,n}} \EE[\psi_{n}(x)^{-q}]\\
      &= s^{q} \frac{1}{\abs{\mbb{T}_{d,n}}} \sum_{x \in \mbb{T}_{d,n}} \EE[\psi_{n}(x)^{1+q}],
    \end{aligned}
  \end{equation}
  where in the last line we used the reflection property of the $t$-field (see Lemma~\ref{lem:reflection-property}).
  Subject to the assumption that \eqref{eq:64} holds true, we may choose $q=1$ and $s = n^{-2}$ in \eqref{eq:58}.
  An application of the Borel-Cantelli lemma then yields that $\sum_{x \in \mbb{T}_{d,n}} \psi_{n}(x) \geq \abs{\mbb{T}_{d,n}}^{1 - o(1)}$ a.s.\ as $n\to \infty$.
  Together with \eqref{eq:56} and \eqref{eq:57}, this implies \eqref{eq:55}.
\end{proof}

\newpage
\section{Results for the $\HH^{2|2}$-Model} \label{sec:h22-results}
\subsection{Asymptotics for the $\HH^{2|2}$-Model as $\beta\searrow \beta_{\mrm{c}}$ (Proof of Theorem~\ref{thm:h22-asymptotics})}
\begin{proof}[Proof of Theorem~\ref{thm:h22-asymptotics}.]
  By Theorem~\ref{thm:vrjp-near-crit} it suffices to show that
  \begin{equation}
    \label{eq:48}
    \eva{x_{0}^{2}}_{\beta}^{+} = \lim_{h\searrow 0} \lim_{n\to\infty} \eva{x_{0}^{2}}_{\beta;h,\mbb{T}_{d,n}} = \EE_{\beta}[L_{\infty}^{0}].
  \end{equation}
  For this, we use the $\HH^{2|2}$-Dynkin isomorphism (Theorem~\ref{thm:h22-dynkin}):
  \begin{equation}
    \label{eq:113}
    \eva{x_{0}^{2}}_{\beta;h,\mbb{T}_{d,n}}
    = \int\limits_{0}^{\infty}\dd{t} \EE_{\beta;\mbb{T}_{d,n}}\big[e^{-ht}\, \unit_{X_{t}=0}\big],
  \end{equation}
  where, subject to $\mbb{E}_{\beta;\mbb{T}_{d,n}}$, $(X_{t})_{t\geq 0}$ is a VRJP on $\mbb{T}_{d,n}$ started at $0$.
  Coupling the VRJP on $\mbb{T}_{d,n}$ with a VRJP on the infinite tree $\mbb{T}_{d}$ up to the time they first visit the leaves of $\mbb{T}_{d,n}$, we get
  \begin{equation}
    \label{eq:116}
    \abs{\EE_{\beta;\mbb{T}_{d,n}}[\unit_{X_{t} = 0}] - \EE_{\beta;\mbb{T}_{d}}[\unit_{X_{t}}=0]}
    \leq \PP_{\beta;\mbb{T}_{d}}[T_{n} \leq t],
  \end{equation}
  with $T_{n}$ being the VRJP's hitting time of $\de\mbb{T}_{d,n} = \{x\in \mbb{T}_{d,n}: \abs{x} = n\}$.
  By definition of the VRJP, the time it takes to reach $\de\mbb{T}_{d,n}$ is stochastically lower bounded by an exponential random variable of rate $d\beta/n$.
  Consequently, the right hand side of \eqref{eq:116} converges to zero as $n\to \infty$.
  By this observation and the monotone convergence theorem we have
  \begin{equation}
    \label{eq:117}
    \eva{x_{0}^{2}}_{\beta}^{+}
    = \lim_{h\searrow 0} \int\limits_{0}^{\infty}\dd{t} e^{-ht} \EE_{\beta;\mbb{T}_{d}}\big[\unit_{X_{t}=0}\big]
    =  \int\limits_{0}^{\infty}\dd{t} \EE_{\beta;\mbb{T}_{d}}\big[\unit_{X_{t}=0}\big]
    = \EE_{\beta;\mbb{T}_{d}}[L_{\infty}^{0}],
  \end{equation}
  which proves the claim.
\end{proof}

\subsection{Intermediate Phase for the $\HH^{2|2}$-Model (Proof of Theorem~\ref{thm:int-phase-h22})}
\label{sec:h22-intermediate}
In this section, we want to prove Theorem~\ref{thm:int-phase-h22} on the intermediate phase of the $\HH^{2|2}$-model.
We will make use of the STZ-Anderson model, as defined in Definition~\ref{def:stz-anderson}, making use of its restriction properties as discussed in \cite{letac_multivariate_2020,poudevigne-auboiron_monotonicity_2022}.

\indent
The proof consists of three parts: First we evaluate the quantity on the left hand side of \eqref{eq:161} on a graph consisting of a single vertex (and a coupling to a ghost vertex).
Then we reduce the actual quantity in \eqref{eq:161} onto the case of a single vertex with a random \emph{effective magnetic field} $h^{\mrm{eff}}$.
As $h\searrow 0$, the law of $h^{\mrm{eff}}$ can be expressed in terms of the $t$-field and we can deduce Theorem~\ref{thm:int-phase-h22} from Theorem~\ref{thm:vrjp-multifractality} on the VRJP's multifractality.

\begin{lemma}\label{lem:h22-observable-two-points}
  Consider the $\HH^{2|2}$-model on a single vertex $0$ with magnetic field $h>0$.
  For $\eta \in (0,1)$ we have
  \begin{equation}
    \label{eq:65}
    \eva{z_{0} \abs{x_{0}}^{-\eta}}_{h;\{0\}}
    = h^{\eta} \times g_{\eta}(h)
  \end{equation}
  with
  \begin{equation}
    \label{eq:75}
    g_{\eta}(h) \coloneqq \frac1{\pi}e^{h} (2h)^{(1-\eta)/2}\, \Gamma(\tfrac12 - \tfrac{\eta}{2}) K_{(1-\eta)/2}(h).
  \end{equation}
  In particular
  \begin{equation}
    \label{eq:84}
    c_{\eta} \coloneqq \frac1{\pi} 2^{-\eta} \, \Gamma(\tfrac12 - \tfrac{\eta}{2})^{2} = \lim_{h\searrow 0} g_{\eta}(h)
  \end{equation}
\end{lemma}

\begin{proof}
  For convenience, lets write $\eva{\cdot} = \eva{\cdot}_{h;\{0\}}$.
  By $e^{t_{0}} = z_{0}+x_{0}$ and $y_{0} = s_{0}e^{t_{0}}$, see \eqref{eq:36}, we have
  \begin{equation}
    \label{eq:67}
    \begin{aligned}
      \eva{z_{0} \abs{x_{0}}^{-\eta}}
      = \eva{z_{0} \abs{y_{0}}^{-\eta}}
      = \eva{(e^{t_{0}} + x_{0}) \abs{y_{0}}^{-\eta}}
      = \eva{e^{t_{0}} \abs{y_{0}}^{-\eta}}
      &= \eva{e^{t_{0}} \abs{s_{0}}^{-\eta} e^{-\eta t_{0}}}\\
      &= \eva{e^{(1-\eta) t_{0}} \abs{s_{0}}^{-\eta}}.
    \end{aligned}
  \end{equation}
  The last line can be interpreted in purely probabilistic terms:
  $t_{0}$ follows the law of a $t$-field increment with inverse temperature $h>0$ and conditionally on $t_{0}$, $s_{0}$ is a Gaussian random variable with variance $e^{-t_{0}}/h$.
  Consequently,
  \begin{equation}
    \label{eq:272}
    \begin{aligned}
      \EE[|s_{0}|^{-\eta}|t_{0}]
      &= \sqrt{\frac{h e^{t_{0}}}{2\pi}} \int_{-\infty}^{+\infty} \dd{s} |s|^{-\eta} e^{-he^{t_{0}} s^{2}/2}\\
      &= (h e^{t_{0}})^{\eta/2} \frac{1}{\sqrt{2\pi}} \int_{-\infty}^{+\infty} \dd{x} |x|^{-\eta} e^{-x^{2}/2}\\
      &= (h e^{t_{0}})^{\eta/2}\, \frac{2^{-\eta/2}}{\sqrt{\pi}} \Gamma(\tfrac12 - \tfrac{\eta}{2}).
    \end{aligned}
  \end{equation}
  With \eqref{eq:67} we obtain
  \begin{equation}
    \label{eq:73}
    \eva{z_{0} \abs{x_{0}}^{-\eta}}
    = h^{\eta/2}\, \frac{2^{-\eta/2}}{\sqrt{\pi}} \Gamma(\tfrac12 - \tfrac{\eta}{2}) \EE_{h}[e^{(1-\eta/2) T}],
  \end{equation}
  where $T$ denotes a $t$-field increment at inverse temperature $h$.
  Expressing the exponential moments of $T$ in terms of the modified Bessel function of second kind $K_{\alpha}$, as in \eqref{eq:162}, and using small-argument asymptotics for the latter, we obtain
  \begin{equation}
    \label{eq:74}
    \EE_{h}[e^{(1-\eta/2) T}]
    = \frac{\sqrt{2h} e^{h}}{\sqrt{\pi}} K_{(1-\eta)/2}(h) \sim h^{\eta/2} \times \frac{2^{(1-\eta)/2} \Gamma(\tfrac12 - \tfrac{\eta}{2})}{\sqrt{2\pi}} \qq{as} h\searrow 0.
  \end{equation}
  Inserting this into \eqref{eq:73} yields the claim.
\end{proof}

\paragraph{Effective Weight.}
Before proceeding, we need to introduce the notion of \emph{effective weight} for the STZ-field:
Consider an STZ-Anderson model $H_{B}$ as in \ref{def:stz-anderson} and suppose the underlying graph $G = (V,E)$ is finite.
Write $G_{B} = (H_{B})^{-1}$.
Then, for $i_{0},j_{0}\in V$, the effective weight between these two vertices is defined by
\begin{equation}
  \label{eq:188}
  \beta_{i_{0}j_{0}}^{\mrm{eff}} \coloneqq \frac{G_{B}(i_{0},j_{0})}{G_{B}(i_{0},i_{0})G_{B}(j_{0},j_{0}) - G_{B}(i_{0},j_{0})^{2}}.
\end{equation}
Another expression can be deduced using Schur's complement:
Write $V_{0} = \{i_{0}, j_{0}\}$ and $V_{1} = V\setminus\{{i_{0}, j_{0}}\}$ and decompose $H_{B}$ as
\begin{equation}\label{eq:190}
  H_{B} = \mqty(H_{00}& H_{01}\\ H_{10}& H_{11}),
\end{equation}
with $H_{00}$ being the restriction of $H_{B}$ to entries with indices in $V_{0}$ and analogously for the other submatrices.
By Schur's decomposition we have
\begin{equation}
  \label{eq:95}
  \begin{aligned}
    G_{B}\vert_{V_{0}}
    &= H_{B}^{-1}\vert_{V_{0}}\\
    &= (H_{00} - H_{01}H_{11}^{-1}H_{10})^{-1}\\
    &=\Bigg(
      \begin{matrix}
        B_{i_{0}} - [H_{01}H_{11}^{-1}H_{10}](i_{0},i_{0})
        & -\beta_{i_{0}j_{0}} - [H_{01}H_{11}^{-1}H_{10}](i_{0},j_{0})\\
        -\beta_{j_{0}i_{0}} - [H_{01}H_{11}^{-1}H_{10}](j_{0},i_{0})
        & B_{j_{0}} - [H_{01}H_{11}^{-1}H_{10}](j_{0},j_{0})
      \end{matrix}\Bigg)^{-1}.
  \end{aligned}
\end{equation}
Note that \eqref{eq:188} reads as $\beta_{i_{0}j_{0}}^{\mrm{eff}} = G_{B}(i_{0},j_{0}) / \det(G_{B}\vert_{V_{0}}) = G_{B}(i_{0},j_{0}) \det([G_{B}\vert_{V_{0}}]^{-1})$.
Hence using the familiar formula for the inverse of a $2\times 2$-matrix we get
\begin{equation}
  \label{eq:222}
  \beta_{i_{0}j_{0}}^{\mrm{eff}} = \beta_{i_{0}j_{0}} + [H_{01}H_{11}^{-1}H_{10}](i_{0},j_{0}),
\end{equation}
which is measurable with respect to $B\vert_{V_{1}}$.
The relevance of the effective weight stems from the following Lemma (see \cite[Section~6]{poudevigne-auboiron_monotonicity_2022})

\begin{lemma}\label{lem:effective-weight-t-field}
  For a finite graph $G=(V,E)$ with positive edge-weights $\{\beta_{ij}\}_{ij\in E}$ and a pinning vertex $i_{0}$, consider the natural coupling of an STZ-field $(B_{i})_{i\in V}$ and a $t$-field $(T_{i})_{i\in V}$ (see Remark~\ref{rmk:natural-coupling}).
  For a vertex $j_{0} \in V\setminus \{i_{0}\}$ write $V_{0} \coloneqq \{i_{0}, j_{0}\}$ and $V_{1} \coloneqq V\setminus\{{i_{0}, j_{0}}\}$.

  \indent
  Then, conditionally on $B\vert_{V_{1}}$, the $t$-field $T\vert_{V_{0}} = (T_{i_{0}}, T_{j_{0}})$ is distributed as a $t$-field on $V_{0}$, pinned at $i_{0}$, with edge-weight given by $\beta_{i_{0}j_{0}}^{\mrm{eff}} = \beta_{i_{0}j_{0}}^{\mrm{eff}}(B\vert_{V_{1}})$.
\end{lemma}

Moreover, the notion of effective weight and effective conductance are directly related:

\begin{lemma}[Effective Conductance vs.\ Weight]\label{lem:eff-conductance-weight}
  Consider the setting of Lemma~\ref{lem:effective-weight-t-field}.
  For $j_{0} \in V\setminus \{i_{0}\}$, let $C^{\mrm{eff}}_{i_{0}j_{0}}$ denote the effective conductance between $i_{0}$ and $j_{0}$ in the $t$-field environment $\{\beta_{ij}e^{T_{i} + T_{j}}\}_{ij \in E}$.
  Then
  \begin{equation}
    \label{eq:80}
    C^{\mrm{eff}}_{i_{0}j_{0}} = e^{T_{j_{0}}} \beta^{\mrm{eff}}_{i_{0}j_{0}}.
  \end{equation}
\end{lemma}

This statement is proved in Appendix~\ref{sec:eff-cond-weight}.
In the following, we will come back to the setting of the regular tree.

\paragraph{Reduction to Two Vertices on the Tree.}
We denote by $\tilde{\mbb{T}}_{d,n}$ the graph obtained by adding an additional \emph{ghost vertex} $\mfk{g}$ connected to every vertex of the graph $\mbb{T}_{d,n}$.
For the $\HH^{2|2}$-model (and consequently the $t$-/$s$-field) we refer to the model on $\mbb{T}_{d,n}$ \emph{with magnetic field} $h>0$ as the model on $\tilde{\mbb{T}}_{d,n}$, pinned at the ghost $\mfk{g}$, with weights $\beta_{x\mfk{g}} = h$ between the ghost and any other vertex.

\begin{lemma}[Effective Magnetic Field at the Origin]\label{lem:reduction-two-vertices}
  Consider the natural coupling of $t$-field, $s$-field and STZ-field on $\tilde{\mbb{T}}_{d,n}$, at inverse temperature $\beta > 0$ and with magnetic field $h>0$, pinned at the ghost $\mfk{g}$.
  The random fields are denoted by $T_{x}$, $S_{x}$ and $B_{x}$, respectively ($x \in \tilde{\mbb{T}}_{d,n}$).
  Write $V_{0} \coloneqq \{0, \mfk{g}\}$ and $V_{1} \coloneqq \tilde{\mbb{T}}_{d,n}\setminus \{0,\mfk{g}\}$ and define $H_{11} \coloneqq H_{B}\vert_{V_{1}}$.

  \indent
  Conditionally on $B\vert_{V_{1}}$, the $t$-/$s$-field at the origin $(T_{0}, S_{0})$ follows the law of a $t$-/$s$-field on $\{0,\mfk{g}\}$ with \emph{effective magnetic field}
  \begin{equation}
    \label{eq:87}
    h^{\mrm{eff}} \coloneqq \beta^{\mrm{eff}}_{0\mfk{g}} = h + h \beta \sum_{x,y\in V_{1}: y\sim 0} H_{11}^{-1}(y,x).
  \end{equation}
\end{lemma}

\begin{proof}
  By Lemma~\ref{lem:effective-weight-t-field}, conditionally on $B\vert_{V_{1}}$, the $t$-field at the origin $T_{0}$ has the law of a $t$-field increment at inverse temperature $h^{\mrm{eff}}$.
  We claim that the analogous fact is true for the joint measure of $(T_{0}, S_{0})$.

  \indent
  Recall that, conditionally on $\{T_{x}\}$, the law of $\{S_{x}\}$ is that of Gaussian free field, pinned at $\mfk{g}$, edge-weights given by the $t$-field environment $\{\beta_{ij}e^{T_{i} + T_{j}}\}$ over edges in $\tilde{\mbb{T}}_{d,n}$ with $\beta_{x\mfk{g}} = h$.
  Let $C_{0\mfk{g}}^{\mrm{eff}}$ denote the effective conductance between the origin $0$ and the ghost $\mfk{g}$ in the $t$-field environment.
  Then, conditionally on $\{T_{x}\}$, we have that $S_{0}$ is a centred normal random variable with variance given by the effective resistance $1/C^{\mrm{eff}}_{0\mfk{g}}$ (see \cite[Proposition~2.24]{lyons_probability_2016}).
  By Lemma~\ref{lem:eff-conductance-weight} we have $C^{\mrm{eff}}_{0\mfk{g}} = e^{T_{0}} \beta_{0\mfk{g}}^{\mrm{eff}} = e^{T_{0}} h^{\mrm{eff}}$.
  To conclude, it suffices to note that $h^{\mrm{eff}}$ is measurable with respect to $B\vert_{V_{1}}$.
\end{proof}

\begin{lemma}[Law of Effective Magnetic Field as $h\searrow 0$.]\label{lem:law-eff-mag-field}
  Consider the setting of Lemma~\ref{lem:reduction-two-vertices}.
  Further consider a $t$-field $\{T^{\mrm{(0)}}_{x}\}$ on $\mbb{T}_{d,n}$, pinned at the origin, at the same inverse temperature $\beta$.
  Then we have that
  \begin{equation}
    \label{eq:90}
    \frac{h^{\mrm{eff}}}{h} \overset{\mrm{law}}{\longrightarrow} \sum_{x\in\mbb{T}_{d,n}} e^{T^{\mrm{(0)}}_{x}} \qq{as} h\searrow 0.
  \end{equation}
\end{lemma}

\begin{proof}
  By \eqref{eq:87} it suffices to show that
  \begin{equation}
    \label{eq:100}
    \beta\sum_{y\in V_{1}\colon y\sim 0}H_{11}^{-1}(y,x) \overset{\text{law}}{\longrightarrow} e^{T_{x}^{(0)}} \qq{as} h\searrow 0.
  \end{equation}
  We start by decomposing the restriction of $H_{B}$ to $\mbb{T}_{d,n}$, \ie without the ghost vertex $\mfk{g}$, as follows
  \begin{equation}
    \label{eq:99}
    H_{B}\vert_{\mbb{T}_{d,n}} = \mqty(B_{0}& -\beta_{0}^{\top}\\ -\beta_{0}^{\top}& H_{11}),
  \end{equation}
  where we write $\beta_{0} = [\beta \unit_{y\sim 0}]_{y\in V_{1}}$.
  By Schur's complement we have
  \begin{equation}
    \label{eq:101}
    (H_{B}\vert_{\mbb{T}_{d,n}})^{-1} =
    \Bigg(\begin{matrix}
      (B_{0} - \beta_{0}^{\top} H^{-1}_{11}\beta_{0})^{-1}& (B_{0} - \beta_{0}^{\top} H^{-1}_{11}\beta_{0})^{-1} \beta_{0}^{\top} H^{-1}_{11}\\ \cdots & \cdots
    \end{matrix}\Bigg).
  \end{equation}
  As a consequence, for any $x\in V_{1}$
  \begin{equation}
    \label{eq:102}
    \frac{(H_{B}\vert_{\mbb{T}_{d,n}})^{-1}(0,x)}{(H_{B}\vert_{\mbb{T}_{d,n}})^{-1}(0,0)}
    = (\beta_{0}^{\top}H^{-1}_{11})(0,x) = \beta\sum_{y\in V_{1}\colon y\sim 0} H^{-1}_{11}(y,x).
  \end{equation}
  We now note that as $h\searrow 0$ the law of $B\vert_{\mbb{T}_{d,n}}$ converges to that of a STZ-field on $\mbb{T}_{d,n}$, as can be seen from \eqref{eq:209}.
  Consequently, by Proposition~\ref{prop:stz-t-field-coupling}, the law of the left hand side in \eqref{eq:102} converges to that of $e^{T_{x}^{(0)}}$, which proves the claim.
\end{proof}

\begin{proof}[Proof of Theorem~\ref{thm:int-phase-h22}]
  Combining Lemma~\ref{lem:h22-observable-two-points} and \ref{lem:reduction-two-vertices} we have
  \begin{equation}
    \label{eq:118}
    \lim_{h\searrow 0} h^{-\eta}\eva{z_{0} \abs{x_{0}}^{-\eta}}_{\beta,h;\mbb{T}_{d,n}}
    =
    \lim_{h\searrow 0} \EE_{\beta,h}[\Big(\frac{h^{\mrm{eff}}}{h}\Big)^{\eta} g_{\eta}(h^{\mrm{eff}})]
  \end{equation}
  We note that by \cite[Proposition~6.1.2]{poudevigne-auboiron_monotonicity_2022} we have $\EE[h^{\mrm{eff}}] \leq h \abs{\mbb{T}_{d,n}}$.
  Hence, for any fixed $C>0$ we have $h^{\mrm{eff}} \leq C$ with probability $1-o(1)$ as $h\searrow 0$.
  Lemma~\ref{lem:law-eff-mag-field} therefore implies
  \begin{equation}
    \label{eq:120}
    \lim_{h\searrow 0} h^{-\eta}\eva{z_{0} \abs{x_{0}}^{-\eta}}_{\beta,h;\mbb{T}_{d,n}}
    = c_{\eta} \EE_{\beta}[\Big(\sum_{x\in\mbb{T}_{d,n}} e^{T^{\mrm{(0)}}_{x}}\Big)^{\eta}],
  \end{equation}
  with $c_{\eta} > 0$ given in \eqref{eq:84}.
  Consequently, application of Lemma~\ref{lem:frac-local-time-t-field} and Theorem~\ref{thm:vrjp-multifractality} concludes the proof.
\end{proof}

\newpage
\appendix
\section{Tail Bounds for the $t$-field increments.}
\label{sec:tail-bound-t-field}
In this section, we apply the Cramér-Chernoff method \cite{boucheron_concentration_2013} to prove a doubly-exponential lower tail-bound for sums of independent (negative) $t$-field increments.
Consider the Fenchel-Legendre dual of the $t$-field's log-moment-generating function:
\begin{equation}
  \label{eq:115}
  \Psi^{\ast}_{\beta}(\tau)
  = \sup_{\lambda \geq 0}(\lambda \tau - \log \EE_{\beta} [e^{-\lambda T}]).
\end{equation}

\begin{lemma}[Lower Tail bound for sums of $t$-Field Increments]
  \label{lem:t-field-sum-tail-bound}
  Let $\{T_{i}\}_{i=1,\ldots,n}$ denote independent random variables distributed according to the $t$-field increment measure $\mcl{Q}^{\mrm{inc}}_{\beta}$ (see Definition~\ref{def:t-field-inc}).
  For any $\tau > 0$ we have
  \begin{equation}
    \label{eq:128}\textstyle
    \PP_{\beta}[\sum_{i=1}^{n} T_{i} \leq - n\tau]
    \leq \exp\Big[-n \Psi^{\ast}_{\beta}(\tau)\Big],
  \end{equation}
  Moreover, $\Psi_{\beta}^{\ast}$ is bounded from below as
  \begin{equation}
    \label{eq:4}
    \begin{aligned}
      \Psi^{\ast}_{\beta}(\tau)
      &> \sup_{0 < \rho < 1} [\rho\, \frac{\beta  e^{\tau}}{2} - \beta(1-\sqrt{1-\rho}) + \tfrac12 \log(1-\rho)]\\
      &\geq (\tfrac38\beta e^{\tau} - \log[2e^{\beta/2}]).
    \end{aligned}
  \end{equation}
\end{lemma}

To prove this, we note that for a $t$-field increment $T$, the random variable $e^{\pm T}$ follows a (reciprocal) inverse Gaussian distribution.
For completeness, recall that a random variable $X > 0$ is said to follow an \emph{inverse Gaussian} distribution,  $X\sim \mrm{IG}(\mu,\beta)$, if it has density
\begin{equation}
  \label{eq:111}
  \frac{e^{\beta/\mu}}{\sqrt{2\pi/\beta}}\, e^{-\frac{\beta}{2}(\frac{x}{\mu^{2}} + \frac{1}{x})} \frac{\dd{x}}{x^{3/2}}
\end{equation}
over the positive real numbers.
Similarly, $Y > 0$ follows \emph{reciprocal inverse Gaussian} distribution, $Y \sim \mrm{RIG}(\mu,\beta)$, if it has density
\begin{equation}
  \label{eq:114}
  \frac{e^{\beta/\mu}}{\sqrt{2\pi/\beta}} e^{-\frac{\beta}{2}(y + \frac{1}{\mu^{2}y})} \frac{\dd{y}}{\sqrt{y}}
\end{equation}
over the positive real numbers.
With this convention, we have $e^{T} \sim \mrm{IG}(1,\beta)$ and $e^{-T} \sim \mrm{RIG}(1,\beta)$.
Also recall the moment-generating functions (MGF):
\begin{equation}
  \label{eq:5}
  \begin{aligned}
    \EE[e^{\lambda X}] &= e^{\frac{\beta}{\mu}\left(1- \sqrt{1-2\mu^{2}\lambda/\beta}\right)} \qq{for} \lambda < \beta/(2\mu^{2}),\\
    \EE[e^{\lambda Y}] &= \frac{e^{\frac{\beta}{\mu}\left(1- \sqrt{1-2\lambda/\beta}\right)}}{\sqrt{1-2\lambda/\beta}} \qq{for} \lambda < \beta/2.
  \end{aligned}
\end{equation}
With this, we have everything we need:
\begin{proof}[Proof of Lemma~\ref{lem:t-field-sum-tail-bound}]
  By Markov's inequality one easily derives the Chernoff bound
  \begin{equation}
    \label{eq:121}
    \PP[T \leq -\tau] \leq e^{-\Psi_{\beta}^{\ast}(\tau)}.
  \end{equation}
  Similarly, for independent $t$-field increments $\{T_{i}\}$ one obtains
  \begin{equation}
    \label{eq:122}\textstyle
    \PP[\sum_{i=1}^{n} T_{i} \leq - n\tau] \leq e^{-n\Psi_{\beta}^{\ast}(\tau)}.
  \end{equation}
  In the following, we establish lower bounds on $\Psi^{\ast}_{\beta}$.
  We start by bounding $\EE_{\beta}[e^{-\lambda t}]$, using the elementary inequality $x^{\lambda} \leq (\lambda/e)^{\lambda} e^{x}$ for $x > 0$:
  \begin{equation}
    \label{eq:123}
    \begin{aligned}
      \EE[e^{-\lambda T}]
      &= \rho^{-\lambda} \EE[(\rho e^{-T})^{\lambda}]\\
      &\leq \left( \tfrac{\lambda}{\rho e} \right)^{\lambda} \EE[e^{\rho e^{-T}}],
    \end{aligned}
  \end{equation}
  with the right hand side being finite and explicit for $0 < \rho < \beta/2$ by the MGF for the reciprocal inverse Gaussian distribution \eqref{eq:5}.
  Consequently, for any $\lambda, \tau > 0$ and $0 < \rho < \beta/2$ we have
  \begin{equation}
    \label{eq:125}\textstyle
    \lambda \tau - \log \EE[e^{-\lambda T}] \geq \lambda (\tau - \log(\lambda/\rho) + 1) - \log \EE[e^{\rho e^{-T}}].
  \end{equation}
  In $\lambda$, the right hand side is maximised for $\lambda = \rho e^{\tau}$, which yields
  \begin{equation}
    \label{eq:126}\textstyle
    \Psi^{\ast}_{\beta}(\tau) \geq \sup_{\rho > 0} (\rho e^{\tau} - \log\EE[e^{\rho e^{-T}}]).
  \end{equation}
  After inserting \eqref{eq:5} and rescaling $\rho \mapsto \tfrac{\beta}{2}\rho$, first bound in \eqref{eq:4} follows.
  For the second bound, one may simply choose $\rho = 3/4$.
\end{proof}

\section{Uniform Gantert-Hu-Shi Asymptotics for $\tau^{\beta}_{x}$: Proof of Theorem~\ref{thm:uniform-ghs}}
\label{sec:uniform-ghs}

We will stay close to the original proof by Gantert, Hu ans Shi \cite{gantert_asymptotics_2011}, but get rid of some of the technical details as we only require a lower bound not a precise limit.
Also note that Gantert \emph{et al.} prove their result for general branching random walks, whereas we only show the result for a deterministic offspring distribution.
A crucial technical ingredient to Gantert \emph{et al.}'s proof is their extension of Mogulskii's Lemma (Lemma~\ref{lem:mogulski}), which we also make use of.

\begin{definition}
Let $\rho_{\beta}(\delta,n)$ be the probability that there exists $|x|=n$ such that for all $i\in[n]$, $\tau_{x_i}\leq \delta i$.
\end{definition}

\begin{definition}
  Let $\tau = \tau^{\beta}$ be a random variable distributed as the increment of $\{\tau^{\beta}_{x}\}_{x\in\mbb{T}_{d}}$.
Let $M_{\beta}$ be such that $\PP_{\beta}\left(\tau \geq M_{\beta} \right) = 2/3 $ and let $p_d>0$ be the probability that a Galton-Watson tree where the reproduction law is given by a binomial $\mrm{Bin}(d,2/3)$ survives.
We now define for any $\delta>0$ small enough and for any $n\in\NN$ the set $G_{n,\delta}$ as follows: 
\[
\begin{aligned}
G_{n,\delta}:=\{|x|=n \text{ such that }& \tau_{x_i}\leq \frac{1}{2} \delta i, \forall i\in [(1-\delta/(2M_{\beta})n] \text{ and }\\
& \text{for all } \Big(1-\frac{\delta}{2M_{\beta}}\Big)n+1\leq k \leq n,\  \tau_{x_k}-\tau_{x_{k-1}} \leq M_{\beta} \}.
\end{aligned}
\]
\end{definition}

The idea is that if $G_{n,\delta}$ is not empty, it means that there is a vertex $x$ such that $|x|=n$ and 
$\forall i\in [n], \tau_{x_i}\leq \delta i$.
Then started at all the vertices of $G_{n,\delta}$ we can see if the corresponding sets $G_{n,\delta}$ are not empty.
This allows us to create a Galton-Watson tree.
The exact definition of $G_{n,\delta}$ is chosen to ensure that if it is not empty it contains many vertices.
In turn this means that if the Galton-Watson tree we construct is not empty then it is infinite with high probability.
To compute everything precisely we will use \ref{lem:mogulski} but first we need a preliminary result.
The following results allows us to show that if $G_{n,\delta}$ is not empty then with high probability it has many vertices.

\begin{lemma}[Lemma 1 of \cite{mcdiarmid_children_1995}]\label{lem:McDiarmid}
Let $(Z_n)_{n\in\NN}$ be a supercritical Galton Watson tree. There exists $\eta>1$ such that:
\[
\PP[Z_n<\eta^n]=\PP[Z \text{ is finite}]+o(\eta^{-n}).
\]
\end{lemma}
\begin{corollary}\label{lem:McDiarmid+}
Let $(Z_n)_{n\in\NN}$ be a supercritical Galton Watson tree. There exists $\eta>1$ such that:
\[
\PP[1\leq Z_n \leq \eta^n]=o(\eta^{-n}).
\]
\end{corollary}
\begin{proof}
The Galton-Watson tree conditioned on dying is a sub-critical Galton Watson tree and thus the probability that it survives up to time $n$ decreases exponentially in $n$. This coupled with \ref{lem:McDiarmid} gives the desired result.
\end{proof}

Now the goal is to give a lower bound on the probability that $G_{n,\delta}$ is not empty. 
First we express this in terms of $\rho_{\beta}$.

\begin{lemma}[{\cite[Lemma~4.3]{gantert_asymptotics_2011}}]
Let $\delta>0$. We have:
\[
\PP_{\beta}[G_{n,\delta}\not= \emptyset] \geq p_d \rho_{\beta}(\delta/2,n).
\]
\end{lemma}
\begin{proof}
Let $L:=\Big\lfloor \Big(1-\frac{\delta}{2M_{\beta}}\Big)n\Big\rfloor$.
\[
\begin{aligned}
\PP_{\beta}[G_{n,\delta}\not= \emptyset]
= &\PP_{\beta}\left[\exists |x|=L \text{ such that } \tau_{x_i}\leq \frac{1}{2} \delta i, \forall i\in [L]\right] \ldots \\
& \ldots \times \PP_{\beta} \left[ \exists |x|=n-L \text{ such that } \max_{1\leq k \leq n-L} \tau_{x_i}-\tau_{x_{i-1}} \leq M_{\beta} \right]\\
\geq &\rho_{\beta}(\delta/2, n) p_d.
\end{aligned}
\]
\end{proof}

Once we have this lower bound, we need to show that with high probability if $|G_{n,\delta}|$ is not empty then it has many children with high probability. 

\begin{lemma}
Let $L:=\Big\lfloor \Big(1-\frac{\delta}{2M_{\beta}}\Big)n\Big\rfloor$. There exists $\eta>1$ such that for $n-L$ large enough (this only depends on $d$):
\[
\PP_{\beta}\big[ 1\leq |G_{n,\delta}| \leq \eta^{n-L} \big| |G_{n,\delta}|\geq 1 \big] = o\Big(\frac{1}{\eta^{n-L}}\Big).
\] 
\end{lemma}
\begin{proof}
If $|G_{n,\delta}|\geq 1$, it means that there exists $x$ such that $|x|=n$ and
\[
\tau_{x_i}\leq \alpha \delta i, \forall i\in [L] \text{ and } \max_{L+1\leq k \leq n} \tau_{x_i}-\tau_{x_{i-1}} \leq M. 
\]
Now, if we restrict $G_{n,\delta}$ to the descendant of $x_L$, we get a Galton-Watson tree conditioned to survive up to time $n-L$ and where the reproduction law is a binomial $B\big(n,\PP_{\beta}(\tau \leq M_{\beta})\big)$ which does not depend on $\beta$. Then, by $\ref{lem:McDiarmid+}$, we have the desired result.
\end{proof}

What is left is to give a lower bound $\rho$. The goal of the next lemmata is to give a lower bound of $\rho$ by terms for which we can apply Lemma~\ref{lem:mogulski}.

\begin{lemma}[Lemma 4.5 of \cite{gantert_asymptotics_2011}]\label{lem:weird_lower_bound}
For any $n\geq 1$ and any $i\in[n]$, let $I_{i,n}\subset \RR$ be a Borel set. We have:
\[
\PP_{\beta}\left[\exists |x|=n \text{ such that } \forall i\in [n], \tau_{x_i} \in I_{i,n}\right]
\geq \frac{\EE_{\beta} \left[e^{S_n} 1_{\forall i\in [n], S_i \in I_{i,n}}\right]}{1+(d-1)\sum_{j=1}^n h_{j,n}},
\]
where $h_{j,n}$ is defined by:
\[
h_{j,n}:= \sup\limits_{u\in I_{j,n}} \EE_{\beta}\left[e^{S_{n-j}}1_{\forall l \in [n-j], S_l+u\in I_{l+j,n}}\right].
\]
\end{lemma}

\begin{lemma}\label{lem:lower_bound_rho}
For any $\beta>\beta_c$ we have:
\[
\rho_{\beta}(n^{-2/3},n)  \geq \frac{\PP_{\beta}\left[\frac{ i}{n}-1 \leq \frac{S_i}{n^{1/3}} \leq \frac{ i}{n} \ \forall i\in [n]\right]}{1+(d-1) n e^{2 n^{1/3}}} .
\]
\end{lemma}
\begin{proof}
Let  $I_{i,n}:= \big[\frac{ i}{n^{2/3}}-  n^{1/3}, \frac{ i}{n^{2/3}}\big]$. We have:
\begin{equation}\label{eq:lower_bound_rho}
\begin{aligned}
\rho_{\beta}( n^{-2/3},n) 
\geq& \PP_{\beta} [ \exists |x|=n \text{ such that } \tau_{x_i}\in I_{i,n} \forall i\in[n]] \\
\geq& \frac{\EE_{\beta} \left[e^{S_n} 1_{\forall i\in [n], S_i \in I_{i,n}}\right]}{1+(d-1)\sum_{j=1}^n h_{j,n}} \text{ by lemma } \ref{lem:weird_lower_bound},
\end{aligned}
\end{equation}
where $h_{j,n}$ is as in lemma \ref{lem:weird_lower_bound}.
The numerator of \ref{eq:lower_bound_rho} can be bounded as follows:
\begin{equation}
\EE_{\beta} \left[e^{S_n} 1_{\forall i\in [n], S_i \in I_{i,n}}\right]
\geq  e^{(1-1)n^{1/3}}\PP\big[\forall i \in [n], S_i \in I_{i,n} \big].
\end{equation}
As for the denominator, we have:
\[
\begin{aligned}
h_{j,n}=&\sup_{u\in I_{j,n}}\EE\big[e^{S_{n-j}}1_{\forall i \in [n-j], S_i\in[(i+j)/n^{2/3}-\lambda n^{1/3}-u,(i+j)/n^{2/3}-u]}\big]\\
\leq & e^{ (i+j)/n^{2/3} -  j /n^{2/3} +  n^{1/3}}\\
\leq & e^{2 n^{1/3}}.
\end{aligned}
\]
From this we get the desired result.\\
\end{proof}

Now we have everything we need to prove the result we want.

\begin{proof}[Proof of Theorem~\ref{thm:uniform-ghs}]
Given the tree $\mathbb{T}^d$ and the $\tau$-field on it we create the new tree $\tilde{\mathbb{T}}$ as follows: we look at all the vertices $x$ at distance $n$ of the origin, and we only keep those that are in $G_{n,\delta}$. 
Then we look at the trees started at those vertices and we apply the same procedure recursively. 
The tree we obtain is thus a Galton-Watson tree with reproduction law given by the law of $|G_{n,\delta}|$. 
Furthermore, by definition of $G_{n,\delta}$, if $\tilde{\mathbb{T}}$ is infinite then there exists an infinite path $\gamma$ in $\mathbb{T}^d$ such that for all $i\in\NN,\ \tau_{\gamma_i} \leq \delta i$. 
Now we just need to give a lower bound on the probability that $\tilde{\mathbb{T}}$ is infinite. 
By the lemmata \ref{lem:weird_lower_bound} and \ref{lem:lower_bound_rho}, we have by taking $\delta_n:=2n^{-2/3}$:
\[
\PP_{\beta}[G_{n,\delta_n}\not=\emptyset]\geq p_d \frac{\PP_{\beta}\left[\frac{ i}{n}-1 \leq \frac{S_i}{n^{1/3}} \leq \frac{ i}{n} \ \forall i\in [n]\right]}{1+(d-1) n e^{2 n^{1/3}}}.
\] 
Now we want to apply \ref{lem:mogulski} but unfortunately we are not exactly in the conditions of the theorem, we would need $\frac{S_i}{n^{1/3}} \leq \frac{ i}{n} +\text{ something}$.
To get that, we say that there exists some constant $c$ such that uniformly on some interval $[\beta_c,\beta_c+a]$ we have:
\[
\PP_{\beta}[S_1 \in (-2,-1)]\geq c.
\]
Therefore for any $\delta>0$:
\[
\PP_{\beta}[\forall i\in [\delta n^{1/3}] (S_{i}-S_{i-1}) \in (-2,-1)]\geq e^{\log(c) \delta n^{1/3}}.
\]
Now, we get for any $\epsilon>0$ small enough:
\[
\begin{aligned}
&\PP_{\beta}\bigg[\frac{ i}{n}-1 \leq \frac{S_i}{n^{1/3}} \leq \frac{ i}{n} \ \forall i\in [n]\bigg]\\
\geq & \PP_{\beta}\bigg[\forall i\in [\epsilon n^{1/3}],\  (S_{i}-S_{i-1}) \in (-2,-1)\bigg] \PP_{\beta}\bigg[\frac{ i}{n}-1+2\epsilon \leq \frac{S_i}{n^{1/3}} \leq \frac{ i}{n} + \epsilon \ \forall i\in [n-\epsilon n^{1/3}]\bigg] \\
\geq & e^{\log(c) \epsilon n^{1/3}} \PP_{\beta}\bigg[\frac{ i}{n}-1+2\epsilon \leq \frac{S_i}{n^{1/3}} \leq \frac{ i}{n} + \epsilon \ \forall i\in [n]\bigg].
\end{aligned}
\]
Finally we satisfy the condition of our lemma \ref{lem:mogulski}.
We have by lemma \ref{lem:mogulski} that for any interval of the form $[\beta_c,\beta_c+a]$ there exists some explicit constant $C_a$ such that : 
\[
  \limsup_{n\rightarrow\infty}\sup_{\beta\in[\beta_c,\beta_c+a]} n^{-1/3}\log\PP_{\beta}\bigg[\frac{ i}{n}-1 + 2 \epsilon \leq \frac{S_i}{n^{1/3}} \leq \frac{ i}{n} + \epsilon \ \forall i\in [n]\bigg] \leq C_\delta.
\]
Define $f_{\beta}$ by $f_{\beta}:=\EE_{\beta}[s^{|G_{n,\delta}(\alpha)|}]$ and let $q_{\beta,n}$ be the extinction probability. We have $q_{\beta,n}=f_{\beta}(\beta,n)$. For any $r<q_{\beta,n}$ we have:
\[
q_{\beta,n}
=f_{\beta}(0)+\int_0^{q_{\beta,n}}f_{\beta}^{'}(s) \mathrm{d} s 
= f_{\beta}(0)+\int_0^{q_{\beta,n}-r}f_{\beta}^{'}(s)\mathrm{d} s+ \int_{q_{\beta,n}-r}^{q_{\beta,n}}f_{\beta}^{'}(s) \mathrm{d} s. 
\]
Now, using that $f_{\beta}$ is convex and therefore $f_{\beta}^{'}$ is non-decreasing, we get:
\[
q_{\beta,n}
\leq f_{\beta}(0)+(q_{\beta,n}-r)f_{\beta}^{'}(q_{\beta,n}-r)+ r f_{\beta}^{'}(q_{\beta,n})
\leq f_{\beta}(0)+(1-r)f_{\beta}^{'}(1-r)+ r. 
\]
Now, $f_{\beta}(0)=\PP_{\beta}[G_{n,\delta_n}=\emptyset]$ and $f_{\beta}^{'}(1-r)=\EE_{\beta}[|G_{n,\delta_n}|(1-r)^{|G_{n,\delta_n}|-1}]$ which is bounded from above by $\frac{1}{1-r}\EE_{\beta}(|G_{n,\delta_n}|e^{-r|G_{n,\delta_n}|})$. Now if we take $r<1/2$ we get:
\[
1-q_{\beta,n} \geq \PP_{\beta}[G_{n,\delta_n}\not=\emptyset] -2\EE_{\beta}[|G_{n,\delta_n}|e^{-r|G_{n,\delta_n}|}]-r.
\]
From this we get:
\[
\begin{aligned}
1-q_{\beta,n} 
\geq & \PP_{\beta}[G_{n,\delta_n}\not=\emptyset] - \frac{2}{r^2}\PP_{\beta}\left[1\leq |G_{n,\delta_n}| \leq r^2 \right] -\frac{2e^{-1/r}}{r^2}-r\\
\geq & \PP_{\beta}[G_{n,\delta_n}\not=\emptyset] - \frac{2}{r^2}\PP_{\beta}\left[1\leq |G_{n,\delta}| \leq r^2 \right] -2r \text{ for } r \text{ small enough}.
\end{aligned}
\]
By taking $r=\eta^{-n}$ we get that for $n$ large enough, for some constant $C>0$, for any $\beta\in[\beta_c,\beta_c+a]$:
\[
1-q_{\beta,n}\geq e^{-Cn^{1/3}}.
\]
Then by noticing that $n=(2/\delta_n)^{3/2}$ we get the desired result.
\end{proof}

\section{Effective Conductance and Effective Weight}
\label{sec:eff-cond-weight}

Before starting with the proof of Lemma~\ref{lem:eff-conductance-weight}, we would like to remind the reader of the definition of the effective weight \eqref{eq:188} as well as the discussion following it.

\begin{proof}[Proof of Lemma~\ref{lem:eff-conductance-weight}]
  Let $D_{T}$ denote the graph Laplacian on $G$ with weights given by the $t$-field environment $\{\beta_{ij}e^{T_{i} + T_{j}}\}_{ij \in E}$.
  The effective resistance (\ie inverse effective conductance) can be expressed as
  \begin{equation}
    \label{eq:86}
    1/C^{\mrm{eff}}_{i_{0}j_{0}} = (-D_{T}\vert_{V\setminus\{i_{0}\}})^{-1}(j_{0},j_{0}),
  \end{equation}
  where $D_{T}\vert_{V\setminus\{i_{0}\}}$ denotes $D_{T}$ with deletion of the row and column corresponding to $i_{0}$.
  Recall that on $V\setminus\{i_{0}\}$ we have $B_{x} = \sum_{y\sim x}\beta_{xy}e^{T_{y} - T_{x}}$.
  Defining the diagonal matrices $L_{T} = \mrm{diag}(\{e^{T_{x}}\}_{x\in T\setminus\{i_{0}\}})$, one may check that
  \begin{equation}
    \label{eq:93}
    -D_{T}\vert_{V\setminus\{i_{0}\}} = L_{T} \,H_{B}\vert_{V\setminus\{i_{0}\}}\, L_{T}.
  \end{equation}
  Inserting this into \eqref{eq:86} yields
  \begin{equation}
    \label{eq:94}
    e^{-T_{j_{0}}} C^{\mrm{eff}}_{i_{0}j_{0}} = \frac{e^{T_{j_{0}}}}{(H_{B}\vert_{V\setminus\{i_{0}\}})^{-1}(j_{0},j_{0})} = \frac{H_{B}^{-1}(i_{0},j_{0})}{H_{B}^{-1}(i_{0},i_{0}) \, (H_{B}\vert_{V\setminus\{i_{0}\}})^{-1}(j_{0},j_{0})}
  \end{equation}
  Using \eqref{eq:95} and the familiar expression for the inverse of a 2x2-matrix, we have
  \begin{equation}
    \label{eq:96}
    \frac{H_{B}^{-1}(i_{0},j_{0})}{H_{B}^{-1}(i_{0},i_{0})}
    = \frac{\beta_{i_{0}j_{0}} + [H_{01}H_{11}^{-1}H_{10}](i_{0},j_{0})}{B_{j_{0}} - [H_{01}H_{11}^{-1}H_{10}](j_{0},j_{0})}.
  \end{equation}
  Note that the numerator equals $\beta_{i_{0}j_{0}}^{\mrm{eff}}$.
  On the other hand, using a Schur decomposition for $H_{B}\vert_{V\setminus\{i_{0}\}}$, decomposing $V\setminus\{i_{0}\}$ into $\{j_{0}\}$ and $V_{1}$, one may compute
  \begin{equation}
    \label{eq:97}
    (H_{B}\vert_{V\setminus\{i_{0}\}})^{-1}(j_{0},j_{0})
    = 1/(B_{j_{0}} - [H_{01}H_{11}^{-1}H_{10}](j_{0},j_{0})).
  \end{equation}
  Inserting \eqref{eq:96} and \eqref{eq:97} into \eqref{eq:94} we obtain
  \begin{equation}
    \label{eq:98}
    e^{-T_{j_{0}}} C^{\mrm{eff}}_{i_{0}j_{0}} = \beta_{i_{0}j_{0}} + [H_{01}H_{11}^{-1}H_{10}](i_{0},j_{0}) = \beta^{\mrm{eff}}_{i_{0}j_{0}},
  \end{equation}
  which proves the claim.
\end{proof}

\begin{lemma}[Reflection Property of the $t$-Field]\label{lem:reflection-property}
  Consider a finite graph $G = (V,E)$ with positive edge weights $\{\beta_{ij}\}_{ij \in E}$.
  Let $\{T_{x}\}_{x\in V}$ denote a $t$-field on $G$ with weights $\{\beta_{ij}\}$, pinned at some vertex $i_{0}$.
  For any $q\in \RR$ and $x \in V$ we have
  \begin{equation}
    \label{eq:205}
    \EE[e^{q T_{x}}] = \EE[e^{(\tfrac12 - q)T_{x}}].
  \end{equation}
\end{lemma}

\begin{proof}
  On a graph with two vertices, the claim follows from the density of the $t$-field increment measure (Definition~\ref{def:t-field-inc}).
  On a larger graph, we consider the natural coupling of the STZ-field $\{B_{x}\}_{x\in V}$ and the $t$-field.
  By \cite[Section~6.1]{poudevigne-auboiron_monotonicity_2022} we know that conditionally on $B_{y}$ for $y \in V\setminus \{i_{0}, x\}$, the $t$-field on $\{i_{0}, x\}$ follows the law of a $t$-field on this reduced graph (still pinned at $i_{0}$) with edge-weights given by an effective weight $\beta_{i_{0}x}$ (the latter being measurable with respect to the STZ-field outside $\{i_{0}, x\}$).
  Consequently, the claim follows from the statement on two vertices.
\end{proof}

\newpage
\setlength\bibitemsep{0.15\baselineskip}
\renewcommand*{\bibfont}{\small}
\printbibliography
\end{document}